\documentclass[11pt,reqno]{amsart}

\usepackage{amsmath, amsthm, amssymb}
\usepackage{amsfonts}
\usepackage[ansinew]{inputenc}
\usepackage[dvips]{epsfig}
\usepackage{graphicx}
\usepackage[english]{babel}
\usepackage{hyperref}

\usepackage{cite}
\usepackage{graphicx}
\usepackage{amscd}
\usepackage{bm}
\usepackage{enumerate}

\usepackage{verbatim}
\usepackage{hyperref}
\usepackage{amstext}
\usepackage{latexsym}
\theoremstyle{plain}
\newtheorem{theorem}{Theorem}[section]

\newtheorem{corollary}[theorem]{Corollary}
\newtheorem{proposition}[theorem]{Proposition}
\newtheorem{lemma}[theorem]{Lemma}
\newtheorem{remark}[theorem]{Remark}

\numberwithin{theorem}{section}
\numberwithin{equation}{section}

\newcommand{\average}{{\mathchoice {\kern1ex\vcenter{\hrule height.4pt
width 6pt depth0pt} \kern-9.7pt} {\kern1ex\vcenter{\hrule
height.4pt width 4.3pt depth0pt} \kern-7pt} {} {} }}

\def\R{\mathbb{R}}

\renewcommand{\a }{\alpha }
\renewcommand{\b }{\beta }
\renewcommand{\d}{\delta }

\newcommand{\D }{\Delta }

\newcommand{\e }{\varepsilon }
\newcommand{\g }{\gamma}

\newcommand{\G }{\Gamma}
\renewcommand{\l }{\lambda }

\newcommand{\n }{\nabla }
\newcommand{\vp }{\varphi }

\newcommand{\s }{\sigma }
\newcommand{\Sig }{\Sigma}
\renewcommand{\t }{\tau }
\newcommand{\z }{\zeta}
\renewcommand{\th }{\theta }

\renewcommand{\o }{\omega }
\renewcommand{\O }{\Omega }

\newcommand{\ov}{\overline}

\newcommand{\be}{\begin{equation}}
\newcommand{\ee}{\end{equation}}

\newcommand{\de}{\partial}

\newcommand{\ti}{\widetilde}

\renewcommand{\k}{\kappa}

 \usepackage{mathrsfs}

\newcommand{\calH }{\mathcal{H}}

\newcommand{\calC }{\mathcal{C}}

\newcommand{\calD }{\mathcal{D}}
\newcommand{\calQ }{\mathcal{Q}}
\newcommand{\calB }{\mathcal{B}}

\newcommand{\N}{\mathbb{N}}

\newcommand{\cA}{{\mathcal A}}

\newcommand{\cK}{{\mathcal K}}

\newcommand{\calR}{{\mathcal R}}

\renewcommand{\epsilon}{\varepsilon}

\newcommand{\Ds}{ (-\D)^s}

\newcommand{\x}{ \xi}

\begin{document}
 

\author[Mouhamed M. Fall]
{Mouhamed Moustapha Fall}
\address{M.M.F.: African Institute for Mathematical Sciences in Senegal, 
KM 2, Route de Joal, B.P. 14 18. Mbour, S\'en\'egal}
\email{mouhamed.m.fall@aims-senegal.org, mouhamed.m.fall@gmail.com}

\thanks{The author is in debt to Sven  Jarohs and Tobias Weth, for taking his attention to the problems that are studied in this paper. He would like to thank them for useful discussions. The author's work is supported by the Alexander von Humboldt foundation and partially by  DAAD and BMBF (Germany) within the project 57385104.
}


 \begin{abstract}
   \noindent
 We study boundary regularity for solutions   to a class of equations involving the so called regional fractional Lapacians $(-\Delta)^s_\Omega $,  with $\Omega\subset {\mathbb{R}^N}$. Recall that  the regional fractional Laplacians   are  generated by   L\'evy-type  processes which are not allowed to jump outside $\Omega$. We consider  weak   solutions  to the equation   $(-\Delta)^s_\Omega w(x)=p.v.\int_{\Omega}\frac{w(x)-w(y)}{|x-y|^{N+2s}}\, dy=f(x)$, for $s\in (0,1)$ and $\Omega\subset\mathbb{R}^N$, subject to zero Neumann or Dirichlet boundary conditions.  The boundary conditions are defined by considering $w$ as well as the test functions in the fractional Sobolev spaces $H^s(\Omega)$ or $H^s_0(\Omega)$ respectively. 
 While the interior regularity is well understood for these problems, little is known in the boundary regularity, mainly for the Neumann problem. Under optimal regularity assumptions on $\Omega$ and provided $f\in L^p(\Omega)$, we show that  $w\in C^{2s-N/p}(\overline \Omega)$ in the case of  zero Neumann boundary conditions. As a consequence for $2s-N/p>1$,  $w\in C^{1,2s-\frac{N}{p}-1}(\overline{\Omega})$.  As what concerned the Dirichlet problem, we obtain ${w}/{\delta^{2s-1}}\in C^{1-N/p}(\overline\Omega)$, provided $p>N$ and $s\in (1/2,1)$, where $\delta(x)=\textrm{dist}(x,\partial\Omega)$. 
 To prove these results, we first  classify all solutions having a certain growth at infinity when $\O$ is a half-space and the right hand side is zero. We then carry over a fine blow up and some compactness arguments to get the results.   
 \end{abstract}

\title{Regional fractional Laplacians: Boundary regularity}
\maketitle
\section{Introduction}
We consider $\O$ an open subset of $\R^N$ with Lipschitz boundary. The present paper is concerned with boundary regularity of solutions to some equations involving nonlocal operators generated by symmetric stable   processes describing  motions of random particles in a region $\O$ which are only allowed to jump inside  $\O$ but are either reflected in $\O$ or killed when they reach the boundary $\de \O$.
In the Brownian case these phenomenon can be described by the Laplace operator subject to Neumann or Dirichlet boundary conditions. A natural  generalization  in the fractional setting was considered by Bogdan,  Burdzy and  Chen  in  \cite{BBC}, where they constructed      censored symmetric stable processes in $\O$.      The generated infinitesimal operators, denoted by $\Ds_{\O}$, are obtained  by limiting the integral in the fractional Laplacian in the region $\O$.
Recall that the fractional Laplacian of  a function  $w\in C^{2}_c(\R^N)$ is given by
\be \label{eq:def-fracd-lap}
\Ds w(x)= c_{N,s} p.v.\int_{\R^N}\frac{w(x)-w(y)}{|x-y|^{N+2s}}\, dy,
\ee
where $s\in (0,1)$ and  $c_{N,s}=\frac{s4^s\Gamma(\frac{N}{2}+s)}{\pi^{N/2}\Gamma(1-s)}$. The regional fractional Laplacian operator  we are interested in  is defined as
\be \label{eq:first-region-Lapla}
\Ds_{\O} w(x)= c_{N,s}p.v.\int_{\O}\frac{w(x)-w(y)}{|x-y|^{N+2s}}\, dy.
\ee
We mention that  the meaning of boundary conditions  has to be made more precise since in some case, depending on the fractional power $s$, it can be meaningless.   However by considering some natural Sobolev spaces these problems can be defined in a natural manner. Indeed, let  $f\in L^1_{loc}(\R^N)$ and $s\in (0,1)$. We consider functions   $u\in H^s ( \O )$ satisfying
\be\label{eq:first-eq1}
\calC_{s,\O}(u,\vp)=\int_{\O}f(x)\vp(x)\, dx \qquad\textrm{  for all $\vp \in C^1_{c} (\ov \O )$} 
\ee
and $v\in H^s_0(\O) $ satisfying
\be\label{eq:first-eq2}
\calC_{s,\O}(v,\vp)=\int_{\O}f(x)\vp(x)\, dx  \qquad\textrm{  for all $\vp \in C^1_c (\O )$,} 
\ee
where for $u,v\in H^s(\O)$, we define
$$
\calC_{s,\O}(u,v):=  \frac{c_{N,s}}{2}\int_{\O\times \O}\frac{(v(x)-v(y))(u(x)-u(y))}{|x-y|^{N+2s}}\, dx dy.
$$
Here and in the following $H^s(\O)$ denotes the usual fractional Sobolev space and $H^s_{0}(\O) $ is the closure of $C^1_c(\O)$ with respect to the  $H^s(\O)$-norm. The above two problems have an interesting intersection. Indeed, if $\O$ is bounded,  Lipschitz and  $s\in (0,1/2]$ then $H^s(\O)=H^s_0(\O)$, see e.g. \cite[Theorem 1.4.2.4]{Grisv}. However  these two spaces are different for $s>1/2$.   As a consequence  \eqref{eq:first-eq1} and \eqref{eq:first-eq2} are equivalent for $s\in (0,1/2]$.\\
We notice that for $s\in (0,1)$, $f\in L^2(\O)$, with $\int_{\O}f(x)\, dx=0$,  and $\O$ bounded then   \eqref{eq:first-eq1} has a unique solution among the  set of functions $u\in H^s(\O)$ satisfying $\int_{\O}u(x)\, dx=0$.  On other hand, provided $s\in (1/2,1)$ and $\O$ is bounded, the  Dirichlet problem  \eqref{eq:first-eq2} has unique solution  in $H^s_{0}(\O) $ for all $f\in L^2(\O)$.\\
 The Dirichlet forms associated to the two variational  problems  \eqref{eq:first-eq1} and \eqref{eq:first-eq2}  generate some fractional order operators  and are given   by \eqref{eq:first-region-Lapla} when acting on smooth functions.  \\
 Problems involving these   fractional order operators    have been intensively studied in the recent years both in the analytic and probabilistic point of view, \cite{BBC,Guan-Ma,Mou,Guan,AV}.   
  In the recent literature the terminology used for these operators are, respectively, \textit{Reflected or Regional fractional Laplacian} and \textit{Censored fractional Laplacian} \cite{BBC,Guan,BFV}.  Recall that their counterparts in the local case ($s=1$) are, respectively, the Poisson problems with Neumann and Dirichlet boundary conditions. In fact,  thanks to the normalization constant $c_{N,s}$, as $s\to1$  equations  \eqref{eq:first-eq1} and \eqref{eq:first-eq2} tend precisely to the classical Poisson problem with Neumann and Dirichlet boundary conditions, respectively, see also Section \ref{s:Append} below. 
For the variational equation \eqref{eq:first-eq2}, we can write, for $s\in (1/2,1)$,  
\be\label{eq:first-eq-Dir}
\begin{cases}
\Ds_{ \O} v=f& \qquad\textrm{ in $\O$,}\\
v=0 & \qquad\textrm{ on $\de \O$,}
\end{cases} 
\ee
and we note that it makes sense to invoke the trace of $v\in H^s_0(\O)$ on $\de\O$, for $s\in (1/2,1)$,  thanks to the trace theorem, see e.g. \cite{Ding}.\\
In the case of \eqref{eq:first-eq1},    we quote in particular the paper \cite{Guan} which contains useful results and integration by parts formula, where a natural  boundary (Neumann) operator appears to be  $N_s(u)(\s):=\lim_{t\searrow0}\frac{ u(\s+t\nu)-u(\s)}{t^{2s-1}}  $ for $\s\in\de\O$, where $\nu$ is the   unit interior normal vector of  $\O$. \\

The aim of this paper  is to study boundary regularity for solutions $u$ and $v$ to \eqref{eq:first-eq1} and    \eqref{eq:first-eq2}, respectively.
  The interior regularity  of both problems are well understood and can be naturally deduced from the one of the fractional Laplacian. Indeed,  both $u1_{\O}$ and $v1_{\O}$ solve a problem of the form
\be \label{eq:From-reg-tofrac-V}
\Ds w+Vw=f\qquad\textrm{ in $\O$,}
\ee
with $V\in C^\infty(\O)$. On the other hand in the case  of  the Dirichlet   problem, boundary Harnack inequalities and Green function estimates are well studied for $s>1/2$, see \cite{BBC,BFV,CK,Huyuan}. In fact, provided $\O$ is of class $C^{1,1}$ and $f\in L^\infty(\O)$,  it is known that any   solution $v$ to \eqref{eq:first-eq-Dir} satisfies
\be \label{eq:reg-Dir-known}
|v(x) |\leq  C\d(x)^{2s-1}  \quad \textrm{ for $x\in \O$,}
\ee
where $\d(x):= \textrm{dist}(x,\de\O)$. However,    the boundary regularity for problem \eqref{eq:first-eq1} is less understood and the only paper, beside the present one,  dealing with  this appeared few days ago. Indeed,  Audrito, Felipe-Navarro and   Ros-Oton showed recently  in \cite{AFR}  that if $\O $ is of class $C^1$ and $f\in L^p(\O)$, $p>\frac{N}{2s}$,  then  $u\in C^\a(\ov \O)$, for some $\a>0$. Moreover for $s\in (1/2,1)$ and $2s-N/p>2s-1$ they obtain $u\in C^{2s-1+\a}(\ov\O)$, for some $\a>0$. We would like to mention that the authors in  \cite{AFR} considered also  an other   interesting fractional order equation with exterio nonlocal Neumann condition.  \\

 %
We prove in this paper optimal boundary regularity estimates  for  solutions to   \eqref{eq:first-eq1} and \eqref{eq:first-eq2}.  Our result for solutions to \eqref{eq:first-eq1}, improves  those obtained in \cite{AFR}, since under the same assumptions on $f$ and $\O$, we have $u\in C^{ 2s-N/p}(\ov \O)$, if $2s-N/p<1$.  Moreover for $2s>1$, we also obtain H\"older estimates,  up to the boundary, of the gradient   of  both solutions to  \eqref{eq:first-eq1} and \eqref{eq:first-eq2}, provided $2s-N/p>1$.  We  then derive further qualitative properties of $u$ and $v$ for $s>1/2$. Indeed, we show   that the normal derivative of $u$ vanishes on $\de\O$ and  a weighted normal derivative of $v$ is proportional to  ${v}/{\d^{2s-1}  }$ on $\de\O$. This latter fact turns out to be useful in order to obtain monotonicity (in the normal direction) of the  solutions  near the boundary  in the spirit of the Hopf boundary point lemma for classical elliptic equations.
  In the same vain, we improve  \eqref{eq:reg-Dir-known} to a  H\"older  regularity of ${v}/{\d^{2s-1}  }$ up to the boundary.    We mention that in the case of the fractional Laplacian with zero exterior Dirichlet data such type of regularity, of the ratio between the solution and the power of the distance function, has been first obtained by Ros-Oton and Serra \cite{RS-2} followed by \cite{RS16b,RS16a,Grubb1,Fall-reg-1}.  \\

We notice that prior to the recent paper  \cite{AFR} and the present paper, even the continuity  of solutions to \eqref{eq:first-eq1}  up to the boundary  was an open question.\\

%
%

\subsection{Boundary regularity in the Neumann case}
Our first main result for regional (or reflected) fractional Laplacian is the following.

\begin{theorem}\label{th:amin1}
Let $s_0\in (0,1)$, $s\in [s_0,1)$ and $\O$ be an open subset of $\R^N$ of class $C^{1}$  and  $f\in L^p(\O)$, for some $p>\frac{N}{2s_0}$.  Let $u\in H^s_{loc}(\ov \O)\cap L^2(\O) $ be a solution to  \eqref{eq:first-eq1}. Then      for every $\e\in (0, 1)$ and for every $\O'\subset\subset\ov\O$ there exists 	 $C>0$ depending only on $N,s_0,\O,\e,\O' $ and $p$	 such that 
$$
\|u\|_{C^{\min(2s-N/p, 1-\e)}(\O')}\leq C \left(  \|u\|_{L^2(\O)}+ \|f\|_{L^p(\O)}   \right). 
$$ 
\end{theorem}

In the case of higher order regularity, we obtain the 
\begin{theorem}\label{th:amin2}
Let $s\in (1/2,1)$ and $\O$ be an open subset of $\R^N$ of class $C^{1,\b}$, with   $\b>0$. Let  $f\in L^p(\O)$  and  $u\in H^s_{loc}(\ov \O)\cap L^2(\O) $ be a solution to  \eqref{eq:first-eq1}.  If  $2s-N/p> 1$,   then $u\in C^{1,\min(2s-\frac{N}{p}-1,\b) }_{loc}(\ov\O)$ and   for every $\O'\subset\subset\ov\O$   there exists 	 $ C>0$ depending only on $N,s,\O,\b,\O' $ and $p$	 such that 
 $$
\| \n u \|_{C^{\min(2s-\frac{N}{p}-1,\b)}(\O')}\leq C \left(  \|u\|_{L^2(\O)}+ \|f\|_{L^p(\O)}   \right). 
$$
Moreover the normal derivative of $u$ vanishes on $\de\O$. More precisely,  for all $\s\in \de \O$,
\be  \label{eq:vanish-norm-der} 
\frac{\de u}{\de\nu}(\s):=  \lim_{t\searrow 0}\frac{u(\s+t\nu(\s))-u(\s)}{t}=0,
\ee
 where $\nu$ is the unit interior normal vectorfield of $\de\O$.
\end{theorem}
We observe that, Theorem \ref{th:amin1} and Theorem \ref{th:amin2} provide similar regularity properties as the interior regularity of weak solutions to     equations involving nonlocal operators of class $C^\b$,  which states that $u\in C^{1,\min(2s-N/p-1,\b)}_{loc}(\O)$, see \cite{Fall-reg-2}. This is, in fact,  the  case  in the local setting ($s=1$) with zero Neumann boundary condition on a $C^{1,\b}$ domain. 
Such high boundary regularity does not in general hold in the Dirichlet problem. The reason for this in the fractional setting can be seen, intuitively on the one hand, from the fact that the variational problem \eqref{eq:first-eq1} allow for a larger class of test functions. On the other hand, one can look at harmonic functions, with locally finite energy, in the half-space $\O=\{(x',x_N)\in \R^{N-1}\times\R\,:\,x_N>0\}$ with respect to the regional fractional Laplacians and the fractional Laplacian. In the Neumann case, they are affine functions (of the form $a\cdot x'+b$) as long as their pointwise growth is strictly smaller than $2s$. However in the Dirichlet case they are  proportional to $x_N^{2s-1}$, while in the case of the fractional Laplacian with  zero exterior data, they are proportional to $(x_N)^s_+$.   \\

\subsection{Boundary regularity in the Dirichlet case}
We now turn to solutions to the  Censored fractional Laplacian. As mentioned earlier, it is the same as the reflected one in the case $s\in (0,1/2]$, provided $\O$ has Lipschitz boundary and is bounded.  
%
\begin{theorem}\label{th:amin1-Dir}
Let $s_0\in (1/2,1)$,  $s\in [s_0,1)$ and  $\O$ be an open subset of $\R^N$ of class $C^{1}$  and  $f\in L^p(\O)$ for some $p>\frac{N}{2s_0}$.  Let $v\in H^s_{0}(\O) $ be a solution to   \eqref{eq:first-eq2}. Then  for every $\e\in (0,2s_0-1)$ and  $\O'\subset\subset\ov\O$ ,      there exists 	 $C>0$ depending only on $N,s_0,\O,\e,\O'$ and $p$	 such that 
$$
\|v\|_{C^{\min(2s-N/p,2s-1-\e)}(\O')}\leq C \left(  \|v\|_{L^2(\O)}+ \|f\|_{L^p(\O)}   \right).
$$
\end{theorem}
In the case of higher order regularity we obtain the
\begin{theorem}\label{th:amin3}
Let $s\in (1/2,1)$ and $\O$ be an open subset of $\R^N$ of class $C^{1,\b}$, $\b>0$ and  $\O'\subset\subset\ov\O$. Let  $f\in L^p(\O)$  and  $v\in H^s_0(\O)$ be a solution to  \eqref{eq:first-eq2}. Then     
\begin{itemize}
\item[(i)] if $2s-N/p> 2s-1$,   there exists 	 $ C>0$ depending only on $N,s,\O,\b,\O' $ and $p$	 such that 
\be\label{eq:iiii1} 
\|v\|_{C^{2s-1}(\O')}+\|  v/\d^{2s-1}\|_{C^{\min(1-N/p,\b)}(\O')}\leq C \left(  \|v\|_{L^2(\O)}+ \|f\|_{L^p(\O)}   \right)
\ee
\item[(ii)] if $2s-N/p> 1$ and $\b>2s-1$,   there exists 	 $C>0$ depending only on $N,s,\O,\b,\O' $ and $p$	 such that 
\be\label{eq:iiii2}  
\|\d^{2-2s}\n v \|_{C^{2s-\frac{N}{p}-1}(\O')}\leq C \left(  \|v\|_{L^2(\O)}+ \|f\|_{L^p(\O)}   \right)
\ee
and 
\be \label{eq:weihted-grad-estim}
\d^{2-2s}\n v =(2s-1) \frac{v}{\d^{2s-1} } \nu \qquad\textrm{ on $\de\O$,}
\ee
where $\nu$ is the interior normal vectorfield of $\de\O$ and  $\d$ coincides with     $\textrm{dist}(x,\de\O)$ in a neighbourhood of $\de\O$, Lipscitz continuous  and positive in the  complement, in $\O$,  of this neighbourhood. 
\end{itemize}
\end{theorem}
We recall, in Theorem \ref{th:amin3},  that for $\s\in \de\O$, we define $\frac{v}{\d^{2s-1} }(\s)=\lim_{\stackrel{x\to \s}{x\in\O}} \frac{v}{\d^{2s-1} }(x)$ and $(\d^{2-2s}\n v)(\s)= \lim_{\stackrel{x\to \s}{x\in\O}} \d^{2-2s}(x)\n v(x) $. Next,  we observe that if $\O$ is of class $C^{1,1}$ then the extension of $\d$ to  the signed distance function to  $\de\O$ is    of class $C^{1,1}$ in a neighbourhood of $\de\O$ and its gradient  coincides with the unit interior normal  $\nu$   on $\de\O$. Therefore \eqref{eq:iiii1}, \eqref{eq:iiii2}  and \eqref{eq:weihted-grad-estim} imply, for $2s-N/p> 1$,  that 
$$
|\n \left( v/\d^{2s-1} \right)|\leq C \d ^{2s-\frac{N}{p}-2}\left(  \|v\|_{L^2(\O)}+ \|f\|_{L^p(\O)}   \right) \qquad\textrm{ in $\O$.}
$$
We notice that  Theorem \ref{th:amin3}-$(i)$, in the case of the fractional Laplacian with zero exterior data, was first obtained by Grubb in \cite{Grubb1} in the case of $C^\infty$ domains and in \cite{Fall-reg-1} in the case of $C^{1,\b}$ domains and even with a larger class of right hand sides containing $L^p(\O)$. Now   \eqref{eq:iiii2} and \eqref{eq:weihted-grad-estim}  was recently obtained for the fractional Laplacian in \cite{FS}.\\

We point out that the  above results in the present paper provide new insights to the Poisson problem involving  these regional nonlocal  operators  and, as such, we believe that they  might   motivate the study of several problems involving these operators in the spirit  of the fractional Laplacian.\\

\begin{remark}
We  note that starting  with $2s_0-N/p>1$,   the constant $C$ in  Theorem \ref{th:amin2} and Theorem \ref{th:amin3}-$(i)$ can be chosen to  remain bounded as $s\to 1$, while for Theorem \ref{th:amin3}-$(ii)$, we need to assume that $\b\geq1$. We refer to Remark \ref{rem:const-s} below for more details.
\end{remark}
To prove these results, we first make a full  classification of all solutions with  growth smaller than $2s$ when $\O$ is a half-space and $f\equiv0$. We then carry out a blow up and some compactness arguments to obtain H\"older continuity up to the boundary.    We note that the classification of the solutions on the half-space was left as a challenging  open problem in \cite{AFR}. To achieve it, we use several ingredients of independent interest, including some reflection principles, Hardy-type inequalities and the Caffarelli-Silvestre extension \cite{CSilv}.  The blow up argument is  partly inspired  by the work of Serra in \cite{Serra} and \cite{Fall-reg-2}, where we study optimal regularity estimates involving a class of  nonlocal operators extending the classical operators in divergence form, see also \cite{RS16a, Fall-reg-1} for other uses.\\
The main difficulties reside on the fact that the operators depend  on the domain. We use local parameterization that flatten the boundary which allows to work on a fixed domain during the blow up process. However, in this case the resulting operator falls in a class of nonlocal operators in divergence form (see e.g. \cite{Fall-reg-2}). To deal with these nontranslation invariant kernels, we use a freezing argument of the kernels in the spirit of \cite{Fall-reg-2}. We also mention that, except in dimension 1 where we have the exact expression of the potential $V$ in  \eqref{eq:From-reg-tofrac-V}, we do not use     \eqref{eq:From-reg-tofrac-V} in our argument because of some scaling issues. Recall that $V(x)\sim \d(x)^{-2s}$ and thus scales as $\Ds$.

The paper is organized as follows. Section \ref{s:Apri} contains the a priori estimates in the half-space. The Liouville-type results are proven in Section \ref{s:Liouville}. The blow up argument leading to    boundary regularity on curved domains for the regional and the censored   fractional Laplacians are given in Section \ref{s:Reg-Neum} and Section \ref{s:Reg-Dir}, respectively. We finally collect the proofs of the main results in Section \ref{s:poof-MR}.\\

\noindent
\textbf{Note added in proof.} In the  recent paper \cite{Fall-Ros-Oton},  Ros-Oton  and the author have  proven the Schauder  estimates  for problems  \eqref{eq:first-eq1} and \eqref{eq:first-eq2}.
 
\section{Notations}\label{s:NP}
 For $E$ a Lipschitz  open subset of $\R^N$ and a   weight   function $a$,  we use the standard notations for weighted
 Lebesgue spaces $L^p(E; a(x))=\{u:E\to \R^M\,:\,\int_{E}|u(x)|^p a(x)dx<\infty \}.$ See e.g. \cite{Grisv},    the fractional Sobolev space $H^{s}(E)$
is  given by the set   of measurable functions $u$  such that  
$$
\|u\|_{H^{s}(E) }^2:=\|u\|_{L^2(E)}^2+[u]^2_{H^s(E)}<\infty,    
$$
where  $[u]^2_{H^s(E)}:=\int_E\int_E\frac{|u(x)-u(y)|^2}{|x-y|^{N+2s}}dxdy<\infty $.
We denote by  $H^{s}_0(E) $  the closure of $ C^\infty_c(E)$ with respect to the norm $\|\cdot\|_{H^{s}(E) } $. We say that $u\in H^s_{loc}(E)$ (resp. $H^s_{loc}(\ov E)$) if for all $\vp\in C^1_c(E)$ (resp. $\vp \in C^1_c(\ov E)$), we have  $\vp u\in H^s(E)$. We define the Hilbert space
$$
\calH^s_0(E)=\{ u\in H^s(\R^N)\,:\,  \quad\textrm{ $u=0$ on $\R^N\setminus E$}\},  
$$
endowed with the norm $\| u\|_{H^s(\R^N)}$.
An interesting characterization of $\calH^s_0(E) $, see  \cite{Grisv},  is that  
\be \label{eq:def-calHs0}
\calH^s_0 (E)=\left\{u\in H^{s}_0(E)\,:\, \int_{E}\frac{u^2(x)}{\d^{2s}_E(x)}dx<\infty, \quad\textrm{ $u=0$ on $\R^N\setminus E$}\right\},
\ee
where $\d_E(x):=\textrm{dist}(x,\de E)$.\\
We denote by $B_r(x_0)$, the ball  of $\R^N$ centred at $x_0$ with radius $r$ and $B_r=B_r(0)$. Moreover  $B_r'(z)=\de\R^N_+\cap B_r(z)$ for $z\in \de\R^N_+$ and $B_r'=B_r'(0)$. Moreover  $B_r^+(z)=B_r(z)\cap \R^N_+$.\\
For  the   H\"older and Lipschitz seminorm of $u$, we write
$$
[u]_{C^{0,\a}(\O)}:=\sup_{x\not=y\in\O} \frac{|u(x)-u(y)|}{|x-y|^\a},
$$ 
for $\a\in (0,1]$. If there is no ambiguity, when   $\a\in (0,1)$, we will write $[u]_{C^{\a}(\O)}$ instead of $[u]_{C^{0,\a}(\O)}$. 
If $m\in \N$ and $\a\in (0,1)$, the H\"older space $\|u\|_{C^{m,\a}(\O)}$  is given by the set of functions in $C^m(\O)$ such that 
$$
\|u\|_{C^{m+\a}(\O)}:=\|u\|_{C^{m,\a}(\O)}=  \sup_{\g\in \N^N, |\g|\leq  m}   \| \de^\g  u\|_{ L^{\infty}(\O)}+ \sup_{\g\in \N^N, |\g|= m}  \|\de^\g u\|_{ C^{\a}(\O)}<\infty  .
$$
Recall that for $u,v\in H^s(\O)$, 
$$
\calC_{s,\O}(u,v):=  \frac{c_{N,s}}{2}\int_{\O\times \O}\frac{(v(x)-v(y))(u(x)-u(y))}{|x-y|^{N+2s}}\, dx dy.
$$
We note that for $u\in H^s_{loc}(\O)\cap L^1(\O;\frac{1}{1+|x|^{N+2s}})$ and  $v\in H^s(\O)$ with compact support in $\O$ then $ |\calC_{s,\O}(u,v)|<\infty$.
For $G_1$ and $G_2$  subsets of $\R^N$, with ${G_2} \subseteq G_1$ and $f\in L^1_{loc}(G_2)$, we say that  $u\in H^s_{loc}({G_1}) \cap L^1(G_1;\frac{1}{1+|x|^{N+2s}})$ satisfies
$$
\Ds_{ {G_1}} u= f\qquad \textrm{ in  $G_2$}
$$
if  $\calC_{s,G_1}(u,v)=\int_{  G_2}f v\, dx$ for all $v\in C^1_c(G_2\cup ( {G_1}\cap \de G_2  ))$.
\section{A priori H\"older  regularity estimate on the half-space}\label{s:Apri}
We introduce the reflection function with respect to the $e_N$ direction given by 
$$
\calR: \R^N\to \R^N, \qquad \calR(x)= x-2e_N\cdot x= (x',-x_N).
$$
Here and in the following, we define
$$
\calC_{e}:=\{\vp \in C^\infty_c(\R^N)\,:\, \vp (x)=\vp(\calR(x)), \qquad\textrm{ for all $x\in \R^N$ }\} 
$$
and
$$
H^s_{loc,e}(\R^N):= \{u \in H^s_{loc}(\R^N)\,:\, u (x',x_N)=u(\calR(x)), \qquad\textrm{ for all $x\in \R^N$}\} . 
$$
Moreover for $u\in H^s_{loc}(\ov{\R^N_+})$, we define
\begin{align}\label{eq:reflex-u}
\ov u(x',x_N):=
\begin{cases}
u(x',x_N)& \qquad\textrm{ for $x_N\geq 0$,}\\
u(x',-x_N)& \qquad\textrm{ for $x_N< 0$.}
\end{cases}
\end{align}

 We define the kernels $\ov K, K^+_s:\R^N\times \R^N\to [0,\infty]$ by 
 \be\label{eq:def-K_plus}
K^+_s(x,y):= c_{N,s} 1_{\R^N_+\times\R^N_+}(x,y) |x-y|^{-N-2s} ,
 \ee
\begin{align}\label{eq:Kernel-for-reflex}
\ov{K}(x,y):= 1_{x_N y_N> 0}(x,y) \frac{c_{N,s} }{2|x-y|^{N+2s}}+  1_{x_N y_N<0}(x,y)\frac{c_{N,s} }{2( |x'-y'|^2+(x_N+y_N)^2)^{\frac{N+2s}{2}}}.
\end{align}
We observe that $\ov K$ is an even extension of the kernel $K^+_s$ with respect to the $e_N$ direction. On the other hand,
the  kernel $\ov K$  is \textit{not}  strongly elliptic on  $\R^{N}\times \R^N$ (i.e. it is not bounded from above and below by $|x-y|^{-N-2s}$) but has some nice properties.
\begin{enumerate}
\item[(i)] For every $x,y\in \R^N$, 
\be \label{eq:ovK-sym}
\ov K(x,y)=\ov K(y,x)=\ov K(\calR(x), y) .
\ee
\item[(ii)] Since $(x_N+y_N)^2\leq (x_N-y_N)^2$ for $x_Ny_N\leq 0$, we have that  
\be\label{eq:ovK-coerciv}
\ov K(x,y) \geq \frac{c_{N,s}}{2|x-y|^{N+2s}} \qquad\textrm{  for all $x,y\in \R^N$.}
\ee
\item[(iii)] For a function $\eta:\R^N\times \R^N$ and measurable sets $A,B\subset\R^N$, we have that 
\begin{align} \label{eq:change-inK}
&\int_{A\times B}\eta(x,y)\ov K(x,y)\, dxdy \nonumber\\
& = \int_{A\times B} {\eta(x,y)}K^+_s(x,y) \, dxdy  +  \int_{\calR(A)\times \calR(B)}{\eta(\calR(x),\calR(y))}{K^+_s(x,y)}  \, dxdy\\
& + \int_{ A\times \calR(B)} {\eta(x,\calR(y))}{K^+_s(x,y)}\, dxdy+ \int_{ \calR(A)\times B}{\eta(\calR(x),y)} K^+_s(x,y)\, dxdy.  \nonumber
\end{align}
As a consequence if $A=\calR(A)$ and $B=\calR(B)$ then, for every $v\in \calH_e^s$, 
\be \label{eq:froKtoKplus}
\int_{A\times B}( v(x)- v(y))^2 \ov K(x,y)\, dxdy= 2 \int_{A\times B}{( v(x)-v(y))^2}K^+_s(x,y)  \, dxdy.
\ee
\end{enumerate}
We start with the following crucial result for getting a priori H\"older estimate below.
\begin{lemma}\label{lem:reflect-sol}
Suppose that $u\in H^s_{loc}(\ov{ \R^N_+})\cap L^1(\R^N_+; (1+|x|^{N+2s})^{-1})$ and $f\in L^\infty(\R^N)$ solves 
$$
\Ds_{\ov{\R^N_+}}u=f \qquad\textrm{ in  $\R^N_+$}
$$
i.e.  for all $\vp\in H^s_{loc}(\ov{ \R^N_+})$ with  Supp$(\vp)\subset \subset{ \R^N}$
\be \label{eq:even-extend-weak}
\int_{\R^N\times \R^N}( u(x)- u(y))(\vp(x)-\vp(y) )  K^+_s(x,y)\, dxdy=\int_{\R^N}f(x)\vp(x)\, dx.
\ee
Then  $\ov u \in H^s_{loc,e}({ \R^N})\cap  L^1(\R^N; (1+|x|^{N+2s})^{-1})$ and for all  $\vp\in H^s_{loc,e}(\R^N)$, with Supp$(\vp)\subset \subset{ \R^N}$
$$
\int_{\R^N\times \R^N}(\ov u(x)-\ov u(y))(\vp(x)-\vp(y) )\ov  K(x,y)\, dxdy=2\int_{\R^N}f(x)\vp(x)\, dx.
$$
\end{lemma}
\begin{proof}
The fact that $\ov u \in H^s_{loc,e}({ \R^N})$ follows from \eqref{eq:froKtoKplus}. Moreover by a change of variable, we have   $\ov u\in  L^1(\R^N; (1+|x|^{N+2s})^{-1})$.
We let  $\vp\in \calC_e$ and put $\eta(x,y)=(\ov  u(x)-\ov u(y))(\vp(x)-\vp(y) )$. Then  for all $x,y\in \R^N$,
$$
\eta (x,y)=\eta (y,x)=\eta (\calR(x), y).
$$
Using \eqref{eq:change-inK} and the fact that $\R^N=\calR(\R^N)$,   we get 
\begin{align*}
&\int_{\R^N\times \R^N}\eta(x,y)\ov K(x,y)\, dxdy= 2 \int_{\R^N\times \R^N}( u(x)- u({y}))(\vp(x)-\vp(y) )  K^+_s(x,y)\, dxdy.
\end{align*}
The proof is thus complete thanks \eqref{eq:even-extend-weak}.
\end{proof}

In our next result we prove  H\"older continuity of $\ov u$ (solution to \eqref{eq:even-extend-weak}) on $B_1$. The argument  is based on the De Giorgi iteration techniques. In view of the recent work  of  \cite{Cozzi}, it suffices to prove that $\ov u$ belongs to a fractional  De Giorgi class, as defined in \cite{Cozzi}. Since our proof uses  similar arguments as in \cite{Cozzi}, we shall only put more  attention on the fact that $\ov K$ is not bounded from above by $|x-y|^{-N-2s}$ on $\R^{N}\times \R^N$. But using the symmetry properties of the tests functions and the solution, we can reach our goal.\\

We fix the notations as in the aforementioned paper. For a real number $a$, we write $a_+=\max(a,0)$ and $a_-=\max(-a,0)$. For a measurable real valued function $v$ and a real number $k$,
$$
A^+_v(k):= \{x\in \R^N\,:\, v>k\}, \qquad  
A^+_v(k,x_0,r)=A^+_v(k)\cap B_r(x_0).
$$
\begin{proposition}\label{prop:DGclass}
Let $\ov u \in H^s_{loc,e}({ \R^N})\cap L^1(\R^N; (1+|x|^{N+2s})^{-1})$ and $f\in L^\infty(\R^N)$ satisfy, for all $\vp\in H^s_{loc,e}(\R^N)$, Supp$(\vp)\subset \subset{ \R^N}$, 
\be \label{eq:even-extend-weak-1}
\int_{\R^N\times \R^N}(\ov u(x)-\ov u(y))(\vp(x)-\vp(y) )\ov  K(x,y)\, dxdy=\int_{\R^N}f(x)\vp(x)\, dx.
\ee
Let $k\in \R$  and $x_0\in \R^N$, with $x_0\cdot e_N=0$. Then for  $0<\rho<\t\leq 1$, we have 
\begin{align*}
&\frac{1}{2} [w_\pm]_{H^s(B_\rho(x_0))} +\frac{1}{2}    \int_{B_{\rho}(x_0)} w_\pm(x )\int_{B_{\t}(x)  } \frac{w_\mp(y)}{|x-y|^{N+2s}} \, dy dx \nonumber\\
&\leq \frac{C(N) }{(\t-\rho)^2 (1-s)}\|w_\pm\|^2_{L ^2(B_\t(x_0))}+ \frac{16}{(\t-\rho)^{N+2s} } \|w_\pm\|_{L^1(B_\t (x_0))} \int_{|y|\geq \rho }\frac{ w_\pm(y)} {|y-x_0|^{N+2s}} \, dy\\
& +  \frac{2}{c_{N,s}} \|f\|_{L^\infty(\R^N)}^2A^\pm(k,x_0,\t).
\end{align*}
where $w_\pm=(\ov u-k)_\pm$.
\end{proposition}
\begin{proof}
We follow very closely the argument in  \cite[Proof of Proposition 8.5]{Cozzi}.         We let $\eta\in \calC_e$  be  such that $0\leq \eta\leq 1$, with Supp$(\eta)=\ov{B_{\frac{\rho+\t}{2}}}$ and $\eta=1$ on $B_\rho$.  For simplicity, we assume that $x_0=0$ and we write $A^\pm (k,r)=A^+_{\ov u}(k,0,r) $.
    We use $\vp =\eta^2 w_+ \in H^s_{loc,e}(\R^N)$  as a test function  in \eqref{eq:even-extend-weak-1} to get 
\begin{align*}
\int_{\R^N\times \R^N}(\ov u(x)-\ov u(y))(\vp(x)-\vp(y) )\ov  K(x,y)\, dxdy\leq \int_{\R^N}|f(x)| \eta^2(x) w_+(x)\, dx.
\end{align*}
As a consequence
\begin{align}\label{eq:sol-even}
&\int_{B_\t\times B_\t}(\ov u(x)-\ov u(y))(\vp(x)-\vp(y) )\ov  K(x,y)\, dxdy \\
&+2 \int_{B_\t\times \R^N\setminus B_\t}(\ov u(x)-\ov u(y))(\vp(x)-\vp(y) )\ov  K(x,y)\, dxdy \leq \int_{A^+(k,\frac{\t+\rho}{2})}|f(x)| \eta^2(x) w_+(x)\, dx.
\nonumber
\end{align}
Now from the argument in \cite[Section 8]{Cozzi}, we have 
\begin{align*}
\int_{B_\t\times B_\t}&(\ov u(x)-\ov u(y))(\vp(x)-\vp(y) )\ov  K(x,y)\, dxdy\\
&\geq \frac{1}{2} \int_{B_\rho\times B_\rho}  {(w_+(x)-w_+(y))^2}\ov K(x,y)\, dx dy  +\frac{1}{2} \int_{B_\rho}w_+(x)\int_{B_\t}w_-(y)\ov K(x,y)\, dy dx \\
&-2\int_{B_\t\times B_\t}\max(w_+(x),w_+(y))^2 {(\eta(x)-\eta(y))^2}\ov K(x,y)\, dx dy.
\end{align*}
Hence, by \eqref{eq:ovK-coerciv}, we have 
\begin{align}
\int_{B_\t\times B_\t}&(\ov u(x)-\ov u(y))(\vp(x)-\vp(y) )\ov  K(x,y)\, dxdy \nonumber\\
&\geq \frac{c_{N,s}}{2} [w_+]_{H^s(B_\rho)} +\frac{c_{N,s}}{2} \int_{B_\rho}w_+(x)\int_{B_\t}w_-(y)|x-y|^{-N-2s}\, dy dx  \nonumber\\
&-2\int_{B_\t\times B_\t}\max(w_+(x),w_+(y))^2 {(\eta(x)-\eta(y))^2}\ov K(x,y)\, dx dy. \label{eq:est1-DG}
\end{align}
Now by \eqref{eq:change-inK}, we get 
\begin{align*}
&\int_{B_\t\times B_\t}\max(w_+(x),w_+(y))^2 {(\eta(x)-\eta(y))^2}\ov K(x,y)\, dx dy\\
&= 4  \int_{B_\t\times B_\t}\max(w_+(x),w_+(y))^2 {(\eta(x)-\eta(y))^2}K_+(x,y)\, dx dy\\
&\leq  8 c_{N,s}  \|\n \eta\|_{L^\infty(B_{\frac{\rho+\t}{2}})}^2 \|w_+\|^2_{L ^2(B_\t)}\sup_{x}\int_{|x-y|<2\t} |x-y|^{-N-2s+2}\, dy\\
&\leq \frac{c_{N,s} C(N) }{(\t-\rho)^2 (1-s)}\|w_+\|^2_{L ^2(B_\t)}.
\end{align*}
Using this in \eqref{eq:est1-DG}, we find that 
\begin{align}
&\frac{1}{c_{N,s}}\int_{B_\t\times B_\t}(\ov u(x)-\ov u(y))(\vp(x)-\vp(y) )\ov  K(x,y)\, dxdy\\
&\geq \frac{1}{2} [w_+]_{H^s(B_\rho)} +\frac{1}{2} \int_{B_\rho}w_+(x)\int_{B_\t}w_-(y)|x-y|^{-N-2s}\, dy dx   -\frac{  C(N) }{(\t-\rho)^2 (1-s)}\|w_+\|^2_{L ^2(B_\t)} .
\label{eq:est2-DG}
\end{align}
Next, following once more  \cite[Section 8]{Cozzi}, for every $r_0>0$,  we also have  that 
\begin{align*}
&\int_{B_\t\times \R^N\setminus B_\t}(\ov u(x)-\ov u(y))(\vp(x)-\vp(y) )\ov  K(x,y)\, dxdy\\
&\geq   \int_{B_\rho} w_+(x )\int_{B_{r_0}(x)\setminus B_\t } w_-(y)\ov K(x,y)\, dy dx\\ 
&-2\int_{A^+_{\ov u}( k,\frac{\t+\rho}{2})} w_+(x) \int_{\{\ov u(y)>\ov u(x)\}\setminus B_\t}(\ov u(y)-\ov u(x))\ov K(x,y)\, dy\, dx\\
&\geq   \int_{B_\rho} w_+(x )\int_{B_{r_0}(x)\setminus B_\t } w_-(y)\ov K(x,y)\, dy dx -2\int_{A^+_{\ov u}( k,\frac{\t+\rho}{2})} w_+(x) \int_{\R^N\setminus B_\t}{ w_+(y)}\ov K(x,y)\, dy\, dx.
\end{align*}
This then implies that
\begin{align*}
&\int_{B_\t\times \R^N\setminus B_\t}(\ov u(x)-\ov u(y))(\vp(x)-\vp(y) )\ov  K(x,y)\, dxdy\\
&\geq   \int_{B_\rho} w_+(x )\int_{B_{r_0}(x)\setminus B_\t } w_-(y)\ov  K(x,y) \, dy dx -2\int_{B_{\frac{\t+\rho}{2}}} w_+(x) \int_{\R^N\setminus B_\t}{ w_+(y)}\ov K(x,y)\, dy\, dx.
\end{align*}
Moreover from    \eqref{eq:ovK-coerciv} and \eqref{eq:change-inK}, we get, 
\begin{align*}
&\int_{B_\t\times \R^N\setminus B_\t}(\ov u(x)-\ov u(y))(\vp(x)-\vp(y) )\ov  K(x,y)\, dxdy\\
&\geq c_{N,s}  \int_{B_\rho} w_+(x )\int_{B_{r_0}(x)\setminus B_\t } \frac{w_-(y)}{|x-y|^{N+2s} }\, dy dx -8\int_{B_{\frac{\t+\rho}{2}}} w_+(x) \int_{\R^N\setminus B_\t}{ w_+(y)} K_+(x,y)\, dy\, dx\\
&\geq  c_{N,s} \int_{B_\rho} w_+(x )\int_{B_{r_0}(x)\setminus B_\t } \frac{w_-(y)}{|x-y|^{N+2s} }\, dy dx -8c_{N,s}\int_{B_{\frac{\t+\rho}{2}}} w_+(x) \int_{\R^N\setminus B_\t} \frac{ w_+(y)}{|x-y|^{N+2s}} \, dy\, dx.
\end{align*}
Using that $|x-y|\geq \frac{\t-\rho}{2}|y|$ we then deduce that 
\begin{align}
&\int_{B_\t\times \R^N\setminus B_\t}(\ov u(x)-\ov u(y))(\vp(x)-\vp(y) )\ov  K(x,y)\, dxdy  \nonumber\\
&\geq  c_{N,s}  \int_{B_\rho} w_+(x )\int_{B_{r_0}(x)\setminus B_\t } \frac{w_-(y)}{|x-y|^{N+2s}} \, dy dx  -\frac{8c_{N,s} }{(\t-\rho)^{N+2s} } \|w_+\|_{L^1(B_\t )} \int_{|y|\geq \t }\frac{ w_+(y)} {|y|^{N+2s}} \, dy.
 \label{eq:aaa1}
\end{align}
Now, using Young's inequality, we can estimate
$$
 \int_{A^+(k,\frac{\t+\rho}{2})}|f(x)| \eta^2(x) w_+(x)\, dx\leq 2 \|f\|_{L^\infty(\R^N)}^2A^+(k,\frac{\t+\rho}{2})+ \frac{2}{(\t-\rho)^2}\|w_+\|^2_{L^2(B_\t)}.
 $$
Combining  the above estimate with  \eqref{eq:aaa1},   \eqref{eq:sol-even} and \eqref{eq:est2-DG}, we conclude that  
\begin{align*}
&\frac{c_{N,s}}{2} [w_+]_{H^s(B_\rho)} +\frac{c_{N,s}}{2}    \int_{B_\rho} w_+(x )\int_{B_{r_0}(x)  } \frac{w_-(y)}{|x-y|^{N+2s}} \, dy dx \nonumber\\
&\leq \frac{c_{N,s} C(N) }{(\t-\rho)^2 (1-s)}\|w_+\|^2_{L ^2(B_\t)}+ \frac{16 c_{N,s}}{(\t-\rho)^{N+2s} } \|w_+\|_{L^1(B_\t )} \int_{|y|\geq \t }\frac{ w_+(y)} {|y|^{N+2s}} \, dy\\
&+  2 \|f\|_{L^\infty(\R^N)}^2A^+(k,\t).
\end{align*}
Next, testing the equation with $\eta^2 w_-$, we get the same estimates as above, with $w_-$ in the place of $w_+$ and vice versa. Now by the translation invariance of the problem we get the result.
\end{proof}
Our  next result provides, in particular, H\"older continuity up to the boundary for   functions which  are harmonic with respect to $\Ds_{\ov{\R^N_+}}$. 
\begin{theorem}\label{th:the-Apri-Hold}
Let $f\in L^\infty(\R^N)$ and suppose that $u\in H^s_{loc}(\ov{ \R^N_+})\cap L^1(\R^N_+; (1+|x|^{N+2s})^{-1})$  solves 
$$
\Ds_{\ov{\R^N_+}}u=f \qquad\textrm{ in  $B_2^+$.}
$$
Then there exist $\a_0=\a_0(N,s)>0$ and $C=C(N,s)>0$ such that 
$$
\|u\|_{C^{\a_0}\left(\ov{B_1^+}\right)}\leq  C \left(  \|u\|_{L^2(B_2^+)}+ \int_{\{|y|\geq 1/2\}\cap\R^N_+}\frac{|u(y)|}{|y|^{N+2s}}dy+ \|f\|_{L^\infty(\R^N)}\right).
$$
\end{theorem}
\begin{proof}
By Lemma \ref{lem:reflect-sol},  $\ov u=u\circ\calR\in H^s_{loc,e}(\R^N)\cap  L^1(\R^N; (1+|x|^{N+2s})^{-1})$. Moreover by Proposition \ref{prop:DGclass}, $\ov u$ belongs to the De Giorgi  class in the sense of  \cite{Cozzi}.  We can therefore apply  \cite[Theorem 6.2  and Theorem 6.4]{Cozzi} to get 
$$
| \ov u(x)-\ov u(x_0)|  \leq C |x-x_0|^{\a_0} \left(  \|\ov u\|_{L^2(B_2)}+ \int_{\{|y|\geq 1/2\}}\frac{|\ov u(y)|}{|y|^{N+2s}} dy+ \|f\|_{L^\infty(\R^N)} \right)
$$
for all $x\in B_1$ and $x_0\in B_1'$. Combining this with interior regularity, we get the desired result.
\end{proof}

\section{  Liouville theorem}\label{s:Liouville} 
The main result of the present section is the following classification result of  functions $u\in H^s_{loc}(\ov{\R_+^N})  $ solving
\begin{equation}\label{eq:flat-eq-good-rr}
	\Ds_{\ov{\R^N_+}}u=0\qquad\text{ in  $\R_+^N$}
	\end{equation}
and having a  polynomial growth of order  smaller than $2s$.  
The argument we develop below will provides also a Liouville-type result in the case of the Censored fractional Laplacian $\Ds_{\R^N_+}$.  Its proof requires several preliminary results.
\begin{theorem}\label{th:Liouville}
Let $u\in H^s_{loc}(\ov{\R_+^N})   $  be a solution to \eqref{eq:flat-eq-good-rr}  and suppose that, for some constants $C>0$ and $\e\in (0,2s)$,
\be\label{eq:grow-final-liouville} 
\|u\|_{L^2(B_R)}\leq C R^{\frac{N}{2}+\e} \qquad\textrm{ for all $R\geq 1$}.
\ee
\begin{itemize}
\item[(i)] If $2s\leq  1$ then $u$ is a constant function on $\R_+^N$.
\item[(ii)] If $2s>1$ then $u(x',x_N)=a+  c\cdot x'$ for all $x=(x',x_N)\in \R_+^N$, for some constants $a\in \R$ and $c\in \R^{N-1}$.
%
\end{itemize}
\end{theorem}
The following result provides   regularity estimates of  the tangential derivatives of  the solution on $\R^N_+$. This takes advantages on the translation invariant of the problem  in the tangential direction which allows  to estimate  the incremental quotient of $u$ given by $\frac{u(x+h e)-u(x)}{|h|^\a}$, $e\in \R^{N-1}$, as in \cite[Section 5.3]{Cab-Caff}.   However, since we are dealing here with nonlocal operators,  one should cut off the solution at each step in order to deal with some tail.
\begin{proposition} \label{prop:high-tangderiv}
We consider  $u\in H^s_{loc}(\ov{\R^N_+})  \cap  L^1(\R^N_+; (1+|x|^{N+2s})^{-1})$ satisfying
\begin{equation}\label{eq:flat-eq-good}
\Ds_{\ov{\R^N_+}}u=0\quad\text{in}\quad\R^N_+.
	\end{equation}
	Then 
$$
	\|\n_{x'} ^2 u\|_{L^\infty(B_1^+)}\leq C\left(  \|u\|_{L^2(B_2^+)}+ \int_{\{|y|\geq 1/2\}\cap \R^N_+}|u(y)| |y|^{-N-2s}\, dy \right).
$$
	\end{proposition}
\begin{proof}
For simplicity, we assume that $\|u\|_{L^2(B_2^+)}+ \int_{\{|y|\geq 1/2\}\cap \R^N_+}|u(y)| |y|^{-N-2s}\, dy\leq 1$. Hence by  Theorem \ref{th:the-Apri-Hold}, we have $\|u\|_{C^\a(B_2^+)}\leq C$.
Let $r_i:=1+2^{-i}$ and $\phi_i\in C^\infty_c(B_{r_i} )$ such that $\phi_i=1$ on $B_{r_{i+1}}$.
Let $v_i=\phi_i u$ which satisfies
$$
	(-\Delta)^s_{\ov{\R^N_+}}v_i= f_i  \qquad\text{ in $ B_{r_{i+2}}^+$,}
$$
where 
$$
f_i(x):= \phi_{i+2}(x) \int_{|y|\geq r_{i}}v_i(y)1_{\R^N_+}(y) |x-y|^{-N-2s}\, dy.
$$
It clearly satisfies $\|f_i\|_{C^2(\R^N)}\leq C.$ In view of   Theorem \ref{th:the-Apri-Hold}, we get 
\be \label{eq;est-u_i-CC}
	\| v_i\|_{C^{\a}(\R^N_+)}\leq   C.
\ee
For $e\in \R^{N-1}$ and $|e|=1$ and $h\in \R\setminus\{0\}$, we define
$$
u^{h,e}(x):= \frac{u(x+h e)-u(x)}{|h|^\a}, \qquad 
$$
Then we can consider $w_{i,h,\a}(x)=  \frac{v_i(x+h e)-v_i(x)}{|h|^\a}$. Then by \eqref{eq;est-u_i-CC},  for  $h\in B_{r_{i+5}}$, 
$$
 \|w_{i,h,\a}\|_{L^\infty(\R^N)}\leq C  .   
$$
Moreover, by the translation invariant of the problem in the tangential direction, 
$$
	(-\Delta)^s_{\ov{\R^N_+}} w_{i,h,\a} = g_{i}  \qquad\text{ in $ B_{r_{i+7}}^+$,}
$$
with $$ g_i(x)= f_i^{h,e}(x)+   \phi_{i+5}(x) \int_{|y|\geq r_{i+5}} v_i^{h,e}(y)1_{\R^N_+}(y) |x-y|^{-N-2s}\, dy.$$
Clearly  $\|g_i\|_{C^2(B_{r_{i+6}})}\leq C.$
Therefore,  by Theorem \ref{th:the-Apri-Hold}, 
$$
\| w_{i,h,\a} \|_{C^\a(B_{r_i+8}^+ )}\leq C .
$$
Hence, see  \cite[Lemma  5.6]{Cab-Caff}, we have that $u$ in $C^{2\a}$ in the direction of $e$ if $2\a<1$ while $u$ in $C^{0,1}$ in the direction of $e$ if $2\a>1$. Iterating this procedure a finite number of times, we will find an integer $n$ such  that $u \in C^{0,1}(B_{r_{n}})$   and 
$$
	\|\n_{x'}  u\|_{L^\infty(B_{r_{n}})}\leq C.
$$
By the same argument as above considering $\de_{x_i} u$, with $i=1,\dots,N-1$, in the place of $u$, we  will get the Lipschitz bound of  $\de_{x_i} u$. This then gives the result.
\end{proof}

\begin{proposition}\label{prop:fromNDto1D}
We consider  $u\in H^s_{loc}(\ov{\R^N_+})  $ satisfying
\begin{equation}
\label{eq:flat-eq-good-r}
	\Ds_{\ov{\R^N_+}}u=0\quad\text{in}\quad\R^N_+
	\end{equation}
and 
$$
\|u\|_{L^2(B_R)}^2\leq C R^{N+2\e} \qquad\textrm{ for all $R\geq 1$,}
$$ 
for some $C>0$ and $\e<\min(1, 2s)$.
Then $u$ is one dimension. More precisely,  it  is a function on $\R_+$.
	\end{proposition}
	
	\begin{proof}
Since for all $R>0$, $u(\cdot R)$ solves \eqref{eq:flat-eq-good-r}, 	by Proposition \ref{prop:high-tangderiv} and Theorem \ref{th:the-Apri-Hold}, we have that 
	\begin{align*}
	\|\n_{x'} u\|_{L^\infty(B_R^+)}& \leq C R^{-1}\left(R^{-N/2}  \|u\|_{L^2(B_{2R}^+)}+R^{2s} \int_{\{|y|\geq R/2\}\cap \R^N_+}|u(y)|  |y|^{-N-2s}\, dy \right).
	\end{align*}
We now estimate, for $R>0$,	
	\begin{align*}
\int_{\{|y|\geq R\}\cap\R^N_+}\frac{|u(y)|}{|y|^{N+2s}}dy&\leq \sum_{i=0}^\infty  \int_{\{R 2^{i+1}\geq |y|\geq 2^iR\}\cap\R^N_+}\frac{|u(y)|}{|y|^{N+2s}}dy\\
&\leq   \sum_{i=0}^\infty (2^i R)^{-N-2s}\|u\|_{L^1(B_{2^{i+1}R})}\leq C \sum_{i=0}^\infty (2^i R)^{-N-2s} (2^{i+1}R)^{N+\e}\\
&\leq  C R^{-2s+\e}  \sum_{i=0}^\infty (2^i)^{-2s+\e} \leq C R^{-2s+\e}.
\end{align*}
We then deduce that $ \|\n_{x'} u\|_{L^\infty(B_R^+)}\leq  C R^{\e-1} $.
	Letting $R\to \infty$, we see that $\n_{x'} u=0$ on $\R^N_+$.
	\end{proof}
\subsection{The one dimensional problem }
This section is devoted to the classification of solutions  $u\in H^s_{loc}([0,\infty))\cap L^1_s(\R_+)$ satisfying some growth control and solving the equation
\be \label{eq:hal-to-study}
\Ds_{\ov{\R_+}} u=0 \qquad\textrm{ on $\R_+$,} 
\ee
where, here and in the following, we define
$$
L^1_s(A):=L^1(A;(1+|x|)^{-1-2s}) \qquad\textrm{ for $A$ a measurable subset of $\R$}.
$$
Let $v\in C^2_{loc}(0,\infty)\cap L^1_s(\R_+)$ and  define $\ti v=v 1_{(0,\infty)}$. 
A direct computation shows that, for $x\in \R_+$,   
\be \label{eq:def-ti-v}
\Ds \ti v (x)= \Ds_{{\R_+}}v(x)+  v(x)\int_{\R_-}\frac{c_{1,s}}{|x-y|^{1+2s}}\, dy= \Ds_{{\R_+}}\ti v(x)+ a_s x^{-2s} \ti v(x),
\ee
where $\Ds$ is the standard fractional Laplacian (see \eqref{eq:def-fracd-lap}) and  $a_s:=\frac{c_{1,s}}{2s} $. Indeed,  by a change of variable, we deduce that for $x\in \R_+$,
\be \label{eq:killing-meas-1D} 
\int_{\R_-}\frac{1}{|x-y|^{1+2s}}\,  dy= x^{-2s} \int^0_{-\infty} |1-r|^{-1-2s}\, dr=\frac{ x^{-2s}}{2s}.
\ee
As a consequence, if $u$ solves \eqref{eq:hal-to-study}  then  we find that  $\ti u=u 1_{(0,\infty)}$ solves 
\be \label{eq:ti-Hardy}
\Ds \ti u-a_s x^{-2s} \ti u=0 \qquad\textrm{ on $\R_+$.} 
\ee
The above discussion shows that  to study  \eqref{eq:hal-to-study}  we can consider \eqref{eq:ti-Hardy}. Now the study of the  latter problem leads  to the  study of optimal Hardy-type inequalities and allows  us to use the  local version of the problem via the  extension on the upper half-space $\R^2_+$.  
\\
We  define 
$$
\R^{2}_+:=\{z=(x,t)\in \R^{2}\,:\, t>0\}, \qquad \de  \R^{2}_+=:\R.
$$
We recall the Poisson kernel with respect to the extended operator $\textrm{div}(t^{1-2s}\cdot)$ on $\R^{2}_+$, see e.g. \cite{CSilv},  given by 
$$
P_{s}(z):=b_{s} \frac{ t^{2s}}{|z|^{1+2s}}=b_{s} \frac{ t^{2s}}{|(x,t)|^{1+2s}},
$$
where $b_s$ is a normalization constant  such that $\int_{\R}P_{s}(x,t)\, dx=1$ for all $t>0$.
We recall the following result \cite{CSilv}.
\begin{lemma}\label{lem:Caff-Energ}
Let $u \in L^1_s(\R)$ and define
$$
 w(x,t)=b_{s} t^{2s}\int_{\R} \frac{u(y)}{(|x-y|^2+ t^{2})^{(1+2s)/2}}\, dy .
$$
Then 
$$
\textrm{div}(t^{1-2s} \n w)=0 \qquad\textrm{ on $\R^{2}_+$.}
$$
\begin{itemize}
\item[$(i)$] If $u \in H^s_{loc}(\R)$,  then $w\in H^1_{loc}(\ov {\R^{2}_+};t^{1-2s} )$. Moreover for every $\Phi\in C^\infty_c( {\R^{2}} )$ we have 
$$
\int_{\R^{2}_+} t^{1-2s}\n w(x,t)\cdot \n \Phi(x,t) \, dxdt=\ov\k_s \frac{c_{1,s}}{2}\int_{\R^2}\frac{(u(x)-u(y))  (\Phi(x,0)-\Phi(y,0)) }{|x-y|^{1+2s}}\, dxdy,
$$
for some constant $\ov\k_s>0$.
\item[$(ii)$]
If   $u\in C(a,b)$ then 
$$
\lim_{t\to 0} w(x,t)=u(x) \qquad\textrm{ for all $x\in (a,b)$.}
$$
\item[$(iii)$] If $u\in C^{2}(a,b)$ then
$$
-\lim_{t\to 0} t^{1-2s}\de_t w(x,t)=\ov\k_s \Ds u(x)\qquad\textrm{ for all $x\in (a,b)$}.
$$
\end{itemize}

\end{lemma}
\begin{proof}
To check $(i)$, it suffices to write $u=\chi_R u+(1-\chi_R)u$ with $\chi_R\in C^\infty_c(\R)$ such that $\chi\equiv 1$ on $ (-2R,2R)$. Then one can easily see that   $w\in H^1((-R,R)\times (0,R)  ;t^{1-2s}  )$.\\
Now $(ii)$ and $(iii)$ are well known. 
\end{proof}
We recall the classification of the functions of the form $\o_\g:=x^{\g}1_{[0,\infty)}$ solving $\Ds_{\R_+} x^\g=0$ on $\R_+$, for $\g\geq 0$. Such functions are all known. Indeed for $\g\in (-1,2s)$, we have, see e.g. \cite{BD}, 
$$
\Ds_{\R_+} \o_\g=- \mu(\g,s) x^{-2s}\o_\g \qquad\textrm{ on $\R_+$},
$$
where 
$$
 \mu(\g,s)= c_{1,s}\int_{0}^1\frac{(t^\g-1)(1-t^{2s-1-\g})}{(1-t)^{1+2s}}\,dt.
$$
We  observe that the above integral is absolutely convergent and
\be\label{eq:def-mu-s-g}
  \mu(\g,s)=0 \qquad\textrm{ if and only if} \qquad \g\in \{0;2s-1\} .
\ee
We also notice, from \eqref{eq:def-ti-v},  that 
\begin{align}
\Ds  \o_\g=\left( a_{s} -\mu(\g,s)  \right)x^{-2s}\o_\g \qquad\textrm{ on $\R_+$}.
\end{align}
 Moreover,   by considering the extended function
 \be \label{eq:defw-gamma}
 w_\g(x,t):=b_{s} t^{2s}\int_{\R_+} \frac{\o_\g(y)}{(|x-y|^2+ t^{2})^{(1+2s)/2}}\, dy,
 \ee
 we get, thanks to Lemma \ref{lem:Caff-Energ},  
 \begin{align}\label{eq:eq-w-g}
\begin{cases}
\textrm{div}(t^{1-2s} \n w_\g)=0  & \qquad\textrm{ in $\R_+^2$}\\
-\lim_{t\to 0}t^{1-2s} \de_t w_\g=\ov \k_s \left( a_{s} -\mu(\g,s)  \right)  x^{-2s} w_\g & \qquad\textrm{ on $\R_+$}\\
w_\g=0 &\qquad\textrm{ on $ \R_-$}.
\end{cases}
\end{align}
 We recall that  $\o_{\frac{2s-1}{2}}$ is the "virtual" positive  ground  state leading to the sharp fractional Hardy inequality on the half-line. Indeed,  see e.g. \cite[Theorem 1.1]{BD}, for every $u\in C^\infty_c(\R_+)$, we have 
$$
\frac{c_{1,s}}{2} \int_{\R_+\times\R_+}\frac{(u(x)-u(y))^2}{|x-y|^{1+2s}}\,dxdy\geq  -\mu\left((2s-1)/{2},s\right)  \int_{\R_+}x^{-2s}u^2(x)\, dx.
$$
 The constant $-\mu(\frac{2s-1}{2},s) $ is optimal and it is positive for $2s\not=1$ and vanishes for $2s=1$.  
From this and \eqref{eq:killing-meas-1D}, we deduce the following  sharp  Hardy inequality  that, 
for every $u\in C^\infty_c(\R_+)$,   
 \be\label{eq:Hardy-ineq-R}
\frac{c_{1,s}}{2} \int_{\R\times\R}\frac{(u(x)-u(y))^2}{|x-y|^{1+2s}}\,dxdy\geq  \frac{\G\left( \frac{2s+1}{2}\right)^2}{\pi } \int_{\R}x^{-2s}u^2(x)\, dx.
 \ee
Here, we used that
\begin{align*}
\frac{c_{1,s}}{2} \int_{\R\times\R}\frac{(u(x)-u(y))^2}{|x-y|^{1+2s}}\,dxdy&= \frac{c_{1,s}}{2} \int_{\R_+\times\R_+}\frac{(u(x)-u(y))^2}{|x-y|^{1+2s}}\,dxdy\\
&+   {c_{1,s}} \int_{\R_+\times \R\setminus\R_+}\frac{(u(x)-u(y))^2}{|x-y|^{1+2s}}\,dxdy\\
&= \frac{c_{1,s}}{2} \int_{\R_+\times\R_+}\frac{(u(x)-u(y))^2}{|x-y|^{1+2s}}\,dxdy + a_{s} \int_{\R_+ }u^2(x)x^{-2s}\,dx
\end{align*}
and 
 see e.g. \cite[Page 630]{BD},  
 \be\label{eq:hardy-constants}
 \frac{\G\left( \frac{2s+1}{2}\right)^2}{\pi } =  a_{s} -\mu\left((2s-1)/{2},s\right)   .
 \ee
 By Lemma \ref{lem:Caff-Energ} and \eqref{eq:Hardy-ineq-R}, we find that, for all $w\in C^1_c(\R^2)$, with $w=0$ on $(-\infty,0]\times \{0\}$,  
 \be\label{eq:Hardy-Hal-extend}
 \int_{\R^2_+}|\n w(x,t)|^2 t^{1-2s} \,dxdt\geq \ov k_s  \frac{\G\left( \frac{2s+1}{2}\right)^2}{\pi } \int_{\R}x^{-2s}w(x,0)^2\, dx.
 \ee 
 Using polar coordinates, the above inequality yields a sharp inequality on  $S^1_+$. Indeed, let $\psi\in C^1[0,\pi]$ such that $\psi(\pi)=0$ and $f\in C^\infty_c(\R_+)$. Letting $w(r\cos(\th),r\sin(\th))=f(r)\psi(\th)$ in \eqref{eq:Hardy-Hal-extend}, we get 
 \begin{align*}
 \int_0^\infty (f'(r))^2r^{2-2s}\, dr\int_{0}^\pi \sin(\th)^{1-2s}\psi(\th)^2d\th&+ \int_0^\infty (f(r))^2r^{-2s}\, dr\int_{0}^\pi\sin(\th)^{1-2s}(\psi'(\th))^2d\th\\
 &\geq  \ov k_s  \frac{\G\left( \frac{2s+1}{2}\right)^2}{\pi } \int_0^\infty (f(r))^2r^{-2s}\, dr \psi(0)^2.
 \end{align*}
 This implies that 
\begin{align*}
\frac{ \int_0^\infty (f'(r))^2r^{2-2s}\, dr}{ \int_0^\infty (f(r))^2r^{-2s}\, dr   }\int_{0}^\pi \sin(\th)^{1-2s}\psi(\th)^2d\th+   \int_{0}^\pi\sin(\th)^{1-2s}(\psi'(\th))^2 \geq  \ov k_s \frac{\G\left( \frac{2s+1}{2}\right)^2}{\pi } \psi(0)^2.
 \end{align*} 
 Using the classical optimal Hardy inequality, we obtain that $$\inf_{f\in C^\infty_c(\R_+)} \frac{ \int_0^\infty (f'(r))^2r^{2-2s}\, dr}{ \int_0^\infty (f(r))^2r^{-2s}\, dr   }=\frac{(2s-1)^2}{4}.$$ We then deduce  immediately that 
 \begin{align}\label{eq:tra-Sph}
\frac{(2s-1)^2}{4}  \int_{0}^\pi \sin(\th)^{1-2s}\psi(\th)^2d\th+   \int_{0}^\pi\sin(\th)^{1-2s}(\psi'(\th))^2d\th \geq  \ov k_s \frac{\G\left( \frac{2s+1}{2}\right)^2}{\pi } \psi(0)^2.
 \end{align}
 We observe  the following crucial fact that for $2s\not=1$,   
\be \label{eq:asHard}
a_{s}< \frac{\G\left( \frac{2s+1}{2}\right)^2}{\pi },
\ee
which follows from \eqref{eq:hardy-constants} and the fact that  $-\mu\left((2s-1)/{2},s\right)>0$ for $2s\not=1$.
%
 We  consider the eigenvalue problem on $H^{1}((0,\pi);  \sin(\th) ^{1-2s} )$ given by 
\begin{align}\label{eq:11}
\begin{cases}
-  (  \sin(\th) ^{1-2s} \psi'  )'+\frac{(1-2s)^2}{4}\sin(\th) ^{1-2s} \psi =\l  \sin(\th) ^{1-2s} \psi& \qquad\textrm{ in $(0,\pi)$}\\
 -\lim_{\th\to 0}\sin(\th) ^{1-2s} \psi'(\th)= \ov \k_s a_{s}   \psi (0)&   \\
  \psi(\pi) = 0.&   
\end{cases}
\end{align}
 By  \eqref{eq:asHard} and \eqref{eq:tra-Sph}, for $2s\not=1$, it possesses a sequence of increasing eigenvalue $0< \l_1(s)< \l_2(s)\leq,\dots$, with corresponding   eigenfunctions $\psi_k$, normalized as
$$
\int_{ 0}^{\pi} \sin(\th)^{1-2s} \psi^2_k(\th)\, d\th=1.
$$
%
We have the following result.
\begin{lemma}\label{lem:lambda-2}
Suppose that $s\in (0,1)$. Then the following statements hold. 
\begin{itemize}
\item[$(i)$] If  $2s\not= 1$, then  $\l_{1}(s)=\frac{(2s-1)^2}{4}$.
\item[$(ii)$] If $2s\not= 1$ then $\frac{2s-1}{2} + \sqrt{\l_{2}(s)}\geq 2s$.
\item[$(iii)$] If $2s=1$, then   \eqref{eq:11} has a sequence of eigenvalues $0= \l_1(s)< \l_2(s)\leq,\dots$, with  $\psi_1(\th)=\frac{3}{\pi^3}(\pi-\th)$ and $\sqrt{\l_2(s)}>1$.  
\end{itemize}
\end{lemma}
\begin{proof}
To prove $(i)$, we consider $w_\g$ defined by \eqref{eq:defw-gamma}, with $\g=0$.   
It is clear that $w_0(z)=w_0(z/|z|)$ and letting $\widehat{w_0}(\theta)=w_0(\cos(\th),\sin(\th))$,  then $ \widehat{w_0}$  solves  \eqref{eq:11}, with $\l=\frac{(2s-1)^2}{4}$,   pointwise, thanks to \eqref{eq:eq-w-g}. Moreover,  by a change of variable
$$
\widehat{w_0}(\theta)= b_{s}\int_{-\cot(\th)}^\infty\frac{1}{(1+t^2)^{\frac{2s+1}{2}}}\, dt.
$$
From this, we find that $\widehat{w_0}'(\theta)=-b_s(\sin(\th))^{2s-1}$. This implies that  $w_0\in H^{1}((0,\pi);  \sin(\th) ^{1-2s} )$, because $w_0\in L^\infty(0,\pi)$. Since $w_0$ does not change sign we deduce that $\l_1(s)=\frac{(2s-1)^2}{4}$.  \\

To prove $(ii)$, we suppose on the contrary  that $q_+= \frac{2s-1}{2} + \sqrt{\l_{2}(s)}<2s$. Then consider the function 
$$
U(r\cos(\th), r\sin(\th ))=r^{q_+}\psi_2(\th) .
$$
Note that $q_+>0$,  since $\frac{(2s-1)^2}{4}=\l_1(s)<\l_2(s)$.  By direct computations, we have that $U$ solves
\begin{align}\label{eq:U-dplus}
\begin{cases}
\textrm{div}(t^{1-2s} \n U)=0  & \qquad\textrm{ in $\R_+^2$}\\
-t^{1-2s} \n U\cdot e_2= \ov\k_s a_s  x^{-2s} U & \qquad\textrm{ on $\R_+$}\\
U=0 &\qquad\textrm{ on $ \R_-$}.
\end{cases}
\end{align} 
We observe that  $U\big|_{\R}= \o_{q_+}\psi_2(0)  $. Recall that the extension of $\o_{q_+}$ is given by $w_{q_+}$ which solves \eqref{eq:eq-w-g}.  Since also $q_+>\frac{2s-1}{2} $, we have that $ U,w_{q_+}\in H^1_{loc}(\ov{\R^2_+}; t^{1-2s})$.  It follows that $V:= U-w_{q_+}\psi_2(0)    \in C^0(\ov{\R^2_+})$  solves 
\begin{align}
\begin{cases}
\textrm{div}(t^{1-2s} \n V)=0  & \qquad\textrm{ in $\R_+^2$}\\
%
%
V=0 &\qquad\textrm{ on $ \R$}.
\end{cases}
\end{align} 
We then consider the odd reflection (in the $t$-direction) of $V$ denoted by $\ti V\in H^1_{loc}({\R^2}; |t|^{1-2s})$. It solves, weakly, 
\begin{align}
\textrm{div}(|t|^{1-2s} \n \ti V)=0  & \qquad\textrm{ in $\R^2.$}
%
%
\end{align} 
Since by our  assumption the growth of  $V$  is of order  $q_+<2$, it then follows from a well  known Liouville theorem that $\ti V(x,t)=a t|t|^{2s-1}+b+ c x$, for some constants $a,b,c\in \R$, see e.g. \cite[Lemma 2.7]{CSS}. Since  $q_+<2s$,  the  continuity and oddness with respect to $t$ of  $\ti V$    imply that  $\ti V=V\equiv 0$ on $\ov{\R^2_+}$. Hence $U=w_{q_+}\psi_2(0)   $ on $\ov{\R^2_+}$.  Now from \eqref{eq:U-dplus} and the definition of $w_{q_+}$, we deduce that 
$$
\psi_2(0) \Ds \o_{q_+}  =\Ds (U\big|_{\R})=a_{s} x^{-2s} U\big|_{\R}= (a_{s} -\mu(q_+,s) )  x^{-2s} \o_{q_+}\psi_2(0) \quad \textrm{ on $\R_+$.}
$$
We then get  $\mu(q_+,s)=0$ and hence thanks to \eqref{eq:def-mu-s-g},  either $q_+=0$ or $q_+= 2s-1$. We now check that this leads to a contradiction. Indeed, each of these cases implies that    $\sqrt{\l_2(s)}=\pm\frac{2s-1}{2}$  so that $\l_2(s)=\l_1(s)$ which is not possible. \\
 
 Now for $2s=1$, we first notice that the explicit  solution to the first equation in \eqref{eq:11}, for $\l=0$,  are given by the affine  functions. Note, in this case,  that $\ov\k_{1/2}=1$ and $a_{1/2}=\frac{1}{\pi}$. Therefore a positive solution is given by $\pi-\th$. Hence, with the $L^2$-normalization, we find that $\psi_1(\th)=\frac{3}{\pi^3}(\pi-\th)$.\\
  Next, for $\l>0$, looking at eigenfunctions of the form $\psi(\th)=a\cos(\sqrt{\l}\th)+b\sin(\sqrt{\l}\th)$, then using the boundary conditions, we immediately see that $\tan(\sqrt{\l}\pi)=\sqrt{\l}\pi$. This equations has a discrete solutions $\l_2(s)<\l_3(s)<\dots$, with $\sqrt{\l_2(s)}\pi>\pi$. We thus conclude that $\l_2(s)>1$. 
 
\end{proof}
%
%
The following result shows that for  $u\in H^s_{loc}(\ov{\R_+}) \cap L^1_s(\R_+)$ solving \eqref{eq:hal-to-study} 
then $u-u(0)\in L^2((0,R);x^{-2s}) $ for all $R>0$.  With this property, we immediately deduce that $(u-u(0))1_{[0,\infty)}\in H^s_{loc}(\R) $, thanks to  \eqref{eq:def-calHs0}. This will be useful in order to use Lemma \ref{lem:Caff-Energ}.

\begin{lemma}\label{lem:barr}
Let $s\in (0,1)$ and  $u\in H^s_{loc}(\ov{\R_+})  \cap L^1_s(\R_+)$  be a solution to 
$$
\Ds_{\ov{\R_+}} u= 0\qquad\textrm{ on $\R_+$}.
$$
	Then $v:=(u-u(0))1_{[0,\infty)}\in H^s_{loc}(\R)$.
\end{lemma}

\begin{proof}
In view of \eqref{eq:def-calHs0}, it suffices to prove that $\int_{0}^1x^{-2s}v^2(x)\, dx<\infty$. In  the case $2s\leq1$, this follows immediately from Theorem \ref{th:the-Apri-Hold}. Next, we consider the case $2s>1$. Let  $G\in C^1(\R)$ be  given by $G(t)=0$ for $|t|<1$ and  $G(t)=t$  for $|t|\geq 2$. We define $v_k(x)=\frac{1}{k}G(k v(x))$.  By Theorem \ref{th:the-Apri-Hold}, we have that $v\in C([0,1))$ and $v(0)=0$. Therefore  $\vp v_k\in C_c(0,1)\cap H^s(0,1)\subset H^s_0(0,1)$ for all $\vp \in C^\infty_c([0,1))$. 
 By the Hardy inequality, see \cite[Theorem 1.4.4.3]{Grisv} or \cite[Theorem 1.1]{BD}, we have 
$$
\int_0^1x^{-2s}(\vp v_k)^2(x)\, dx\leq C[(\vp v_k)]_{H^s(0,1)}^2\leq C( [ v_k ]_{H^s(0,1)}^2+1)\leq C( [ v ]_{H^s(0,1)}^2  +1).
$$
By  Fatou's lemma, we deduce that $\int_{0}^1x^{-2s}v^2(x)\, dx<\infty$ as desired.
\end{proof}

\begin{proposition}\label{prop:classif-1D-2s-not-1}
Let $u\in H^s_{loc}(\ov{\R_+})  $ and  $w\in H^s_{loc}({\R})\cap C^0(\R)   $,  be such that 
\begin{equation*}
	\Ds_{\ov{\R_+}}u=0\quad\text{in}\quad \R_+,
	\end{equation*}
	\begin{equation*}\label{eq:flat-eq-good-r-Dir}
	\Ds_{{\R_+}}w=0\quad\text{in}\quad \R_+, \qquad w=0 \qquad\textrm{on $(-\infty,0]$} 
	\end{equation*}
and,  for some constants $C>0$ and $\e\in (0,2s)$, 
\be \label{eq:growth-final-Li}
\|u\|_{L^2(-R,R)}^2\leq C R^{1+2\e} \qquad\textrm{ for all $R\geq 1$}
\ee
and 
$$
\|w\|_{L^2(-R,R)}^2\leq C R^{1+2\e} \qquad\textrm{ for all $R\geq 1$}.
$$
\begin{itemize}
\item[(i)] If $s\in (0,1)$ then $u$ is a constant function on $\R_+$.
\item[(ii)] If $2s>1$ then  $w(x)=b  x^{2s-1}$ for all $x\in \R_+$, for some constant $b\in \R$.

\end{itemize}
\end{proposition}

\begin{proof}
By Lemma  \ref{lem:barr}, the function  $v=(u-u(0))1_{[0,\infty)}\in H^{s}_{loc}(\R)$.  We also have that (see \eqref{eq:def-ti-v}) 
\begin{equation*}
	\Ds_{{\R_+}}v=\Ds v-a_{s} x^{-2s} v=0\quad\text{ in }\quad \R_+.
	\end{equation*}
We then  define 
$V(x,t)=(P(t,\cdot)\star v)(x)$. By Lemma \ref{lem:Caff-Energ}, we have  $V\in H^1_{loc}(\ov{\R^2_+}; t^{1-2s})$. Moreover,
\begin{align}\label{eq:Usolves-}
\begin{cases}
\textrm{div}(t^{1-2s} \n V)=0  & \qquad\textrm{ in $\R_+^2$}\\
-t^{1-2s} \n V\cdot e_2= \ov\k_s a_s  x^{-2s} V & \qquad\textrm{ on $\R_+$}\\
V=0 &\qquad\textrm{ on $\R_-$}.
\end{cases}
\end{align} 	
We consider $(\psi_k)_k$, the complete sequence of  eigenfunction corresponding to the eigenvalue problem \eqref{eq:11}.  We then have that
$$
V(r\cos(\th), r\sin(\th ))=\sum_{k=1}^\infty \a_k(r) \psi_k(\th). 
$$
Then using \eqref{eq:Usolves-}, we find that the functions $\a_k$ solves the equation 
$$
\a_k''+\frac{2-2s}{r}\a_k'+\left(\frac{(2s-1)^2}{4}-\l_k(s) \right) \a_k=0 \qquad\textrm (0,\infty).
$$
We then get
$$
\a_k(r)=c^+_k r^{\frac{2s-1}{2}+\sqrt{\l_k(s)}}+c^-_k r^{\frac{2s-1}{2}-\sqrt{\l_k(s)}}, \qquad\textrm{for some constants $c_k^\pm\in \R$.}
$$
Noting that $V\in L^2((-1,1)\times (0,1); |z|^{-2}t^{1-2s})$ (by Hardy's inequality, see e.g. \cite{Fall-Felli}) we thus get $c^-_k=0$. It follows that 
$$
V(r\cos(\th), r\sin(\th ))=\sum_{k=1}^\infty c^+_k r^{\frac{2s-1}{2}+\sqrt{\l_k(s)}} \psi_k(\th). 
$$
From the growth assumption \eqref{eq:growth-final-Li}, the Parseval identity and  Lemma \ref{lem:lambda-2} we have that $c_k^+=0$ for all $k\geq 2$. Now since  $\l_1(s)=\frac{(2s-1)^2}{4} $, we obtain
$$
V(r\cos(\th), r\sin(\th ))=  c^+_1 r^{\frac{2s-1}{2}+\sqrt{\l_1(s)}} \psi_1(\th )= c^+_1 r^{\frac{2s-1}{2}+\frac{2s-1}{2}\textrm{sign}(2s-1)} \psi_1(\th) ,
$$
so that if $2s\leq 1$ then $V(r\cos(\th), r\sin(\th ))=c^+_1 \psi_1(\th)$, while $V(r\cos(\th), r\sin(\th ))=c^+_1 r^{2s-1}\psi_1(\th)$ in the case $2s>1$. 
Recalling that $V(x,0)=v(x)=u(x)-u(0)$ for all  $x\in\R_+$, we thus get $(i)$ for $2s\leq 1$. On the other hand for $2s>1$, we have $u(x)=u(0)+c^+_1 \psi_1(0) x^{2s-1}$ for all $x\in \R_+$.   Since\footnote{Communicated to the author by X. Ros-Oton and proved in \cite{AFR, Guan}. }   $\Ds_{\ov{\R_+}} x^{2s-1}\not=0$ on $\R_+$, we deduce that $c^+_1=0$ and thus  $u$ is a constant function on $\R_+$. \\
To obtain $(ii)$ we simply carry out the above argument  replacing $(u-u(0))1_{[0,\infty)}$ with $w$ (recall that by assumption $w(0)=0$) and we deduce that $w(x)=c^+_1 \psi_1(0) (x)_+^{2s-1}$.
\end{proof}

 
%

\subsection{Proof of Theorem \ref{th:Liouville} (completed)}
\begin{proof}
If $\e<\min (2s,1)$, then by  Proposition \ref{prop:fromNDto1D}, $u$ is a function of $x_N$. Applying Poposition \ref{prop:classif-1D-2s-not-1}, we get the result.\\
We now consider the case $2s>\e\geq 1$.  Since for all $R>0$, $u(\cdot R)$ solves the the same equation as $u$, 	by Proposition \ref{prop:high-tangderiv} and Theorem \ref{th:the-Apri-Hold}, we have that 
	\begin{align}\label{eq:growth-grad}
	\|\n_{x'} u\|_{L^\infty(B_R^+)}& \leq C R^{-1} \left( R^{-N/2}  \|u\|_{L^2(B_{2R}^+)}+R^{2s} \int_{|y|\geq R/2}|u(y)| 1_{\R^N_+}(y) |y|^{-N-2s}\, dy \right) \nonumber\\
	&\leq  C R^{\e-1},  
	\end{align}
where we used that, since $\e<2s$, then   $ \int_{|y|\geq R/2}|u(y)| 1_{\R^N_+}(y) |y|^{-N-2s}\, dy \leq C R^{-2s+\e}$.
Note that $0\leq \e-1<2s-1<1$. On the other hand, for $i=1,\dots,N-1$, we have that    $\de_{x_i}u\in H^s_{loc}(\ov{\R^N_+})\cap  L^1(\R^N_+; (1+|x|^{N+2s})^{-1})$ by Proposition \ref{prop:high-tangderiv} and  \eqref{eq:growth-grad}. Moreover it  satisfies $\Ds_{\ov{\R^N_+}}\de_{x_i} u=0$ on $\R^N_+$\footnote{Since $u_t:=\frac{u(\cdot+  t e_i)-u(\cdot)}{|t |}$, for $i=1,\dots,N-1$, satisfies $\Ds_{\ov {\R^N_+}} u_t=0$, thanks to the estimate of $\de_{x_i}u$ in \eqref{eq:growth-grad}, we can  send $t\to 0$ to see that $\Ds_{\ov{\R^N_+}}\de_{x_i} u=0$.}.  We then deduce,  from Proposition \ref{prop:classif-1D-2s-not-1}-$(i)$, Proposition \ref{prop:fromNDto1D} and \eqref{eq:growth-grad},    that  $\de_{x_i}u$ is constant on $\R^N_+$ for all $i$. Hence $u(x',x_N)=a(x_N)+b(x_N)\cdot x'$. Now noting that  ${u(x+h)-u(x)}=b(x_N)\cdot h$, for all $h\in \R^{N-1}$,  we see that\footnote{We observe that if $f$ is a function of $x_N$ only, then $\Ds_{\ov{\R^N_+}}f=\Ds_{\ov{\R_+}}f$.} $\Ds_{\ov{\R^N_+}}b= 0$ and thus from the growth assumption of $u$ and  Proposition \ref{prop:classif-1D-2s-not-1}, we obtain $b(x_N)=b(0)$ for all $x_N>0$. It is easy to see that $\Ds_{\ov{\R^N_+}} b(0)\cdot x_i=0$ on $\R^N_+$ for $i=1,\dots, N-1$.  From this we deduce that $\Ds_{\ov{\R^N_+}}u=\Ds_{\ov{\R^N_+}}a(x_N)=0 $ and therefore $a(x_N)=a(0)$,  thanks to Proposition \ref{prop:classif-1D-2s-not-1}-$(i)$.

\end{proof}


\section{Regularity estimates on open sets}\label{s:Reg-Neum}
We define a class of kernels arising after performing   a  parameterization that locally flatten the boundary of domain of $\R^N$.

For $\psi\in C^{0,1}(\R^{N-1})$,  we define   the kernel  $K^\psi_s:\R^N\times \R^N\to [0,\infty]$ by 
\be\label{eq:def-K-psi}
K^\psi_s(x,y):=  1_{\R^N_+}(x) 1_{\R^N_+}(y) \frac{c_{N,s}}{|\Psi(x)-\Psi(y) |^{N+2s}}, \qquad \Psi(x)=\Psi(x',x_N):=(x',x_N+\psi(x')).
\ee
We have  
\be \label{eq:defPsi}
|\Psi(x)-\Psi(y) |^2= |x-y|^2+ \mu(x,y) ,  \qquad  \mu(x,y):=2(x_N-y_N)(\psi(x')-\psi(y'))+   (\psi(x')-\psi(y'))^2.
\ee
  By the fundamental theorem of calculus, 
\be \label{eq:mu-control-n-psi}
\frac{|\mu(x,y)|}{|x-y|^2}\leq  \left( 2+\int_0^1|\n \psi(x'+t(y'-x'))|\, dt \right) \int_0^1|\n \psi(x'+t(y'-x'))|\, dt .
\ee
Recalling \eqref{eq:def-K_plus} for the definition of $K^+_s$.  Then  for $x,y\in \R^N_+$, we have  
\begin{align}
&{|x-y |^{N+2s}} \left( K^\psi_s(x,y)-K^+_s(x,y)\right) \nonumber \\
&=-(2s+N){|x-y |^{N+2s}}  \mu(x,y)  \int_0^1 \left(  |x-y|^2+t\mu(x,y)  \right)^{-\frac{N+2s+2}{2}}\,dt \nonumber\\
&=-(2s+N) c_{N,s} \frac{\mu(x,y)}{|x-y|^2} \int_0^1 \left(  1+t \frac{\mu(x,y)}{|x-y|^2}  \right)^{-\frac{N+2s+2}{2}}\,dt .
 \label{eq:remainder-Kernel-flat}
\end{align}
We then get from \eqref{eq:mu-control-n-psi}, for all  $x,y\in \R^N_+$
\be \label{eq:remainder-Kernel-flat0}
|x-y|^{N+2s}|K^\psi_s(x,y)-K^+_s(x,y)| \leq \frac{ c_{N,s}(N+2s) (2+\|\n \psi\|_{L^\infty(\R^{N-1})}) \|\n \psi\|_{L^\infty(\R^{N-1})} }{\left( 1- (2+\|\n \psi\|_{L^\infty(\R^{N-1})}) \|\n \psi\|_{L^\infty(\R^{N-1})} \right)^{\frac{N+2s+2}{2}} }.
\ee
 In the following, for a  kernel $K:\R^N\times \R^N\to [-\infty,\infty]$,  we define  the   bilinear form
\be\label{eq:def-dir-form}
\calD_{K }(u,v)=\frac{1}{2}\int_{\R^{2N}}(u(x)-u(y))(v(x)-v(y)) K(x,y)\, dxdy.
\ee
\subsection{A priori estimates}
We recall the following Caccioppoli type inequality,  see e.g. \cite[Lemma 9.1]{Fall-reg-1}.
\begin{lemma}\label{lem:non-loc-caciop}
Let $\O$ be an open  set, $h\in L^1_{loc}(\O)$ and $\cK$ a nonnegative and symmetric  function defined on $\R^\N\times \R^N$ such  that 
$$
 \int_{\R^{2N}} (w(x)-w(y))(\vp(x)-\vp(y))\cK(x,y)\, dxdy=\int_{\R^N}h(x)\vp(x)\, dx \qquad\textrm{ for all $\vp \in C^\infty_c(\O) $.}
$$
Then for all $\e>0$ and all $\vp \in C^\infty_c(\O) $, we have 
\begin{align*}
(1-\e)\int_{\R^{2N}}(w(x)-w(y))^2&\vp^2(y) \cK(x,y)\,dy dx\leq  \int_{\R^N} |h (  x)| |w(x)|  \vp^2(x)\, dx\\
&+   \e^{-1}   \int_{\R^{2N}}w^2(x)  (\vp(x)-\vp(y))^2\cK(x,y)\, dy dx.
\end{align*}
\end{lemma}
%
%
%
Next, we prove the following   result which is the first step to get  regularity up to the boundary in $C^1$ domains.
 \begin{proposition} \label{prop:bound-Kato-abstract}
Let  $s_0,\ov \a\in (0,1)$ and $p>\frac{N}{2s_0}$.     Then there exist $\e_0,C>0$ such that  for every $\psi\in C^{1}(\R^{N-1})$, $s\in [s_0,1)$  and for every
$ g\in L^p(\R^N)$,  $ v\in H^s(\R^N_+) $  satisfying 
$$
 \| \n \psi\|_{L^{\infty}(\R^{N-1})}<\e_0 
$$
and 
$$
\calD_{K^\psi_s}(v,\vp)=\int_{B_2^+}g(x)\vp(x)\, dx \qquad\textrm{ for all $\vp \in C^\infty_c(B_2) $, }
$$
we have 
\be \label{eq:to-contradict-Propok}
\sup_{r>0} r^{-2\a-N}\sup_{z\in B_1'}  \|  v-  v_{B_r^+(z)}\|_{L^2(B_r^+(z))}^2 \leq  C    (  \|v\|_{  L^2(\R^N_+) } +  \|g\|_{ L^p(\R^N) } )^2,
\ee
where
$u_{A}=\frac{1}{|A|}\int_{A}f(x)\, dx$ and $\a:= \min( 2s-N/p, \ov \a)$.
\end{proposition}
 \begin{proof}
Assume that  the assertion in the proposition does not hold, then   for every   $n\in \N$,  there exist $\psi_n\in C^{1}(\R^{N-1})$,  $s_n\in [s_0,1)$,  $ g_n\in L^p(\R^N)$,  $ v_n\in H^s( \R^N_+)$, with 
$$
\| g_n \|_{L^p(\R^N) } + \|v_n \|_{ L^2(\R^N_+)  }\leq 1
$$
satisfying 
 \be\label{eq:smalL-infty-psy-param}
 \| \n \psi_n\|_{L^{\infty}(\R^{N-1})}<\frac{1}{n}
\ee
and 
\be\label{eq:frac-half-flat-n}
\calD_{K^{\psi_n}_{s_n}}(v_n,\vp)=\int_{B_2^+}g_n(x)\vp(x)\, dx \qquad\textrm{ for all $\vp \in C^\infty_c(B_2) $, }
\ee
while
$$
\sup_{r>0} r^{-N-2\a_n } \sup_{z\in B_1'}  \| v_n-( v_n)_{B_{ r}(z)}\|_{L^2(B_{r}(z))}^2 > n.
$$
where $ \a_n:=\min(\ov\a, 2s_n-N/p))$.
Consequently,  there exists $\ov r_n>0$ and  $z_n \in B_1'$ such that 
\be\label{eq:ov-r-n-un-larger-n-eps}
 \ov r^{-N-2\a_n }_n   \| v_n-( v_n)_{B_{\ov r_n}^+(z_n)}\|_{L^2(B_{\ov r_n}^+(z_n))}^2 > n/2  .
\ee
We consider the (well defined, because $\|  v_n\|_{ L^2(\R^N_+)} \leq 1$) nonincreasing function $\Theta_n: (0,\infty)\to [0,\infty)$ given by
$$
\Theta_n(\ov r)=\sup_{ r\in [ \ov r,\infty)}    r^{-N-2\a_n}   \| v_n-(  v_n)_{B_{ r}^+(z_n)}\|_{L^2(B_{ r}^+(z_n))}^2  .
$$
Obviously, for $n\geq 2$, by \eqref{eq:ov-r-n-un-larger-n-eps},   
\be\label{eq:The-n-geq-n}
  \Theta_n(\ov r_n)> n/2\geq 1 .
\ee
Hence, provided $n\geq 2$,  there exists $r_n\in [\ov r_n,\infty)$ such that 
\begin{align*}
\Theta_n( r_n)&\geq   r^{-N-2\a_n }_n   \| v_n - (  v_n)_{ B_{ r_n}^+ (z_n)  } \|_{L^2(B_{ r_n}^+(z_n))}^2\\
&  \geq \Theta_n(\ov r_n)-1/2\geq (1-1/2)\Theta_n(\ov r_n)\geq \frac{1}{2}\Theta_n( r_n),
\end{align*}
where we used the monotonicity of $\Theta_n$ for the last inequality, while  the first inequality  comes from the definition of $\Theta_n$.
In particular, thanks to  \eqref{eq:The-n-geq-n},  $\Theta_n( r_n)\geq  n/4$. Now  since $ \|  v_n \|_{L^2(\R^N)} \leq 1$,  we have that $  r^{-N-2\a }_n  \geq n/8$, so that $r_n\to 0$ as $n\to \infty$.  
 We now   define the   sequence  of functions
$$
 {w}_n(x)=  \Theta_n(r_n)^{-1/2} r_n^{-\a_n}   \left\{ v_n(r_n x+z_n)   - \frac{1}{|B_1^+|} \int_{B_1^+}  v_n(r_n x+z_n) \, dx  \right\}, 
$$
which,    satisfies
\be \label{eq:w-n-nonzero}
 \|w_n\|_{L^2(B_{1}^+)}^2 \geq \frac{1}{2}  \qquad\textrm{and} \qquad \int_{B_1^+}w_n(x)\,dx =0 \qquad\textrm{ for every $n\geq2$.}
\ee
%
 %
 Using that, for every $r>0$ and $z\in B_1'$,  $    \| v_n-(  v_n)_{B_{ r}^+(z)}\|_{L^2(B_{ r}^+(z))}^2\leq r^{N+2\a_n}\Theta_n(r)$  and the monotonicity of $\Theta_n$, by   \cite[Lemma 3.1]{Fall-reg-1}, we find that 
\be\label{eq:groht-w-n-abs}
 \|w_n\|_{L^2(B_{R}^+)}^2 \leq  C   R^{N+2\a_n}  \qquad\textrm{ for every $R\geq 1$ and $n \geq 2$,}
\ee
for some constant $C=C(N,s_0,p)>0$.\\
%
%
%
%
%
%
%
We define
$$
  K_n(x,y)= K^{\ti \psi_n}_{s_n}(x,y), \qquad\textrm{ where  $\ti \psi_{n}(x)=\frac{1}{r_n}\psi_n(r_n x+z_n)$.}
$$

Since $|z_n|\leq 1$, we have  that $B_{\frac{1}{2r_n}}(0)\subset B_{\frac{2}{r_n}}(\frac{-z_n}{r_n}) $.
Letting
$$
 \ov g_n(x)= \frac{r_n^{2s-\a_n}}{\Theta_n(r_n)^{\frac{1}{2}}}  g_n (r_n x+z_n),
$$ 
 then by a change of variable,  we  obtain that, for all $\vp \in C^\infty_c(B_\frac{1}{2r_n}(0))$ , 
\be\label{eq:frac-half-flat-n-scal}
\frac{1}{2}\int_{\R^{2N}} (w_n(x)-w_n(y))(\vp(x)-\vp(y))K_n(x,y)\, dxdy=\int_{\R^N_+}\ov g_n(x)\vp(x)\, dx.
\ee
We fix $M>1$ and let $n\geq 2$ large, so that $1<M<\frac{1}{2r_n}$.  Let $\chi_M\in C^\infty_c(B_{4M})$ be such that $\chi_M=1$ on $B_{2M}$ and define $W_n:=\chi_M w_n$. Then by  letting  $\vp=\chi_{M/4}\in C^\infty_c(B_{M})$,  we find that
\be\label{eq:frac-half-flat-n-scal-cut}
\frac{1}{2} \int_{\R^{2N}} (W_n(x)-W_n(y))(\vp(x)-\vp(y))K_n(x,y)\, dxdy=\int_{\R^N_+}\ti g_n(x)\vp(x)\, dx,
\ee
with, thanks to \eqref{eq:groht-w-n-abs},  
$$
\|\ti g_n\|_{L^p(\R^N)}\leq \|\ov g_n\|_{L^p(\R^N)} + C(M) .
$$
  By Lemma \ref{lem:non-loc-caciop},   we have 
\begin{align}
&(1-\e)\int_{\R^{2N}}(W_n(x)-W_n(y))^2\vp^2(y) K_n(x,y)\,dy dx\leq 2 \int_{B_M^+} |\ti  g_n  (  x)| |W_n(x)|  \vp^2(x)\, dx \nonumber\\
&+ 2  \e^{-1}   \int_{\R^{2N}}W^2_n(x)  (\vp(x)-\vp(y))^2K_n(x,y)\, dy dx \label{eq:after-Cacc}\\
&\leq 2 \int_{B_M^+} |\ti  g_n  (  x)| |W_n(x)|  \vp^2(x)\, dx+  \e^{-1} \frac{c_{N,s_n}}{1-s_n}C(M)  \|w_n\|_{L ^2(B_{2M})}. \nonumber
\end{align}
By Young's inequality and Sobolev inequality, we have that 
\begin{align*}
\int_{B_M^+} |\ti  g_n  (  x)| W_n(x)|  \vp^2(x)\, dx&\leq \e \int_{B_M^+} |\ti  g_n  (  x)|W_n^2(x)  \vp^2(x)\, dx+ C_\e \int_{B_M^+} |\ti  g_n  (  x)|   \vp^2(x)\, dx\\
&\leq \e C(1-s_n)  \|\vp \ti  g_n\|_{L^p(\R^N)}\|\vp W_n \|_{H^s(\R^N_+)}^2+ C(M)\|\ti g_n\|_{L^p(\R^N)}\\
&\leq    C \e    \| \ti g_n\|_{L^p(\R^N)} \int_{\R^{2N}}(W_n(x)-W_n(y))^2\vp^2(y) K_n(x,y)\,dy dx\\
& +    C  \|w_n\|_{L ^2(B_{2M})}+  C(M)\|\ti g_n\|_{L^p(\R^N)}\\
&\leq C(M) \left(\e \int_{\R^{2N}}(W_n(x)-W_n(y))^2\vp^2(y) K_n(x,y)\,dy dx+ 1    \right),
\end{align*}
where we used \eqref{eq:groht-w-n-abs} in the last inequality.
Recall that $W_n=w_n$ on $B_{M/2^+}$. Using the above inequality in \eqref{eq:after-Cacc}, provided $\e$ is small,  we obtain 
\be\label{eq:local-bnd-w_n}
(1-s_n) [w_n]_{ H^{s_n}\left( B_{M/2}^+ \right) }^2\leq  \int_{\R^{2N}}(W_n(x)-W_n(y))^2\vp^2(y) K_n(x,y)\,dy dx\leq C(M).
\ee
Let $\ov s=\lim_{n\to \infty}s_n\in [s_0,1]$.
Thanks to \eqref{eq:local-bnd-w_n}, we have that $w_n$ is bounded in $H^{s_n}_{loc}(\ov{\R^N_+})$.  Hence by a diagonal argument, up to a subsequence, there exists $ w\in H^{\ov s}_{loc}( \ov{\R^N_+})$ such that 
\be \label{eq:w-n-conv-1}
w_n\to  w \qquad\textrm{ in $L^2_{loc}(\ov{\R^N_+})$} 
\ee
and, for all $\s<\ov s$, 
\be \label{eq:w-n-conv-2}
w_n \to  w \qquad\textrm{ in $H^{\s}_{loc}(\ov{\R^N_+})$} .
\ee
Therefore passing to the limit in \eqref{eq:w-n-nonzero} and \eqref{eq:groht-w-n-abs},  we obtain
\be \label{eq:w-n-nonzero-l}
 \|w\|_{L^2(B_{1}^+)}^2 \geq \frac{1}{2} \qquad\textrm{and} \qquad \int_{B_1^+}w(x)\,dx =0 
\ee
and 
\be\label{eq:groht-w-n-abs-l}
 \|w\|_{L^2(B_{R}^+)}^2 \leq  C   R^{N+2 \b}  \qquad\textrm{ for every $R\geq 1$,}
\ee
where $\b=\min (2\ov s-N/p,\ov \a)$.
We observe that \eqref{eq:groht-w-n-abs-l} and \eqref{eq:w-n-conv-1}  imply
\be \label{eq:w-n-conv-3}
\int_{\R^N_+}\frac{|w_n(x)-  w(x)|}{1+|x|^{N+2s_n}}\, dx \to 0 \qquad \textrm{ as  $n\to \infty$.} 
\ee
Next by \eqref{eq:frac-half-flat-n-scal} and the choice of $\a_n$, we have  for all  $\phi\in C^\infty_c(B_M)$, with $M$  as above, 
\begin{align} 
&\frac{1}{2} \left|\int_{\R^{2N}}(w_n(x)-w_n(y))(\phi(x)-\phi(y)) K_n(x,y)\, dxdy \right|\leq \|\ov g_n\|_{L^p(\R^N)} \|\phi\|_{L^{\frac{p}{p-1}}(\R^N)} \nonumber\\
 &\leq   \Theta_n(r_n)^{-1/2} r_n^{2s-\a_n-N/p}   \| g_n\|_{L^p(\R^N)} \|\phi\|_{L^{\frac{p}{p-1}}(\R^N)} \nonumber\\
 &\leq  \Theta_n(r_n)^{-1/2}  \|\phi\|_{L^{\frac{p}{p-1}}(\R^N)}. \label{eq:close-Liouvill}
\end{align}
From \eqref{eq:smalL-infty-psy-param} and \eqref{eq:remainder-Kernel-flat}, we get
\be \label{eq:K-near-flat-not}
  |K_n(x,y)-    K_{s_n}^+(x,y) | \leq  \frac{C}{n}{|x-y |^{-N-2{s_n}}} 1_{\R^N_+}(x)1_{\R^N_+}(y).
\ee
On the other hand, for all $x,y\in \R^N_+$
\be\label{eq:Ksn-Ksovs}
 |K_{ s_n}^+(x,y)-K_{\ov s}^+(x,y)|\leq C|s_n-\ov s| \frac{1+|\log|x-y||}{|x-y|^{N}}\max(|x-y|^{-2\max(s_n,\ov s)}, |x-y|^{-2\min(s_n,\ov s)})   .
\ee
Hence if $\ov s<1$,  then  by \eqref{eq:K-near-flat-not}, \eqref{eq:w-n-conv-2},  \eqref{eq:w-n-conv-3} and using that  $\Theta_n(r_n)\to\infty$ as $n\to\infty$, we deduce from \eqref{eq:close-Liouvill}  that \footnote{To see \eqref{eq:flat-eq-w}, we used that  $\left|\int_{B_M\times B_M}|(w(x)-w_n(x))-(w(y)-w_n(y))||\phi(x)-\phi(y)| K_n(x,y)\, dxdy \right|\leq [w-w_n]_{H^\s(B_M^+)}[\phi ]_{H^{2s_n-\s}(B_M)}  \to 0$ as $n\to \infty$, provided that $\s$ is close  to $\ov s$. Now, \eqref{eq:groht-w-n-abs} and \eqref{eq:groht-w-n-abs-l} imply that $\|w-w_n\|_{L^2(B_R^+)}^2\leq C R^{N+2s_n-\e}$ for all $R>1$, for some $\e>0$. From  these estimates, we get $\left|\int_{\R^N\times \R^N}|(w(x)-w_n(x))-(w(y)-w_n(y))||\phi(x)-\phi(y)| K_n(x,y)\, dxdy \right| \to 0$ as $n\to \infty$. Similarly, using \eqref{eq:K-near-flat-not}, \eqref{eq:Ksn-Ksovs} and  that $\phi\in C^1_c(\R^N)$, we also have $\left|\int_{\R^N\times \R^N}|w(x)-w(y)||\phi(x)-\phi(y)|| K_n(x,y)-K_{\ov s}^+(x,y)|\, dxdy \right| \to 0$ as $n\to \infty$. }
\be \label{eq:flat-eq-w}
\int_{\R^{N}_+\times \R^N_+  }(w(x)-w(y))(\phi(x)-\phi(y))|x-y|^{-N-2\ov s}\, dxdy =0, \qquad \textrm{  for all $\phi\in C^\infty_c(\R^N)$.}
\ee
%
Since $w\in H^{\ov s}_{loc}(\ov{\R^N_+})$, by Theorem \ref{th:Liouville},  we deduce that $w\equiv Const.$ on $\R^N_+$.  This leads to a contradiction in   \eqref{eq:w-n-nonzero-l}. \\

If now $\ov s=1$, then by  Lemma \ref{lem:nonloc-to-loc} and \eqref{eq:close-Liouvill}, we obtain
$$
\int_{\R^N_+}\n w(x)\cdot \n \phi(x)\, dx=0 \qquad\textrm{ for all $\phi \in C^\infty_c(\R^N)$.}
$$
From this and \eqref{eq:groht-w-n-abs-l}, it follows from the well known Liouville theorem for the Laplacian that $w\equiv Const.$ on $\R^N_+$.  This leads to a contradiction in   \eqref{eq:w-n-nonzero-l}. 

\end{proof}
We state the following result from \cite{RS-1} that we  will use frequently in the following.
\begin{lemma}\label{lem:Prop-of-RS}
Let $ u\in L^\infty(B_1^+)$ and  suppose that for every $x=(x',x_N)\in B_{1}^+$, 
$$
|u(x)| +[u]_{C^\g(B_{\frac{x_N}{2}}(x))}\leq C, \qquad\textrm{ for some constant $C>0$ and $\g\in (0,1)$.}
$$
Then there exists $\ov C >0$ and $r>0$ depending only on $N,C,\g$ such that 
$$
\|u\|_{C^\g(B_{r}^+)}\leq \ov C.
$$
\end{lemma}
As a consequence of Proposition \ref{prop:bound-Kato-abstract}, we get the 
\begin{corollary}\label{eq:cor-reg}
Let  $s_0, \ov\a\in (0,1)$.      Let  $\psi\in C^{1}(\R^{N-1})$, $s\in [s_0,1)$ and $p>\frac{N}{2{s_0}}$. Consider  
$ g\in L^p(\R^N)$,  $ v\in H^s(\R^N_+) $  satisfying 
\be\label{eq:frac-half-flat}
\calD_{K^\psi_s}(v,\vp)=\int_{B_2^+}g(x)\vp(x)\, dx \qquad\textrm{ for all $\vp \in C^\infty_c(B_2) $, }
\ee
Then there exist  $\e_0,C>0$ depending only on $s_0,N,p$ and $\ov\a$   such that if
$$
[\psi]_{C^{1}(\R^{N-1}) }<\e_0,
$$
we have 
\be \label{eq:Calpha-v}
\|v\|_{C^{\a}(\ov{B_{1}^+})}\leq C \left(  \|v\|_{L^2(\R^N_+)}+ \|g\|_{L^p(\R^N)}  \right),
\ee
with $\a:=\min (\ov\a, 2s-N/p)$.
\end{corollary}
\begin{proof}
In view of \eqref{eq:to-contradict-Propok}, by a classical iteration argument (see  e.g. \cite{Campanato}), we can find a constant $C_0=C_0(N,s_0, p,\b,\ov c,\e_0)>0$   and a function  $m\in L^\infty(B_1')$, with $\|m\|_{L^\infty(B_1') }\leq C_0$ such that 
\be \label{eq:to-contradict-Propok-cool}
\sup_{r>0} r^{-2\a-N}\sup_{z\in B_1'}  \|  v- m(z)\|_{L^2(B_r^+(z))}^2 \leq  C_0    (  \|v\|_{  L^2(\R^N_+) } +  \|g\|_{ L^p(\R^N) } )^2,
\ee
Without loss of generality, we assume that $ \|v\|_{  L^2(\R^N_+) } +  \|g\|_{ L^p(\R^N) }\leq 1$.
Let $\x:=(\x',\x_N) \in B_{1}^+$ and $\rho:=\x_N/2$.  We define  
$$
 u_\x(y)=  v(\x+\rho y)- m(\x') = v(\x'+2\rho e_N+\rho y)- m (\x')   
 $$
  and $\ti g(y)=g(\x+\rho y)$.   We have, for all $\vp\in C^\infty_c(B_2)$ 
\begin{align*}
\calD_{K_{\rho,s}} (u_\x,\vp)&= \rho^{2s}  \int_{B_2}\ti g(y)\vp(y) \,dy,
\end{align*}
where $K_{\rho,s}(x,y)=1_{\{x_N>-2\}}(x)1_{\{y_N>-2\}}(y)\frac{c_{N,s}}{|\Psi_\rho(x)-\Psi_\rho(y)|^{N+2s}}$ and   $\Psi_\rho( x)=\frac{1}{\rho}\Psi(\rho x+\x).$  By interior regularity, see e.g. \cite[Theorem 4.1]{Fall-reg-1}, we have 
\begin{align}\label{eq:Hold-estim}
\|u_\x\|_{C^{\a}(B_{1/2})} &\leq C\left( \|u_\x\|_{L^2(B_1)}+\int_{\{|y|\geq1/2\}\{y_N>-2\}}\frac{|u_\x(y)|}{|y|^{N+2s}}\, dy+\rho^{2s} \|\ti g\|_{L^p(\R^N)}    \right) \nonumber\\
& \leq C\left(  \|u_\x\|_{L^2(B_1)}+\int_{ \{|y|\geq1/2\}\{y_N>-2\}}\frac{|u_\x(y)|}{|y|^{N+2s}}\, dy + \rho^{2s-N/p}  \right).
\end{align}
By \eqref{eq:to-contradict-Propok-cool}, we have
\be\label{eq:est-u-Hol}
 \|u_\x\|_{L^2(B_1)}= \|v-m(\x')\|_{L^2(B_{\rho}^+(\x'))}\leq C \rho^{N/2+\a}  .
\ee
On the other hand, letting
\begin{align*}
w_\x(y)&:=\left( u(\x+y)-m(\x') \right)1_{\{x_N>-2\}}(y)\\
&= \left( u(\x'+2\rho e_N+y)-m(\x') \right)1_{\{x_N>-2\}}(y),
\end{align*}
and using   \eqref{eq:to-contradict-Propok-cool}, we get   
\begin{align*}
&\int_{|y|\geq1/2\cap \R^N_+}\frac{|u_\x(y)|}{|y|^{N+2s}}\, dy =\rho^{2s}\int_{|y|\geq \rho/2 }|y|^{-N-2s}|w_\x(y)|\, dy\\
&\leq \rho^{2s}\sum_{k=0}^\infty\int_{\rho 2^{k}\geq |y|\geq \rho2^{k-1}}|y|^{-N-2s}|w_\x(y)|\, dy \leq \rho^{2s}\sum_{k=0}^\infty (2^k\rho)^{-N-2s}  \int_{\rho2^{k+1} \geq |y| } |w_\x(y)|\, dy\\
&\leq  \rho^{2s}\sum_{k=0}^\infty (2^k\rho)^{-N-2s} ( \rho2^{k+1})^N\int_{ B_{ 2^{k+2}\rho}^+ }| u(\x'+\z)-m(\x')|\, d\z \\
&\leq  C \rho^{2s}\sum_{k=0}^\infty(2^k\rho)^{-N-2s} (\rho 2^{k})^{N+\a}  \leq C \rho^{\a} \sum_{k=0}^\infty 2^{-k(2s-\a)}\leq C \rho^{\a}.
\end{align*}
Hence, combining this with \eqref{eq:est-u-Hol} and  \eqref{eq:Hold-estim}, we then get 
\begin{align}\label{eq:Grad-estim-Hol}
\|u_\x\|_{C^{\a}(B_{1/2})} &\leq C \rho^{\a}.
\end{align}
Consequently,
$$
| v(\x)|\leq |m(\x')|+ C\rho^{\a}\leq C
$$
and 
$$
 [  v-m(\x') ]_{C^{\a }(B_{\rho /2}(\x))}= [ v ]_{C^{\a}(B_{\rho /2}(\x))}\leq  C.
$$
We can apply Lemma \ref{lem:Prop-of-RS} to deduce that
$$
\|   v\|_{C^{\a}(B_{1/2}^+)}\leq C
$$
and the proof is complete.
\end{proof}
%



\subsection{Higher order regularity estimates}
 
The following result provides the first steps toward the higher order boundary regularity.   
Next, we   define,     $r>0$ and $w\in L^2_{loc}(\R^N)$,  
\be\label{eq:def-Pwzr}
P_{w,r} (x) =w_{B_r^+}+\sum_{i=1}^{N-1}\frac{x_i}{\|y_i\|^2_{L^2(B_r^+)}}{\int_{B_r^+}w(y)  y_i\, dy  }=: w_{B_r^+}+\sum_{i=1}^{N-1}{x_i}p^i_{w}(r) ,
\ee
 where for a measurable subset $A\in \R^N$ and a measurable function $u$, we write $u_A=\frac{1}{|A|}\int_{A}u(y)\, dy$.
We note that $P_{w,r} $ is the   $L^2(B_r^+)$-projection of $x\mapsto w (x)-w _{B_r^+}$ on the finite dimension subspace of $L^2(B_r)$, $\textrm{span}\{x_1,\dots, x_{N-1}\}$.
%
%
%
\begin{proposition} \label{prop:bound-Kato-abstract-HB1}
Let  $s\in (1/2,1)$,   $\b\in (0, 2s-1)$,  $2s-N/p>1 $ and  $\e<\min (2s-N/p-1,\b )$.
  Then there exist $\e_0,C>0$ such that for every $\psi\in C^{1,\b}(\R^{N-1})$ and for every
$ g\in L^p(\R^N)$,  $ v\in H^s(\R^N_+) $  satisfying 
\be\label{eq:frac-half-flat}
\calD_{K^\psi_s}(v,\vp)=\int_{B_2^+}g(x)\vp(x)\, dx \qquad\textrm{ for all $\vp \in C^\infty_c(B_2) $,}
\ee
$$
\| g \|_{L^p(\R^N) } + \|v \|_{ L^2(\R^N_+)  }\leq 1
$$
and 
$$
[ \n  \psi]_{C^{0,\b}(\R^{N-1})}<\frac{1}{4}, \qquad  \psi(0)=|\n\psi(0)|=0  
$$
then we have 
\be \label{eq:sup-g1}
\sup_{r>0} r^{-\frac{N}{2}-\g_1}  \|v -P_{v  ,r}   \|_{L^2(B_r)} \leq  C,
\ee
where   $\g_1:= \min (2s-N/p,1+\b )-\e>1  $.
\end{proposition}
 \begin{proof}
Assume that  the assertion   does not hold, then   for every integer   $n\geq 2$,  there exist $\psi_n\in C^{1,\b}(\R^{N-1}),$    $ g_n\in L^p(\R^N)$,  $ v_n\in H^s(\R^N_+)\cap C(\ov{\R^N_+})$, with 
$$
\| g_n \|_{L^p(\R^N) } + \|v_n \|_{ L^2(\R^N_+)  }\leq 1
$$
satisfying 
\be\label{eq:frac-half-flat-n-Hi}
\calD_{K^{\psi_n}_s}(v_n,\vp)=\int_{B_2^+}g_n(x)\vp(x)\, dx \qquad\textrm{ for all $\vp \in C^\infty_c(B_2) $, }
\ee
while  
$$
\sup_{r>0} r^{-N-\g_1}   \|  v_n  -P_{ v_n  ,r}   \|_{L^2(B_r)}> n
$$
and 
\be\label{eq:smalp-Grad-psi}
[ \n  \psi_n]_{C^{0,\b}(\R^{N-1})}<\frac{1}{4}, \qquad  \psi_n(0)=|\n\psi_n(0)|=0.
\ee
Define 
$$
\Theta_n(\ov r)=\sup_{ r\in [ \ov r,\infty)}   r^{-\frac{N}{2}-\g_1}   \| v_n  -P_{{ v_n},r_n}   \|_{L^2(B_r)}  .
$$
Proceeding as in the proof of Proposition \ref{prop:bound-Kato-abstract}, we can find a sequence $r_n$, tending to zero, such that, for all $n\geq 2$,  
\begin{align*}
 r_n^{-\frac{N}{2}-\g_1}  \| v_n  - P_{ v_n  ,r_n}   \|_{L^2(B_{ r_n})} \geq \frac{1}{2}\Theta_n( r_n)\geq\frac{n}{4}.
\end{align*}
 We now   define the   sequence  of functions
\be \label{eq:def-w-n-Hi}
 {w}_n(x)
  = \frac{r_n^{-\g_1} }{  \Theta_n(r_n) }  \left\{  v_n (r_n x  ) 
  -P_{{ v_n}, r_n}  (r_n x)  \right\}.
\ee
It satisfies 
\be \label{eq:wn-HP}
 \|w_n\|_{L^2(B_{1})} \geq \frac{1}{2} 
\ee
and  for every $n\geq2$, $i=1\dots, N-1$,
\be \label{eq:w-n-nonzero-HP}
 \int_{B_1}w_n(x)\, dx =0 \qquad\textrm{and} \qquad  \int_{B_1}w_n(x) x_i\,dx  =0 .
\ee
%
%
 %
By the definition and monotonicity of $\Theta_n$, we can  use   similar arguments as in \cite{Fall-reg-1,Serra}, to obtain 
\be\label{eq:groht-w-n-abs-HP}
 \|w_n\|_{L^2(B_{R})} \leq  C   R^{N/2+\g_1}  \qquad\textrm{ for every $R\geq 1$ and $n \geq 2$,}
\ee
for some constant $C=C(N,\g_1)>0$.\\
For the following, we put
   $\hat g_n(x)= \frac{r_n^{2s-\g_1}}{\Theta_n(r_n)^{\frac{1}{2}} } g_n(r_n x)$,  $\ti\psi_n(x')=\frac{1}{r_n}\psi(r_n x')$
and 
$$
K_n(x,y)=K^{ \ti\psi_n}_s(x,y),
$$  
  Then letting
  $$
P_n(x):=\frac{r_n^{\b-\g_1}}{\Theta_n(r_n)^{\frac{1}{2}} } P_{ v_n,r_n} (r_nx)=\frac{r_n^{1+\b-\g_1}}{\Theta_n(r_n)^{\frac{1}{2}} } \sum_{i=1}^{N-1} p^i_{v_n}(r_n) x_i,
$$
by a change of variable, we thus get 
$$
\calD_{K_n}(w_n,\vp)-r_n^{-\b}\calD_{K_n}(P_n,\vp)= \int_{\R^N_+} \hat g_n(x)\vp(x)\, dx.
$$
Since $\Ds_{\ov{ \R^{N}_+}} x_i=0$ on $\R^N_+$, letting $K'_n(x,y)=r_n^{-\b}(K_n(x,y)-K^+_s(x,y))$, we have 
\be\label{eq:frac-half-flat-n-scal-cut-Hi-cc}
\calD_{K_n}(w_n,\vp)-\calD_{K_n'}(P_n,\vp)= \int_{\R^N_+} \hat g_n(x)\vp(x)\, dx.
\ee
We fix $M>1$ and let $n\geq 2$ large so that $1<M<\frac{1}{2r_n}$.  Let $\chi_M\in C^\infty_c(B_{4M})$ be such that $\chi_M=1$ on $B_{2M}$ and define $W_n:=\chi_M w_n$. Then by  letting  $\vp=\chi_{M/4}\in C^\infty_c(B_{M})$,  we find that
\be\label{eq:frac-half-flat-n-scal-cut-Hiaa}
\calD_{K^{\psi_n}_s}(W_n,\vp)-\calD_{K'_n}(P_n,\vp)= \int_{\R^N_+}\ti g_n(x)\vp(x)\, dx
\ee
with 
$$
\|\ti g_n\|_{L^p(\R^N)}\leq \|\hat g_n\|_{L^p(\R^N)} + C(M) .
$$
From \eqref{eq:mu-control-n-psi},  \eqref{eq:remainder-Kernel-flat} and \eqref{eq:smalp-Grad-psi}, we then obtain, for all $x,y\in \R^N_+$,
\be \label{eq:estK-prime-n}
|K'_n(x,y)|=r_n^{-\b}\left | K_{n}(x,y)- K_n^+(x,y) \right|\leq c_{N,s}  C(N)   \frac{ r_n^{-\b} \min(| r_n x'|^\b+| r_ny'|^\b, 1)}{|x-y|^{N+2s}}.
\ee
By Corollary \ref{eq:cor-reg}, we have that $\|v_n\|_{C^{1-\e}(B_1^+)}\leq C_\e$,  for every $\e\in (0,1)$.   
Recalling \eqref{eq:def-Pwzr},  we  thus get 
\be \label{eq:est-pir-n}
|p^i_{v_n}(r_n) | \leq C_\e r_n^{-\e}.
\ee
From this, we have 
\be\label{eq:Pn1}
\|P_n\|_{H^s(B_R)}\leq \frac{C(R)}{\Theta_n(r_n)   } \qquad \textrm{for all $R>0$ }
\ee
and, since $\b<2s-1$, 
\be\label{eq:Pn2}
\|P_n\|_{L^1(\R^N_+;(1+|x|)^{-N-2s+\b})}\leq  \frac{C}{\Theta_n(r_n)  } . 
\ee
Using  the argument in the proof of Proposition \ref{prop:bound-Kato-abstract} based on    Lemma \ref{lem:non-loc-caciop} and the fractional Sobolev inequality, we obtain   a constant positive constant $C(M)$ only depending on $M$ and $N$ such that 
\begin{align}\label{eq:Bond-Sobl-Hi}
C(M)\int_{\R^{2N}}&(W_n(x)-W_n(y))^2\vp^2(y) K_n(x,y)\,dy dx\leq 1 \nonumber\\
&+\left|  \int_{\R^{2N}} (P_n(x)-P_n(y))(W_n(x)\vp(x)-W_n(y)\vp(y))K_{n}'(x,y)\, dydx     \right|.
\end{align}
 We now use \eqref{eq:estK-prime-n}, \eqref{eq:est-pir-n},  \eqref{eq:Pn1} and \eqref{eq:Pn2} to estimate
 \begin{align*}
&\left|  \int_{\R^{2N}} (P_n(x)-P_n(y))(W_n(x)\vp(x)-W_n(y)\vp(y))K_{n}'(x,y)\, dydx     \right|\\
&\leq \left|   \int_{B_{2M}\times B_{2M}} (P_n(x)-P_n(y))(W_n(x)\vp(x)-W_n(y)\vp(y))K_{n}'(x,y)\, dydx     \right|\\
&+2 \int_{  B_{M}} |W_n(x)\vp(x)|\int_{\R^N\setminus B_{2M}} |P_n(x)-P_n(y)| |K_{n}'(x,y)|\, dydx     \\
&\leq  C(M)\left((1-s) \e [W_n\vp]_{H^s(B_{2M}^+)}^2+ \e^{-1}  [P_n]_{H^s(B_{2M})}^2+   \int_{ \R^N\setminus B_{2M}}\frac{|P_n(y)|}{|y|^{N+2s-\b}}\, dy+1 \right)\\
&\leq C(M) \left( (1-s)  \e [W_n\vp]_{H^s(B_{2M}^+)}^2+ 1\right).
\end{align*}
Combining this with \eqref{eq:Bond-Sobl-Hi}, we then deduce that 
$$
(1-s)[w_n]_{H^s(B_{M/4}^+)}^2\leq C(M).
$$
Thanks to \eqref{eq:local-bnd-w_n}, we have that $w_n$ is bounded in $H^{s}_{loc}(\ov{\R^N_+})$.  Hence by a diagonal argument, up to a subsequence, there exists $ w\in H^{ s}_{loc}( \ov{\R^N_+})$ such that 
$$
w_n\to  w \qquad\textrm{ in $L^2_{loc}(\ov{\R^N})$} 
$$
and, for all $\s< s$, 
$$
w_n \to  w \qquad\textrm{ in $H^{\s}_{loc}(\ov{\R^N})$} .
$$
Moreover   $w_n\to w $ in $  L^1(\R^N_+; (1+|x|^{N+2 s})^{-1})$,  thanks to \eqref{eq:groht-w-n-abs-HP}.  
  Passing to the limit in \eqref{eq:groht-w-n-abs-HP}, we obtain $\|w\|_{L^2(B_R^+)}\leq C R^{\frac{N}{2}+\g_1}$ for all $R\geq1$.
Next as above for all $\phi\in C^\infty_c(B_M)$, by \eqref{eq:estK-prime-n}, \eqref{eq:Pn1} and \eqref{eq:Pn2}, we can  estimate
\begin{align*}
&\left|\int_{\R^{2N}} (P_n(x)-P_n(y))(\phi(x)-\phi(y))K_{n}'(x,y)\, dydx \right|  \nonumber\\
&\leq   {C(N)}{ }\left( [P_n]_{H^s(B_{2M})} [\phi]_{H^s(B_{2M})}+ \|\phi\|_{L^\infty(B_{M})}\left(\|P_n\|_{L^1(B_{2M})}+ \|P_n\|_{L^1(\R^N_+;(1+|x|)^{-N-2s+\b})}\right) \right) \\
&  \leq   \frac{C(N)  }{\Theta_n(r_n) } \|\phi\|_{C^{0,1}(\R^N)}.
\end{align*}
It is also easy to see that $\|\hat g_n\|_{L^p(\R^N)}\leq  \frac{1  }{\Theta_n(r_n) }$.
Therefore, by \eqref{eq:estK-prime-n},  passing to the limit in \eqref{eq:frac-half-flat-n-scal-cut-Hi-cc},  we find that   $(-\D)^{s}_{\ov{\R^N_+}} w=0$ on $\R^N_+$.    We  then deduce from Theorem \ref{th:Liouville} that $w(x',x_N)=a+c\cdot x'$, so that we reach a contradiction by passing to the limit in  \eqref{eq:wn-HP} and \eqref{eq:w-n-nonzero-HP}. 
\end{proof}

\begin{corollary}\label{cor:near-grad-estim}
Let  $s\in (1/2,1)$,   $\b<2s-1$ and   $2s-N/p>1 $. Put $\g=\min(2s-N/p, 1+\b)$.  Consider   $\psi\in C^{1,\b}(\R^{N-1})$, with $\|\psi\|_{ C^{1,\b}(\R^{N-1})}\leq \frac{1}{4}$ and   $\psi(0)=|\n\psi(0)|=0$. Let
$ g\in L^p(\R^N)$ and  $ v\in H^s(\R^N_+)\cap C(\ov{\R^N_+})$  satisfying 
$$
\calD_{K^\psi_s}(v,\vp)=\int_{B_2^+}g(x)\vp(x)\, dx \qquad\textrm{ for all $\vp \in C^\infty_c(B_2) $,}
$$
$$
\| g \|_{L^p(\R^N) } + \|v \|_{ L^2(\R^N_+)  }\leq 1.
$$
Then there exist $\ov C>0$ depending only on $N,s,p$ and $\b$  and  a  constant  $D\in \R^{N-1}$ satisfying 
\be\label{eq:bnd-D}  
|D|\leq \ov C,
\ee
and 
$$
 \| v - v(0)   -D\cdot x'  \|_{L^2(B_r^+)} \leq \ov C r^{\frac{N}{2}+\g}.
$$
\end{corollary}
\begin{proof}
 In view of   \eqref{eq:sup-g1},  using e.g. the argument in \cite{Fall-reg-1,Serra,Campanato},   we can find  a constant $ C$ depending only on $N,s_0,\b,p,N$ and $\e$  and  some constant  $D\in \R^{N-1}$, $d\in \R$ such that
$$
|D|+ |d|\leq  C,
$$
and for every $r>0$  
$$
 \| v - d  -D\cdot x'  \|_{L^2(B_r^+)} \leq  C r^{\min(2s-N/p, 1+\b)-\e}.
$$
By continuity of $v$ on $\ov{\R^N_+}$, thanks to Corollary \ref{eq:cor-reg}, we deduce that $v(0)=d$. 
From the above result we also deduce that  $ \frac{1}{|B_r^+|}\| v- v (0)     \|_{L^2(B_r^+)}\leq C r$. Therefore, we repeat the argument in Proposition \ref{prop:bound-Kato-abstract-HB1} with $\g_1=\min(2s-N/p, 1+\b)$ to deduce that
$$
\| v- v (0)  -D\cdot x'  \|_{L^2(B_r^+)}  \leq \ov C r^{N/2+\min(2s-N/p, 1+\b)}.
$$
Note that the reason for considering $\g_1:=\min(2s-N/p, 1+\b)-\e$ was just because  we applied Corollary \ref{eq:cor-reg} which implied \eqref{eq:est-pir-n}.
\end{proof}

Next, we prove the following crucial result.
\begin{corollary} \label{cor:near-grad-estim-Hold}
Under the hypothesis of Corollary \ref{cor:near-grad-estim},  we have 
\be\label{eq:bbbr1}
\sup_{\x\in B_{1}^+\cap \R  e_N} |\n    v(\x)| + \sup_{\x\in B_{1}^+\cap \R  e_N} [\n   v]_{C^{\g-1}(B_{\x_N/2}(\xi))}\leq  C.
\ee
Moreover  $x_N\mapsto \de_{x_N}v(0,x_N)$ is continuous  at $x_N=0$ and 
\be\label{eq:bbbr2}
\de_{x_N}v(0)=0,
\ee
with  the constant $C$ depending only on  $N,s,p$ and $\b$.
\end{corollary}
\begin{proof}
Let $\x:=(0,2\rho) \in B_{1}^+\cap \R  e_N$.  We define  
$$
 u_\x(y)=  v(\x+\rho y)- v (0) -D\cdot (\rho y') = v(2\rho e_N+\rho y)- v (0) -D\cdot (2\rho e_N+\rho y)' 
 $$
  and $\ti g(y)=g(\x+\rho y)$ and $U(y):=\rho^{1+\b} D\cdot y'$.   We have, for all $\vp\in C^\infty_c(B_1)$, 
\begin{align*}
\calD_{K_{\rho,s}} (u_\x ,\vp)+\rho^{-\b}\calD_{K_{\rho,s}} (U,\vp)&= \rho^{2s} \int_{\Sig} \ti g(y)  \vp(y)\, dy,   
\end{align*}
where $K_{\rho,s}(x,y)=1_{\Sig}(x)1_{\Sig}(y)\frac{c_{N,s}}{|\Psi_\rho(x)-\Psi_\rho(y)|^{N+2s}}$, $\Sig:=\{x_N>-2\} $ and    $\Psi_\rho( x)=\frac{1}{\rho}\Psi(\rho x+\x)$.    Since, by a change of variable, 
$$
p.v. \int_{\R^N}\frac{U(x)-U(y)}{|\Psi_\rho(x)-\Psi_\rho(y)|^{N+2s}}\, dy=0 \qquad \textrm{ for all $x\in \R^N$,}
$$
letting $F_\x(x):=c_{N,s}\rho^{-\b}\int_{\R^N\setminus \Sig }\frac{U(x)-U(y)}{|\Psi_\rho(x)-\Psi_\rho(y)|^{N+2s}}\, dx$, we then get 
\begin{align}\label{eq:DKrho-su-xi}
\calD_{K_{\rho,s}} (u_\x ,\vp) &=  \int_{B_1}\left(\rho^{2s} \ti g(y)  -F_\x(y) \right)\vp(y) \, dy,   
\end{align}
To estimate $F_\x$ we observe that $\Ds_{\ov { \R^N\setminus \Sig }} x_i=0$ on $\R^N\setminus \Sig$,   for $i=1,\dots,N-1$. Consequently, by \eqref{eq:remainder-Kernel-flat}, for $x\in B_1$, we have 
\begin{align*}
| F_\x(x)|&=c_{N,s}\left|  \int_{\R^N\setminus \Sig }{\left(U(x)-U(y)\right)} \rho^{-\b}\left( \frac{1}{|\Psi_\rho(x)-\Psi_\rho(y)|^{N+2s}}-\frac{1}{|x-y|^{N+2s}}\right)\, dy \right|\\
&\leq C(N) c_{N,s}  \int_{\R^N\setminus \Sig }{|U(x)-U(y)|}  \frac{ \rho^{-\b}\int_0^1|\n \psi(\rho x'+t(\rho y'-\rho x'))|\, dt }{|x-y|^{N+2s}} \, dy \\
&\leq C(N){ c_{N,s} }[U]_{C^{0,1}(\R^N)} \int_{\R^N\setminus \Sig }  \frac{  |x|^\b+|y|^\b}{|x-y|^{N+2s-1}} \, dy\leq C(N)\frac{ c_{N,s} }{2s-1+\b} \rho^{1+\b}|D|.
\end{align*}
By   { \cite[Theorem 1.3]{Fall-reg-2}},  applied to \eqref{eq:DKrho-su-xi}, and using \eqref{eq:bnd-D},  we have 
\begin{align}\label{eq:Grad-estim}
\|\n u_\x \|_{C^{\g-1 }(B_{1/2})} &\leq C\left( \|u_\x\|_{L^2(B_1)}+\int_{|y|\geq1/2}\frac{|u_\x(y)|}{|y|^{N+2s}}\, dy+\rho^{2s} \|\ti g\|_{L^p(\R^N)}  +\rho^{1+\b} \right) \nonumber\\
& \leq C\left(  \|u_\x\|_{L^2(B_1)}+\int_{|y|\geq1/2\cap \R^N_+}\frac{|u_\x(y)|}{|y|^{N+2s}}\, dy + \rho^{\g}  \right).
\end{align}
%
By Corollary \ref{cor:near-grad-estim}, we have
\be\label{eq:est-u_x}
 \|u_\x\|_{L^2(B_1)}\leq C \rho^{\g}  ,
\ee
where $\g:=\min(2s-N/p, 1+\b)$. We let
\begin{align*}
w_\x(y)&:=u_\x(y/\rho)=\left( v(\x+y)-v(0)-D\cdot y' \right)1_{\Sig}(y)\\
&= \left( v(2\rho e_N+y)-v(0)-D\cdot(2\rho e_N+ y) \right)1_{\Sig}(y).
\end{align*}
Then, using Corollary \ref{cor:near-grad-estim}, we get   
\begin{align*}
&\int_{|y|\geq1/2\cap \Sig}\frac{|u_\x(y)|}{|y|^{N+2s}}\, dy =\rho^{2s}\int_{|y|\geq \rho/2 }|y|^{-N-2s}|w_\x(y)|\, dy\\
&\leq \rho^{2s}\sum_{k=0}^\infty\int_{\rho 2^{k}\geq |y|\geq \rho2^{k-1}}|y|^{-N-2s}|w_\x(y)|\, dy \leq \rho^{2s}\sum_{k=0}^\infty (2^k\rho)^{-N-2s}  \int_{\rho2^{k+1} \geq |y| } |w_\x(y)|\, dy\\
&\leq  \rho^{2s}\sum_{k=0}^\infty (2^k\rho)^{-N-2s}  \| v-v(0)-D\cdot \z\|_{L^1(B_{ \rho 2^{k+2}}^+ )} \\
&\leq  C \rho^{2s}\sum_{k=0}^\infty(2^k\rho)^{-N-2s} (\rho 2^{k})^{N+\g}  \leq C \rho^{\g} \sum_{k=0}^\infty 2^{-k(2s-\g)}\leq C \rho^{\g}.
\end{align*}
Hence, combining this with \eqref{eq:est-u_x} and  \eqref{eq:Grad-estim}, we then get 
\begin{align}\label{eq:Grad-estim}
\|\n u_\x \|_{C^{\g-1}(B_{1/2})} &\leq C \rho^{\g}.
\end{align}
Scaling back, we get, for all $\xi \in B_{1}^+\cap \R e_N$, 
$$
 | \n v(\x)|+   [\n v ]_{C^{\g-1}(B_{\rho }(\x) )}\leq C .
$$
This proves \eqref{eq:bbbr1}. Note that the above estimate also implies that $\|\n v \|_{C^{ \g-1 }(B_{1/2}^+\cap \R e_N)}\leq C,$ thanks to Lemma \ref{lem:Prop-of-RS} and that
$$
|\de_{x_N}v(\x)|\leq C \rho^{\g-1}.
$$
Hence, letting $\rho=\x_N/2\to 0$, we obtain $\de_{x_N}v(0)=0$, which is  \eqref{eq:bbbr2}.  
\end{proof}

\section{The Dirichlet problem}\label{s:Reg-Dir}
As mentioned earlier, thanks to the fractional Hardy  inequality with boundary singularity $H^s_0(\O)\subset L^2(\O;\d^{-2s}(x))$, for $s\not=1/2$, for every Lipschitz open set $\O$, see e.g. \cite{Grisv}. As a consequence,  we may identify the space $H^s_0(\O)$ with  the Sobolev space
 $$
 \calH^s_s(\O):=\{u\in H^s(\R^N)\,:\, u=0\qquad\textrm{in $\R^N\setminus\O$}\}.
 $$
 
See e.g. \cite{BFV,BBC,CK,Huyuan} we have the following result.
\begin{lemma}
Let $s\in (1/2,1)$ and  $u\in \calH^s_0(\R^N_+) $ and $g\in L^\infty(\R^N)$ such that 
 \be\label{eq:first-eq-Dir-half}
\begin{cases}
\Ds_{ \R^N_+} u=g& \qquad\textrm{ in $B_2^+$,}\\
u=0 & \qquad\textrm{ in $ B_2^-$.}
\end{cases} 
\ee
Then 
$$
\|u\|_{C^{2s-1}(B_1^+)}\leq C\left(  \|u\|_{ L^2(B_{2})}+ \int_{|y|\geq \frac{1}{4}}\frac{|u(y)|}{|y|^{N+2s}}\,dy+  \|g\|_{ L^\infty(\R^N)}  \right).
$$
\end{lemma}
As a consequence of this result we prove the following result thanks to Proposition \ref{prop:classif-1D-2s-not-1}-$(ii)$.
\begin{lemma}\label{lem:Liouville-Dirichlet}
Let $u\in H^s_{loc}(\R^N) $  be  such that 
 \be\label{eq:first-eq-Dir-fin}
\begin{cases}
\Ds_{ \R^N_+} u=0& \qquad\textrm{ in $\R^N_+$,}\\
u=0 & \qquad\textrm{ on $\R^N\setminus \R^N_+$}
\end{cases} 
\ee
and,  for some constants $C>0$ and $\e\in (0,2s)$,
\be\label{eq:growth-HBR}
\|u\| _{L^2(B_R)}\leq C R^{\frac{N}{2}+\e} \qquad\textrm{ for all $R\geq1$.}
\ee
Then   $u(x',x_N)=a (x_N)_+^{2s-1}$ for some constant $a\in\R$.
\end{lemma} 
\begin{proof}
Using the same argument as in the proof of Proposition \ref{prop:high-tangderiv}, we find that 
$$
	\|\n_{x'} ^2 u\|_{L^\infty(B_1^+)}\leq C\left(  \|u\|_{L^2(B_2^+)}+ \int_{|y|\geq 1/2}|u(y)| |y|^{-N-2s}\, dy \right).
$$
Hence letting $u_R(x)=u(R x)$ which  is a  solution to  \eqref{eq:first-eq-Dir-fin}, we then get 
\begin{align*}
\|\n_{x'} ^2 u_R\|_{L^\infty(B_1^+)}&\leq C\left(  \|u_R\|_{L^2(B_2^+)}+ \int_{|y|\geq 1/2}|u_R(y)| |y|^{-N-2s}\, dy \right)\\
&\leq R^{\e-2}+  R^{2s} \int_{|y|\geq 1/2R}|u(y)| |y|^{-N-2s}.
\end{align*}
From this we deduce that 
$$
\|\n_{x'}^2  u\|_{L^\infty(B_R^+)}\leq C R^{\e-2}.
$$
Hence, letting $R\to \infty$, we deduce that $u(x',x_N)=a(x_N)+b(x_N)\cdot x'$. Since $x\mapsto u(x+h)-u(x)$ solves  \eqref{eq:first-eq-Dir-fin} for all $h\in \de\R^N_+$, we deduce that the function $ b_i$  solves  \eqref{eq:first-eq-Dir-fin}  for $i=1,\dots,N-1$. But then Proposition \ref{prop:classif-1D-2s-not-1}-$(ii)$ implies that $b_i(x_N)=c_i (x_N)^{2s-1}_+$, for some constant $c_i$, so that \eqref{eq:growth-HBR} implies that $c_i=0$. It then follows that $a$ solves \eqref{eq:first-eq-Dir-fin}  and thus  Proposition \ref{prop:classif-1D-2s-not-1}-$(ii)$ yields the desired result.
\end{proof}
Next, we state the following result for which its proof is based on a blow up argument   as above and the classification result in Lemma \ref{lem:Liouville-Dirichlet}. Since the proof is very similar   to the one of Proposition \ref{prop:bound-Kato-abstract},  we omit it.
 \begin{proposition} \label{prop:bound-Kato-abstract-Dir}
Let  $s_0\in (1/2,1)$, $\e\in (0,2s_0-1)$,  $p>\frac{N}{2s_0 }$.  Put $\g:=\min(2s-N/p,2s-1-\e)$.   Then there exist $\e_0,C>0$ such that for every $s\in [s_0,1)$, $\psi\in C^{1}(\R^{N-1})$ and for every
$ g\in L^p(\R^N)$,  $ w\in \calH^s_0(\R^N_+) $  satisfying 
$$
 [  \psi]_{C^{1}(\R^{N-1})}<\e_0 
$$
and 
$$
\calD_{K^{\psi}_s}(w,\vp)=\int_{B_2^+}g(x)\vp(x)\, dx \qquad\textrm{ for all $\vp \in C^\infty_c(B_2^+) $, }
$$
we have 
\be \label{eq:bdr-estim-Kato}
\sup_{r>0} r^{-2\g-N}\sup_{z\in B_1'}  \| w\|_{L^2(B_r(z))}^2 \leq  C    (  \|w\|_{  L^2(\R^N) } +  \|g\|_{ L^p(\R^N) }  )^2.
\ee 
\end{proposition}
As a consequence we have the following result.
\begin{corollary}\label{cor:C-gam-reg-Dir}
Let  $s_0\in (1/2,1)$, $\e\in (0,2s_0-1)$,  $p>\frac{N}{2s_0 }$ and $s\in[s_0,1)$. We put $\g:=\min(2s-N/p,2s-1-\e)$. Let $\psi\in C^{1}(\R^{N-1})$ 
$ g\in L^p(\R^N)$,  $ w\in \calH^s_0(\R^N_+) $  satisfying 
and 
\be\label{eq:frac-half-flat}
\calD_{K^{\psi}_s}(w,\vp)=\int_{B_2^+}g(x)\vp(x)\, dx \qquad\textrm{ for all $\vp \in C^\infty_c(B_2^+) $. }
\ee
Then  there exists $\ov C,\e_0$ depending only on $N,s_0,p$ and $\e$  such that, provided $ [  \psi]_{C^{1}(\R^{N-1})}<\e_0$,  
$$
\|w\|_{C^\g({B_{1/2})}}\leq  C    (  \|w\|_{  L^2(B_2) }+ \|w\|_{L^1(\R^N; \frac{1}{1+|x|^{N+2s}})} +  \|g\|_{ L^p(\R^N) }  ).
$$
\end{corollary}
\begin{proof}
Let $\chi\in C^\infty_c(B_4)$ with $\chi\equiv 1$ on $B_{1/2}$. Then we have that $u=\chi_{B_{4}}w$ solves 
$$
\calD_{K^{\psi}_s}(u,\vp)=\int_{B_2^+}f(x)\vp(x)\, dx\qquad\textrm{ for all  $\vp\in C^1_c(B_{1/2})$,} 
$$
for some function $f\in L^p(\R^N)$ satisfying $\|f\|_{L^p(\R^N)}\leq C\left(  \|w\|_{L^1(\R^N; \frac{1}{1+|x|^{N+2s}})}  +  \|g\|_{ L^p(\R^N) } \right) $.
Let  $x_0=(x_0',x_0\cdot e_N)\in  B_1^+$  and define $\rho=\frac{x_0\cdot e_N}{2}$, so that $B_\rho(x_0)\subset B_{3\rho}(x_0')\cap  \R^N_+. $
Next, we define $v_\rho(x):=u(\rho x+ x_0)$  and $f_\rho(x):= \rho^{2s} f (\rho x+ x_0)$. It is plain that $v\in H^s_{loc}(B_2)\cap L^1(\R^N; \frac{1}{1+|x|^{N+2s}})$ and
$$
\calD_{K_{\rho,s}}(v_\rho ,\vp)=\int_{B_2^+}f_\rho(x)\vp(x)\, dx\qquad\textrm{ for all  $\vp\in C^1_c(B_{1})$,} 
$$
where $K_{\rho,s}(x,y)=1_{\{x_N>-2\}}(x)1_{\{y_N>-2\}}(y)\frac{c_{N,s}}{|\Psi_\rho(x)-\Psi_\rho(y)|^{N+2s}}$ and   $\Psi_\rho( x)=\frac{1}{\rho}\Psi(\rho x+x_0)$.
Note that $  \|  f_\rho \|_{L^p(\R^N)}\leq \rho^{2s-N/p} \|f\|_{L^p(\R^N)} $.   By the interior regularity, see e.g.  \cite[Theorem 4.1]{Fall-reg-1}.
\be \label{eq:estime-v-scale}
\|v_\rho\|_{C^{\min(1-\e,2s-N/p)}(B_{1/2})}\leq C\left(   \|v_\rho\|_{L^2(B_1)}+ \|v_\rho\|_{L^1(\R^N; \frac{1}{1+|x|^{N+2s}})} +  \|f_\rho\|_{L^p(\R^N)}  \right)  .
\ee
By  \eqref{eq:bdr-estim-Kato}  and H\"older's inequality,  we have
\begin{align}\label{eq:norm-vL2-cL1s-rho}
\|v_\rho\|_{L^2(B_1)}\leq \rho^{-N/2} \|u\|_{L^2(B_{3\rho}(x_0'))}\leq C \rho^\g \quad \textrm{and } \quad  \|u\|_{L^1(B_{r}(x_0'))}\leq C r^{N+\g}\quad
  \textrm{for  $r>0$.}
\end{align}
Using   the second estimate in \eqref{eq:norm-vL2-cL1s-rho}, we get
\begin{align*}
&\int_{|x|\geq 1}|x|^{-N-2s}|v_\rho(x)|\, dx=\rho^{2s}\int_{|y-x_0|\geq \rho }|y-x_0|^{-N-2s}|u(y)|\, dy\\
&\leq \rho^{2s}\sum_{k=0}^\infty\int_{\rho 2^{k+1}\geq |y-x_0|\geq \rho2^k}|y-x_0|^{-N-2s}|u(y)|\, dy\\
&\leq \rho^{2s}\sum_{k=0}^\infty (2^k\rho)^{-N-2s}  \int_{\rho2^{k+1} \geq |y-x_0| } |u(y)|\, dy \leq  \rho^{2s}\sum_{k=0}^\infty (2^k\rho)^{-N-2s}\|u\|_{L^1(B_{\rho 2^{k+3} }(x_0'))}  \\
&\leq  C \rho^{2s}\sum_{k=0}^\infty(2^k\rho)^{-N-2s} (\rho 2^{k})^{N+\g}  \leq C \rho^{\g} \sum_{k=0}^\infty 2^{-k(2s-\g)}\leq C \rho^{\g}.
\end{align*}
We then conclude that $\|v_\rho\|_{L^1(\R^N; \frac{1}{1+|x|^{N+2s}})} \leq C \rho^{\g}$. It follows from   \eqref{eq:norm-vL2-cL1s-rho} and  \eqref{eq:estime-v-scale}, that
 $$
  \|v_\rho\|_{C^\g(B_{1/2})}\leq  C (\rho^{\g}+  \rho^{2s-N/p} ).
  $$
Scaling back and using Lemma \ref{lem:Prop-of-RS}, we get
 $$
|u(x_0)|+  \|u\|_{C^\g(B_{\rho/2}(x_0))}\leq C ,
  $$
as  claimed.
\end{proof}
The next result provided the first steps toward the higher order boundary regularity for the regional fractional Dirichlet problem. We first  define for $w\in L^2_{loc}(\R^N)$   and $r>0$, 
\be\label{eq:defQ-w-z-r}
Q_{w,r} (x) = \frac{(x_N)_+^{2s-1}}{ \int_{ B_r}(y_N)_+^{2(2s-1)}\, dy     }{\int_{B_r} w(y) (y_N)_+^{2s-1}\, dy  } =:(x_N)^{2s-1}_+q_{w}(r).
\ee
We have the following result which is crucial to deduce higher order regularity of $v/\d^{2s-1}$ up to  the boundary.
 \begin{proposition} \label{prop:bound-Kato-abstract-HB}
Let  $s\in (1/2,1),$   $p>{N}$ and $\b\in (0,1)$.    Then there exist $\e_0,C>0$ such that for every  $\psi\in C^{1,\b}(\R^{N-1})$ and for every
$ g\in L^p(\R^N)$,  $ v\in \calH^s_0(\R^N_+ )$  satisfying 
$$
\calD_{K^{\psi}_s}(v,\vp)=\int_{B_2^+}g(x)\vp(x)\, dx \qquad\textrm{ for all $\vp \in C^\infty_c(B_2^+) $,}
$$
$$
\| g \|_{L^p(\R^N) } + \|v \|_{ L^2(\R^N)  }\leq 1
$$
and 
$$
\| \n  \psi\|_{C^{0,\b}(\R^{N-1})}<\frac{1}{4}, \qquad \psi(0)=|\psi(0)|=0,
$$
then we have 
\be\label{eq:EEEt} 
\sup_{r>0} r^{-\frac{N}{2}- \g}  \| v - Q_{ v,r}  \|_{L^2(B_r)} \leq  C  ,
\ee
where $\g:=\min ( 2s-N/p, \b+2s-1)-\e>2s-1$.
%

%
%
%
%
\end{proposition}
 \begin{proof}
 %
Assume that  the assertion   does not hold, then   for every   $n\in \N$, $\psi_n\in C^{1,\b}(\R^{N-1})$, $ g_n\in L^p(\R^N)$,  $ v_n\in \calH^s_0(\R^N_+) $, with 
$$
\| g_n \|_{L^p(\R^N) } + \|v_n \|_{ L^2(\R^N)  }\leq 1
$$
satisfying 
\be\label{eq:frac-half-flat-n}
\calD_{K^{\psi_n}_s}(v_n,\vp)=\int_{B_2^+}g_n(x)\vp(x)\, dx \qquad\textrm{ for all $\vp \in C^\infty_c(B_2^+) $, }
\ee
while  
$$
\sup_{r>0} r^{-\frac{N}{2}-\g}  \|  v_n -Q_{  v_n,r}  \|_{L^2(B_r)}> n
$$
and 
\be \label{eq:psi-n-small}
\| \n  \psi_n\|_{C^{0,\b}(\R^{N-1})}<\frac{1}{4}, \qquad \psi_n(0)=|\psi_n(0)|=0,
\ee
Consequently,  there exists $\ov r_n>0$   such that 
$$
\ov r_n^{-\frac{N}{2}-\g}  \|v_n- Q_{v_n ,\ov r_n}  \|_{L^2(B_{\ov r_n})} > n/2  .
$$
We consider the (well defined, because $\| \ov v_n\|_{ L^2(\R^N)} \leq 1$) nonincreasing function $\Theta_n: (0,\infty)\to [0,\infty)$ given by
$$
\Theta_n(\ov r)=\sup_{ r\in [ \ov r,\infty)}   r^{-\frac{N}{2}-\g}  \|  v_n -Q_{  v_n,r}  \|_{L^2(B_r)}  .
$$
Proceeding as in the proof of Proposition \ref{prop:bound-Kato-abstract}, we can find a sequence $r_n$ tending to zero such that 
  for $n\geq 2$,  
\be\label{eq:The-n-geq-n-H}
  \Theta_n( r_n)> n/4 .
\ee
Hence, provided $n\geq 2$,  there exists $r_n\in [\ov r_n,\infty)$ such that 
\begin{align*}
 r_n^{-\frac{N}{2}-\g}  \| v_n- Q_{  v_n, r_n}  \|_{L^2(B_{ r_n})} \geq \frac{1}{2}\Theta_n( r_n)\geq\frac{n}{8}.
\end{align*}
 We now   define the   sequence  of functions
$$
 {w}_n(x)=   \frac{r_n^{-\frac{N}{2}-\g} }{  \Theta_n(r_n) }  \left\{  v_n(r_n x)-Q_{ v_n ,r_n} (r_n x)   \right\}.
$$
It satisfies 
\be \label{eq:w-n-nonzero-H}
 \|w\|_{L^2(B_{1})} \geq \frac{1}{2} ,   \qquad\textrm{and} \qquad \int_{B_1}w_n(x) (x_N)_+^{2s-1}\,dx =0 \qquad\textrm{ for every $n\geq2$.}
\ee
%
 %
Using similar arguments as in \cite{Fall-reg-1,RS16a},   we find that 
\be\label{eq:groht-w-n-abs-HB-Dir}
 \|w_n\|_{L^2(B_{R})} \leq  C   R^{\frac{N}{2}+\g}  \qquad\textrm{ for every $R\geq 1$ and $n \geq 2$,}
\ee
for some constant $C=C(N,\g)>0$.\\
We define
$$
  K_n(x,y)=\frac{r_n^{N+2s}}{2} K^{\psi_n}_{s}(r_nx ,r_ny )=\frac{1}{2} K^{\ti\psi_n}_{s}(x,y),
$$
where $\ti\psi_n(x):=\frac{1}{r_n}\psi_n(r_n x) $.
Let $\hat g_n(x)= \frac{r_n^{2s-\g}}{\Theta_n(r_n) }  g_n (r_n x)$ and 
$$
Q_n(x)= \frac{r_n^{\b-\g}  }{\Theta_n(r_n) }  Q_{ v_n ,r_n} (r_n x)=  \frac{r_n^{\b+2s-1-\g} q_{ {v_n}}( r_n)}{\Theta_n(r_n) } (x_N)_+^{2s-1}  .
$$ 
By a change of variable, we get
\be\label{eq:estim-ov-g-n-Dir}
\|\hat g_n\|_{L^p(\R^N)}\leq   \frac{1}{\Theta_n(r_n) } .
\ee
Moreover,
for all $\vp \in C^\infty_c(B_\frac{1}{2r_n}^+(0))$ , 
\begin{align*}
\calD_{K_n}(w_n,\vp)+ r_n^{-\b}\calD_{K_n}(Q_n,\vp)=\int_{\R^N_+}\ov g_n(x) \vp(x)\, dx.
%
%
\end{align*}
We define $ K'_n(x,y):= r_n^{-\b}\left(K_{n}(x,y)- K_n^+(x,y) \right)$, so that
\be \label{eq:K-prime-small-Dir}
|K'_n(x,y)|=r_n^{-\b}\left | K_{n}(x,y)- K_n^+(x,y) \right|\leq c_{N,s}  C(N)   \frac{  r_n^{-\b} \min(| r_n x'|^\b+| r_ny'|^\b, 1)}{|x-y|^{N+2s}}.
\ee
Using that $\Ds_{{\R^N_+}}x^{2s-1}_N=0$ on $\R^N_+$, we thus obtain
\begin{align} \label{eq:frac-half-flat-n-scal-H-Dir}
\calD_{K_n}(w_n,\vp)+ \calD_{K_n'}(Q_n,\vp)=\int_{\R^N_+}\ov g_n(x) \vp(x)\, dx.
%
%
\end{align}
By Corollary \ref{cor:C-gam-reg-Dir},  for every $\e>0$, there exists $C_\e>0$ such that 
\be \label{eq:esti-d-epsil}
|q_{ { v_n}}( r_n)|\leq C_\e r_n^{-\e}.
\ee
From this, we get 
\be\label{eq:Qn1}
\|Q_n\|_{H^s(B_R)} \leq \frac{C(R)}{\Theta_n(r_n)   } \qquad \textrm{for all $R>0$ }
\ee
and, since $\b<1$, 
\be\label{eq:Qn2}
\|Q_n\|_{L^1(\R^N_+;(1+|x|)^{-N-2s+\b})}\leq  \frac{C}{\Theta_n(r_n)  }  
\ee
 Moreover, as in the proof of Proposition \ref{prop:bound-Kato-abstract-HB1} based on Lemma \ref{lem:non-loc-caciop}, we can show that, for fixed  $M>1$ and $n\geq 2$ is large so that $1<M<\frac{1}{2r_n}$,  
\begin{align*}
&(1-s)[w_n]_{ H^{s}\left( B_{M}\right) }^2\leq   C(M).
\end{align*}
Consequently, up to a subsequence, there exists $w\in H^{ s}_{loc}(\R^N) $ such that, for all $\s<s$,
\be \label{eq:pm1}
w_n\to w\qquad \textrm{  in $H^\s_{loc}(\R^N)$}
\ee
and, by \eqref{eq:groht-w-n-abs-HB-Dir}, 
\be   \label{eq:pm2}
 \int_{\R^N}\frac{| w_n(x)-w(x)|}{1+|x|^{N+2s}}\, dx  \to 0    \qquad \textrm{ as $n\to \infty$.}
\ee
%
Passing to the limit in \eqref{eq:w-n-nonzero-H}, we obtain
\be \label{eq:w-n-nonzero-HH}
 \|w\|_{L^2(B_{1})} \geq \frac{1}{2}  \qquad\textrm{and} \qquad \int_{B_1}w(x) (x_N)_+^{2 s-1}\,dx =0 .
\ee
Next by \eqref{eq:frac-half-flat-n-scal-H-Dir}, \eqref{eq:Qn1},  \eqref{eq:Qn2}  and \eqref{eq:estim-ov-g-n-Dir}, we have  for all  $\phi\in C^\infty_c(B_M^+)$, with $M$  as above, 
\begin{align} 
&\frac{1}{2} \left|\int_{\R^{2N}}(w_n(x)-w_n(y))(\phi(x)-\phi(y)) K_n(x,y)\, dxdy \right|  \leq \frac{C}{  \Theta_n(r_n) } \|\phi\|_{C^{0,1}(\R^N)}. \label{eq:close-LiouvillH}
\end{align}
Now  in view of \eqref{eq:K-prime-small-Dir}, \eqref{eq:pm1}, \eqref{eq:pm2} and   \eqref{eq:close-LiouvillH},  we can  let $n\to \infty$ in \eqref{eq:close-LiouvillH}, to obtain $\Ds_{\R^{N}_+}w=0$ on $\R^N_+$ and $w=0$ on $\R^N\setminus \R^N_+$. Passing to the limit in \eqref{eq:groht-w-n-abs-HB-Dir} yields  $\|w\|_{L^2(B_R)}\leq C R^{N/2+\g}$ for all $R\geq1$. By Lemma  \ref{lem:Liouville-Dirichlet}, and since $\g<2s$, we find that $w(x)=b(x_N)^{2s-1}_+$, for some $b\in \R$. This is in contradiction with  \eqref{eq:w-n-nonzero-HH}.  

\end{proof}

\begin{corollary}\label{eq:cor-reg-BR1}
Let  $s\in (1/2,1),$   $2s-N/p>2s-1$ and   $\b\in (0,1)$.     Let $\psi\in C^{1,\b}(\R^{N-1})$ be such that 
$\| \n \psi\|_{C^{0,\b}(\R^{N-1})}<\frac{1}{4}$ and $ \psi(0)=|\psi(0)|=0$.
Let 
$ g\in L^p(\R^N)$,  $ v\in \calH^s_0(\R^N_+) $  satisfying 
$$
\calD_{K^{\psi}_s}(v,\vp)=\int_{B_2^+}g(x)\vp(x)\, dx \qquad\textrm{ for all $\vp \in C^\infty_c(B_2^+) $,}
$$
with
$$
\| g \|_{L^p(\R^N) } + \|v \|_{ L^2(\R^N)  }\leq 1
$$
Then there exist positive constant $C_0$ depending only on $N,s,p$ and $\b$ and a constant  $Q\in \R$ satisfying  
\be\label{eq:bnd-Q} 
| Q| \leq  C_0  
\ee
and 
$$
\sup_{r>0} r^{-\min (2s-N/p,\b+2s-1)-N/2} \| v-Q(x_N)^{2s-1}_+\|_{L^2(B_r)}\leq C_0  . 
$$
\end{corollary}
\begin{proof}
Since $\min (2s-N/p,\b+2s-1)-\e>2s-1$, in view of \eqref{eq:EEEt},   we can use  similar arguments as in \cite{Fall-reg-1,RS16a},  to obtain  a constant  $Q\in \R$ satisfying  
$$
| Q|+  \sup_{r>0} r^{-\min (2s-N/p,\b+2s-1)-\e-N/2}\| v-Q(x_N)^{2s-1}_+\|_{L^2(B_r)}\leq C_0  . 
$$
This implies, recalling \eqref{eq:defQ-w-z-r}, that $q_{v}(r)\leq C$. We can thus  repeat  the proof of  Proposition \ref{prop:bound-Kato-abstract-HB} using $\min(2s-N/p,\b+2s-1) $ in the place of $ \g=\min (2s-N/p,\b+2s-1)-\e$ to get the desired estimate. We recall that the correction $\e$ was considered in the proof of Proposition \ref{prop:bound-Kato-abstract-HB}  in order to have a control on  the estimate \eqref{eq:esti-d-epsil}.
\end{proof}
\begin{corollary}\label{cor:u-odr-ds}
Under the hypothesis of Corollary \ref{eq:cor-reg-BR1}, there exists a constant $C$ depending only on $N,s,p$ and $\b $ with the following properties. %
\begin{enumerate}
\item[(i)] For all ${\x\in B_{1}^+\cap \R  e_N} $, 
$$
 |v(\x)/\x_N^{2s-1}|+   [ v]_{C^{2s-1}({B_{\x_N/2}(\x)})} + [ v/x_N^{2s-1}]_{C^{\min(1 -N/p,\b)}({B_{\x_N/2}(\x)})}\leq  C .
$$
\item[(ii)] If $2s-\frac{N}{p}>1$ and   $  \b>2s-1$ then  for all ${\x\in B_{1/2}^+\cap \R  e_N} $, 
\be \label{eq:ddd1}
  | \x_N^{2-2s} \n v(\x)| +      [x_N^{2-2s} \n v]_{C^{2s-\frac{N}{p}-1}({B_{\x_N/2}(\x)})}\leq  C .
\ee
Moreover $[x_N^{2-2s}\n v](0,\cdot)$ and    $[v/x_N^{2s-1}](0,\cdot)$ are continuous at $x_N=0$ and satisfy
\be  \label{eq:ddd2}
[x_N^{2-2s}\n v](0)= (2s-1) [v/x_N^{2s-1}](0) e_N.
\ee
\end{enumerate}  

\end{corollary}
\begin{proof}
Let $\x:=(0,2\rho ) \in B_{1}^+$ with $\rho<1/2$.  We define $ u_\x(y)=v(\x+\rho y) -Q((\x+ \rho y)\cdot e_N)_+^{2s-1}= v(\x+\rho y) -Q(2\rho+\rho y_N)_+^{2s-1}$ and $g_\x(y)= \rho^{2s}g(\x+\rho y)$.   We have 
\be \label{eq:Dir-near-final-J}
\calD_{K_{\x,s}} (u_\x,\vp )+\rho^{-\b} \calD_{K_{\x,s}} (U_\x,\vp )= \int_{\R^N}  g_\x( x) \vp(x)\, dx  \qquad\textrm{  for all $\vp \in C^\infty_c(B_1)$,}
\ee
where $K_{\x,s}(x,y)=1_{\Sig}(x)1_{\Sig}(y)\frac{c_{N,s}}{|\Psi_\x(x)-\Psi_\x(y)|^{N+2s}}$, $\Sig:=\{(y',y_N)\,:\, y_N>-2\}$,   $\Psi_\x( y)=\frac{1}{\rho}\Psi(\rho y+\x)$ and $$ U_\x(y):=\rho^{2s-1+\b} (y_N+2)^{2s-1}_+.$$ By a change of variable, we get $\Ds_{\Sig} U_\x=0$ on $\Sig$. We then define 
$$
K_{\x,s}'(x,y)=\rho^{-\b} \left(K_{\x,s}(x,y)-1_{\Sig}(x)1_{\Sig}(y)\frac{c_{N,s}}{|x-y|^{N+2s}} \right).
$$
We thus get from \eqref{eq:Dir-near-final-J},  
\be \label{eq:Dir-near-final-JJ}
\calD_{K_{\x,s}} (u_x,\vp )+ \calD_{K_{\x,s}'} (U_\x,\vp )= \int_{\R^N}   g_\x( x) \vp(x)\, dx  \qquad\textrm{  for all $\vp \in C^\infty_c(B_1)$.}
\ee
Clearly by \eqref{eq:remainder-Kernel-flat}, $K_{\x,s}'$ satisfies 
$$
||x-y|^{N+2s} K_{\x,s}'(x,y)|\leq C(|x|^\b+|y|^\b).  
$$
We let $\chi\in C^\infty_c(B_{1/2})$ with $\chi=1$ on  $B_{1/4}$ and put $\ti U_\x=\chi U_\x\in C^2_c(\R^N)$. Then thanks \eqref{eq:Dir-near-final-JJ}, we have the following equation
\be \label{eq:Dir-near-final}
\calD_{K_{\x,s}} (u_x,\vp )+ \calD_{K_{\x,s}'} (\ti U_\x,\vp )=\int_{\R^N} G_\x( x) \vp(x)\, dx  \qquad\textrm{  for all $\vp \in C^\infty_c(B_{1/8})$,}
\ee
where, for all $x\in B_{1/8}$,  
$$
G_\x(x)=g_\x(x)+1_{B_{1/8}}(x) \int_{|y|\geq 1/4}(1-\chi(y)) U_\x(y)K_{\x,s}'(x,y)\, dy.
$$
  By \eqref{eq:Dir-near-final} and  interior regularity (see e.g. \cite[Corollary 3.5]{Fall-reg-2}) and using \eqref{eq:bnd-Q},    we have  that 
\begin{align*}
\|u_\x\|_{C^{2s-1}(B_{1/9})} \leq C&\left( \|u_\x\|_{L^2(B_1)}+ \int_{|y|\geq1/4}|u_\x(y)||y|^{-N-2s}\, dy  + \|G_\x\|_{L^p(\R^N)}   \right)\\
&\leq C \left( \|u_\x\|_{L^2(B_1)}+ \int_{|y|\geq1/4}|u_\x(y)||y|^{-N-2s}\, dy + \rho^{2s-N/p}+\rho^{\b+2s-1} \right).
\end{align*}
We put $\g=\min(2s-N/p,\b+2s-1)$.  In view of Corollary \ref{eq:cor-reg-BR1} we have 
\be\label{eq:u-xi-tail-est}
 \|u_\x\|_{L^2(B_1)}\leq C \rho^{N/2+\g }.
\ee
 Moreover,   using similar arguments as in the proof of Corollary \ref{cor:C-gam-reg-Dir}, we can estimates the terms 
 \be \label{eq:u-xi-tail-est-T}
 \int_{|y|\geq1/4}|u_\x(y)||y|^{-N-2s}\, dy \leq C \rho^\g.
 \ee
  We then obtain
\begin{align}
\|u_\x\|_{C^{2s-1}(B_{1/9})}  
& \leq C\left( \rho^{\g} +\rho^{2s-N/p} +  \rho^{2s+\b-1}   \right)\leq C \rho^{\g}. \label{eq:est-u-x}
\end{align}
In particular, since $\g-(2s-1)\leq \b $, then   provided $\b\leq 2s-1$,   we have
\begin{align*}
 \|u_\x\|_{C^{\g-(2s-1)}(B_{1/9})}  \leq  C \rho^{\g}.
\end{align*}
 Scaling back, we get
\be  \label{eq:v-Qx-prime}
\|  v  -Q  y_N^{2s-1} \|_{L^\infty(B_{\rho/9}(\x))}\leq C  \rho^{\g}, \qquad [ v -Q y_N^{2s-1}   ]_{C^{\g-(2s-1)  }(B_{\rho/9}(\x))}\leq C\rho^{2s-1}.
\ee
The above three estimates,  \eqref{eq:est-u-x} and \eqref{eq:bnd-Q} imply that 
\be \label{eq:lllk}
  [ v]_{C^{2s-1}(B_{\rho/9}(\x))}+|v(\x)/\x_N^{2s-1}| + [ v/y_N^{2s-1}]_{C^{\g-(2s-1)}(B_{\rho/9}(\x))}\leq C,
\ee
recalling that $\x_N=\rho/2$ and $[ 1/y_N^{2s-1}]_{C^{\g-(2s-1)}(B_{\rho/9}(\x))}\leq C\rho^{-\g}$.  We thus get  $(i)$ for $\b\leq 2s-1$.\\
We now prove $(ii)$ and $(i)$ in the case $\b>2s-1$.    We define the new kernel 
$$
\cK_\x(x,y)=c_{N,s}\rho^{-\b}\left( \frac{1}{|\Psi_\x(x)-\Psi_\x(y)|^{N+2s}}-  \frac{1}{|x-y|^{N+2s}} \right)
$$
and rewrite \eqref{eq:Dir-near-final} as
\be \label{eq:Dir-near-final-Gr}
\calD_{K_{\x,s}} (u_\x,\vp ) = \int_{\R^N}\left(  g_\x( x) +F_\x^1(x) -F_\x^2(x)\right)\vp(x)\, dx  \qquad\textrm{  for all $\vp \in C^\infty_c(B_{1/8})$,}
\ee
where  
$$
F_\x^1(x)=  \int_{\R^N\setminus \Sig}\left( \ti U_\x(x)-\ti U_\x(y) \right) \cK_\x(x,y)dy, \quad
F_\x^2(x)=p.v. \int_{\R^N}\left( \ti U_\x(x)-\ti U_\x(y) \right) \cK_\x(x,y)dy.
$$
Since $ |x-y|\geq 1/2$ for $x\in B_{1/8}$ and $y\in \R^N\setminus \Sig$, using \eqref{eq:remainder-Kernel-flat},  we easily see that  
\be\label{eq:F1e}
|F_\x^1(x)|\leq C(N) \| \ti U_\x \|_{L^\infty(\R^N)}\leq C(N) \rho^{\b+2s-1}.
\ee
On the other hand, using polar coordinates,  w can write
\begin{align*}
F_\x^2(x)&=\frac{1}{2}\int_{S^{N-1}}\int_{0}^\infty \frac{2 \ti U_\x(x)- \ti U_\x(x+r\th)- \ti U_\x(x-r\th)}{r^{1+2s}}(\cA_\x(x,r,\th)+\cA_\x(x,-r,\th))\, drd\th\\
&+\frac{1}{2}\int_{S^{N-1}}\int_{0}^\infty \frac{ \ti U_\x(x+r\th)- \ti U_\x(x-r\th)}{r^{1+2s}}(\cA_\x(x,r,\th)-\cA_\x(x,-r,\th))\, drd\th,
\end{align*}
where
$$
\cA_\x(x,r,\th)= r^{N+2s}\cK_\x(x,x+r\th)=c_{N,s}\rho^{-\b}\left( \frac{1}{|\int_0^1 D\Psi(\rho x+t r\rho \th )\th|^{N+2s}} -1\right).
$$
Recall that  $\Psi( y+\x )=y+ (0, y_N+2\rho+\psi(y')   )$. As a consequence,  for $x\in B_{1/8}$, $r>0$ and $\th \in \de B_1$, 
$$
|\cA_\x(x,r,\th)|\leq C (|x|^\b+r^\b), \qquad |\cA_\x(x,r,\th)-\cA_\x(x,r,-\th)|\leq C r^\b.
$$
Since $\ti U_\x\in C^2(\R^N)$ and $1>\b>2s-1>0$,   we obtain
$$
\|F_\x^2\|_{L^\infty(B_{1/8})} \leq C(N)\|\ti U_\x \|_{C^2(\R^N)}\leq C(N) \rho^{\b+2s-1}.
$$
From this, \eqref{eq:F1e}, \eqref{eq:u-xi-tail-est}, \eqref{eq:u-xi-tail-est-T} and the fact that $\|g_\x\|_{L^p(\R^N)}\leq \rho^{2s-N/p}$,  we can apply  \cite[Theorem 4.1]{Fall-reg-1} to  \eqref{eq:Dir-near-final-Gr} and get
$$
\|u_\x\|_{C^{\a}(B_{1/9})}  \leq  C \rho^{\g},
$$
for all $\a\in (0, \min(2s-N/p,1))$.
In particular,  
\begin{align}
 \|u_\x\|_{C^{\g-(2s-1)}(B_{1/9})} + \|u_\x\|_{C^{2s-1}(B_{1/9})} \leq  C \rho^{\g}, 
\end{align}
Then as above, we obtain 
\be \label{eq:ggggh}
  [ v]_{C^{2s-1}(B_{\rho/9}(\x))}+|v(\x)/\x_N^{2s-1}| + [ v/y_N^{2s-1}]_{C^{\g-(2s-1)}(B_{\rho/9}(\x))}\leq C.
\ee
  We thus get  $(i)$ for $\b> 2s-1$.\\
We now finalize the proof of $ (ii)$. Here, we recall that  $2s-N/p>1$ and $\b>2s-1$. Then  \cite[Theorem 1.3]{Fall-reg-2} applied to the equation \eqref{eq:Dir-near-final-Gr} yields
\begin{align}\label{eq:dddf}
\|u_\x\|_{C^{1,\s}(B_{1/9})}  \leq  C \rho^{\g},
\end{align}
where $\s=2s-\frac{N}{p}-1$. Hence by scaling back, we obtain
\be \label{eq:fffr}
\|\n\left(  v-Q y_N^{2s-1} \right)\|_{L^\infty(B_{\rho/9 }(\x) )}\leq C\rho^{\g-1}, \qquad [\n\left(  v-Q y_N^{2s-1} \right)]_{C^{\s}(B_{\rho/9 }(\x) )}\leq C \rho^{\g-1-\s}.
\ee
Let $w(y)=y_N^{2-2s} \n\left(  v-Qy_N^{2s-1} \right)$. Then, from the above estimates,  for $y_1,y_2\in B_{\rho /9}(\x) $, we have 
\begin{align*}
&|w(y_1)-w(y_2)|\leq C  |y_1-y_2| \rho^{1-2s} \|\n\left(  v-Qy_N^{2s-1} \right)\|_{L^\infty(B_{\rho/9 }(\x))}\\
&+ C \rho^{2-2s}  |y_1-y_2|^{\s} [\n\left(  v-Q y_N^{2s-1} \right)]_{C^{\s}(B_{\rho/9 }(\x))}\\
&\leq C |y_1-y_2| \rho^{1-2s+\g-1}+ C \rho^{2-2s+ \g-1-\s}  |y_1-y_2|^{\s}\\
&= C |y_1-y_2|^{\s+1-\s} \rho^{1-2s+\g-1}+ C \rho^{2-2s + \g-1-\s}  |y_1-y_2|^{\s} \leq  C \rho^{\min(2-2s, N/p)}   |y_1-y_2|^\s.
\end{align*}
We then get, recalling that $\s=2s-\frac{N}{p}-1$, 
$$
\left[  w \right]_{C^{2s-\frac{N}{p}-1}(B_{\rho/9}(\x) )} \leq C \rho^{\min(2-2s, N/p)}\leq C.
$$
Moreover by \eqref{eq:fffr}, we have 
\be\label{eq:oppp}
|w(\x)| \leq C\x_N^{\g+1-2s}.
\ee
Note also  that  $y_N^{2-2s} \n v(y)=w(y) + (2s-1)Q  e_N$,   we thus deduce \eqref{eq:ddd1}. \\
Next, from the two estimates above, we can apply Lemma \ref{lem:Prop-of-RS} to deduce that
$$
\| w\|_{C^{2s-\frac{N}{p}-1}({ {B_{1/9}^+} \cap \R e_N)}}\leq C.
$$
Therefore passing to the limit in \eqref{eq:oppp}, we see that $w(0)=0$, so that $\lim_{\x\to 0} \x_N^{2-2s} \n  v(\x)= (2s-1)Q e_N$.
In addition \eqref{eq:lllk}, \eqref{eq:ggggh} and Lemma \ref{lem:Prop-of-RS}   imply that  $ v/y_N^{2s-1}(0,\cdot)\in C^{\min(1-N/p,\b)}\left(0, 1\right)$.
Now,  letting $\x\to 0$ in \eqref{eq:v-Qx-prime}, we find that $[v/y_N^{2s-1}](0)=Q$ and we get  \eqref{eq:ddd2}. 
\end{proof}

\section{Proof of the main results}\label{s:poof-MR}
We start with the following result which allows us to study an equivalent problem in the half-space.
Let $\O$ be an open subset of $\R^N$ of class $C^{1,\b}$, with $0\in \de\O$ and $\b\geq 0$.  Up to a scaling and a translation, we may assume that $B_1\cap \de\O$ is contained in  the graph of a $C^{1,\b}$ function $\g$.  For $q\in \de\O $, let $\nu(q)$ denote the unit interior normal vector of $\O$ at $q$ and denote $T_q\de\O=\nu(q)^\perp$ the tangent plane of $\de\O$ at $q$.  By the local inversion theorem,  there exist $r_0>0$ depending only on $N,\b$ and the $C^{1,\b}$ norm of $\g$ such that for any $q\in B_{r_0}\cap \de\O$,  there exist an orthonormal basis $(E_1(q)\dots,E_{N-1}(q))$  of $T_q\de\O$ and   $\ti \phi\in C^{1,\b}(  B_{r_0}' )$ such that    
$$
x'\mapsto q+\sum_{i=1}^{N-1} x_i E_i(q)+\ti \phi(x')\nu(q):  B_{r_0}' \to \de\O
$$
with $\ti \phi(0)=|\n\ti \phi(0)|=0$.
 We let $\eta\in C^\infty_c(B_{r_0}')$ with $\eta\equiv 1$ on $B_{r_0/2}'$ and put $\phi=\eta \ti\phi$.
For all $q\in B_{r_0}\cap \de\O$, we define
$$
\Upsilon:\R^N\to \R^N, \qquad \Upsilon(x',x_N)=q+ \sum_{i=1}^{N-1} x_i E_i(q)+\left(x_N+ \phi(x') \right)\nu(q).
$$
Then,  see e.g. \cite{Fall-reg-1}, $\Upsilon$  is a global $C^{1,\b}$ diffeomorphism  with $\textrm{Det}D\Upsilon(x)=1$ for all $x\in \R^N$.  We define $Q_\d:=B_\d'\times (-\d,\d)$ and  $\calQ_\d:=\Upsilon(Q_\d )$. We have that $\Upsilon$   parameterizes   $\O\cap \calQ_{ r_0/2} $ over $Q_{r_0/2}\cap \R^N_+$.  It is easy to see that
$$\|D \Upsilon-Id\|_{L^\infty(\R^N)}=\|\n \phi\|_{L^\infty(\R^{N-1})}= \|\n \phi\|_{L^\infty(B_{r_0}')}\to 0 \qquad\textrm{ as $r_0\to 0$. }$$
Since, recalling \eqref{eq:defPsi},  for all $y'\in B_{r_0/2}'$ and $x\in Q_{r_0/2}$,
$$
|\Upsilon(x',x_N)-\Upsilon(y',0)|^2=|(x',x_N)-(y',0)|^2+\left( O(x_N|x'-y'|)+O(|x'-y'|^2) \right)\|\n \phi\|_{L^\infty (B_{r_0}')},
$$
decreasing $r_0$ if necessary,    by Young's inequality, we have that 
$$
 \frac{x_N}{2}\leq  \textrm{dist}(\Upsilon(x',x_N),\de\O)\leq   2 x_N \quad\textrm{ for all $x\in Q_{r_0/4}\cap \R^N_+$.}
$$
We state the following result.
\begin{lemma}\label{lem:flat-rpoblem}
Let $\O$, $r_0$,  $\phi$ and $\Upsilon$ be as above.  
For $s\in [s_0,1)$, with $s_0\in (0,1)$,  let $u\in H^s_{loc}(\ov \O)\cap L^2(\O)$ be a solution to \eqref{eq:first-eq1} and for $s\in [s_0,1)$,  with $s_0\in (1/2,1)$, we let $v\in H^s_0(\O)$ be a solution to \eqref{eq:first-eq2}, with $f\in L^p(\R^N)$ and $p>\frac{N}{2s_0}$. Then there exist $\hat{u}\in H^s(\R^N_+) $,    $\hat{v}\in H^s_0(\R^N_+)$, $\hat{f}\in L^p(\R^N)$  and $\psi\in C^{1,\b}(\R^{N-1})$  such that 
$$
\calD_{K^{\psi}_s}(\hat{u}, \vp)=\int_{B_2}  \hat{f}(x)\vp(x)\, dx \qquad\textrm{ for all $\vp \in C^1_c(B_2)$}
$$
and
$$
\calD_{K^{\psi}_s}(\hat{v}, \vp)=\int_{B_2}  \hat{f}(x)\vp(x)\, dx \qquad\textrm{ for all $\vp \in C^1_c(B_2^+)$,}
$$
where $\psi(x)=\frac{1}{r}\phi(rx)$ and $r=\frac{r_0}{40}$. Moreover, $\hat{u}=u\circ\Upsilon(r x)$,  $\hat{v}=v\circ\Upsilon(r x)$,  
$$
\|\hat{u}\|_{L^2(\R^N_+)}\leq C   \|u\|_{L^2(\O)}, \qquad \|\hat{v}\|_{L^2(\R^N)}\leq  C   \|v\|_{L^2(\O)}
$$
and $
\|\hat{f}\|_{L^p(\R^N)}\leq C  \left(\|u\|_{L^2(\O)}+\|{f}\|_{L^p(\R^N)}   \right).
$
Here, $C= C (N,r_0, s_0,p) $.
\end{lemma}
\begin{proof}
We will prove the result in the Neumann case. The proof of the Dirichlet case  is similar.
 Let $\nu(x,y)=\frac{c_{N,s}}{2|x-y|^{N+2s}}$.  Put $r=\frac{r_0}{40}$ and let $\ti u= u\chi_r$, with $\chi_r\in C^\infty_c(\calQ_{10 r})$ such that $ \chi_r\equiv 1$ on $ \calQ_{9 r}$. We then get 
$$
\Ds_{\ov \O } \ti u=\ti f \qquad\textrm{ in $\calQ_{9 r }$}
$$
for some function $\ti f\in L^p(\R^N)$ and satisfying
\be\label{eq:esima-1}
\|\ti f\|_{L^p(\R^N)}\leq C(N,s_0,r)\left(\|u\|_{L^2(\O)}+ \|f\|_{L^p(\R^N)} \right). 
\ee
Next,   for $\vp\in C^\infty_c(\calQ_{7r})$,  
\begin{align*}
&\int_{\O}\ti f(x)\vp(x)\, dx= \int_{\O\times\O }( \ti u(x)-\ti u(y))(\vp(x)-\vp(y))\nu(x,y) \, dxdy\\
&= \int_{\O\cap \calQ_{8r}\times \O\cap \calQ_{8r} }( \ti u(x)-\ti u(y))(\vp(x)-\vp(y))\nu(x,y) \, dxdy\\
&+ 2 \int_{\O\cap \calQ_{8r}}\int_{ \O\setminus \calQ_{8r} }( \ti u(x)-\ti u(y))(\vp(x)-\vp(y)) \nu(x,y)\, dxdy\\
&= \int_{\O\cap \calQ_{8r}\times \O\cap \calQ_{8r} }( \ti u(x)-\ti u(y))(\vp(x)-\vp(y)) \nu(x,y)\, dxdy\\
&+ 2 \int_{\O\cap \calQ_{7r}} \ti u(x)\vp(x) \int_{ \O\setminus \calQ_{8r} }   \nu(x,y)\, dxdy+ 2 \int_{\O\cap \calQ_{7r}}\vp(x) \int_{\O\setminus \calQ_{8r} } \ti u(y)\nu(x,y)  \, dxdy.
\end{align*}
Letting $$\ti V(x)=21_{\calQ_{7r}}(x) \int_{ \O\setminus \calQ_{8r} }   \nu(x,y)\, dy, \quad F(x)=\ti f(x)-21_{Q_{7r}}(x)\int_{\O\setminus \calQ_{8r} }\ti u(y)\nu(x,y)  \, dy,$$ we deduce that 
\be\label{eq:eqa-calB}
\int_{\O\cap \calQ_{8r}\times \O\cap \calQ_{8r} }( \ti u(x)-\ti u(y))(\vp(x)-\vp(y))\nu(x,y) \, dxdy+ \int_{\O}\ti  V(x) \ti u(x)\vp(x)\, dx= \int_{\O}F(x)\vp(x) 
\ee
and we note that, for some $C=C(N,s_0,r)$ 
\be\label{eq:bnd-rhs}
\min_{\calB_{7r}} \ti V>0, \qquad \|\ti V\|_{L^\infty(\R^N)}\leq C, \qquad \|F\|_{L^p(\R^N)}\leq C ( \|u\|_{L^2(\O)} + \|f\|_{L^p(\R^N)} ). 
\ee
Let $w=\ti  u\circ\Upsilon\in H^s(\R^N_+) $, $\ov  V=\ti V\circ \Upsilon$  and $\ov F= F\circ \Upsilon\in L^p(\R^N)$. Then by a change of variable and using \eqref{eq:eqa-calB},  we get, for $\vp\in C^\infty_c(B_{4r})$,  
\begin{align*}
&\calD_{K^{\phi}_s}(w,\vp):=\frac{1}{2}\int_{\R^N\times\R^N}( w(x)-w(y))(\vp(x)-\vp(y))K^{\phi}_s(x,y)\, dxdy\\
&=\frac{1}{2}  \int_{ B_{6r}^+\times  B_{6r}^+}( w(x)-w(y))(\vp(x)-\vp(y))K^{\phi}_s(x,y)\, dxdy\\
&+   \int_{ B_{6r}^+}\int_{ \R^N\setminus B_{6r}^+}( w(x)-w(y))(\vp(x)-\vp(y))K^{\phi}_s(x,y)\, dxdy\\
&=\frac{1}{2}\int_{\R^N}\left(-\ov V(x) w(x)+\ov F(x)\right)\vp(x)\, dx+     \int_{ B_{4r}^+} w(x)\vp(x)\int_{ \R^N\setminus B_{6r}^+} K^{\phi}_s(x,y)\, dxdy\\
&+    \int_{ B_{4r}^+}  \vp(x)\int_{ \R^N\setminus B_{6r}^+} w(y)K^{\phi}_s(x,y)\, dxdy .
\end{align*}
We then  get, for all $\vp\in C^\infty_c(B_{4r})$, 
$$
\calD_{K^{\phi}_s}(w,\vp)+\int_{\R^N} H(x)u(x)\vp(x)\, dx=\int_{\R^N} G(x)\vp(x)\, dx,
$$
where $$ H(x):=\ov  V(x)+ 1_{ B_{5r}}(x) \int_{ \R^N\setminus B_{6r}^+} K^{\phi}_s(x,y)\, dy $$ and $$G(x):=1_{B_{5r}}(x) \int_{ \R^N\setminus B_{6r}^+} w(y)K^{\phi}_s(x,y)\, dy+\ov F(x).$$
We  now   scale the variables by putting $\hat{u}(x)= w( rx)$, $g(x)=1_{B_{4}}(x)G(r x)$,  $V(x)=1_{B_{4}}(x)H(r x)$  and $\psi(x)=\frac{1}{r}\phi(r x)$. We then have that
\be \label{eq:norm-v} 
\hat{u}\in H^s(\R^N_+), \qquad \|\hat{u}\|_{L^2(\R^N_+)} \leq C    \|u\|_{L^2(\O)} ,
\ee
\be \label{eq:-norm-g}
\|g\|_{L^p(\R^N)}\leq C(N,s_0,r,p) \left(  \|f\|_{L^p(\R^N)}+ \|u\|_{L^2(\O)} \right),
\ee
\be\label{eq:min-of-V}
 \min_{ B_2} V>0, \qquad \|V\|_{L^\infty(B_2)}\leq C (N,s_0,r)
\ee
and for every $\vp\in C^\infty_c(B_2)$,
\be \label{eq:almost-final-v}
\calD_{K^{\psi}_s}(\hat{u},\vp)+\int_{\R^N} V(x)\hat{u}(x)\vp(x)\, dx=\int_{\R^N} g(x)\vp(x)\, dx. 
\ee
Now by the Kato inequality and \eqref{eq:min-of-V}, for all  $\vp\in C^\infty_c(B_2)$ with $\vp\geq 0$,   we have that
\begin{align*}
\calD_{K^{\psi}_s}(|\hat{u}|,\vp)\leq \calD_{K^{\psi}_s}(|\hat{u}|,\vp)+\int_{\R^N} V(x)|\hat{u}(x)|\vp(x)\, dx \leq \int_{\R^N} |g(x)|\vp(x)\, dx. 
\end{align*} 
%
By a direct minimization argument, there exists a unique function $ \eta\in H^s(\R^N_+)$ such that $\eta=|\hat{u}|$ on $\R^N_+\setminus B_2$  and for all $\vp\in C^\infty_c(B_2)$, 
  $$
\calD_{K^{\psi}_s}(\eta,\vp)=\int_{\R^N} |g(x)|\vp(x)\, dx. 
$$
 By construction, we have that 
\be \label{eq:seq-increase}
 |  \hat{u}|\leq  \eta   \qquad\textrm{ in $B_2$.}
\ee
Applying Corollary \ref{eq:cor-reg}, we deduce  that 
\be \label{eq:eta-L-infty}
\|\eta\|_{L^\infty(B_1^+)}\leq C(N,s_0,p,r)   \left( \|\eta\|_{L^2(\R^N_+)}+ \|g\|_{L^p(\R^N)}\right)\leq C (N,s_0,p) \left( \|\hat{u}\|_{L^2(\R^N_+)}+ \|g\|_{L^p(\R^N)}\right).
\ee
From this and  \eqref{eq:seq-increase}, we finally get 
\be\label{eq:est-vL-infty}
\|\hat{u}\|_{L^\infty(B_1^+)}\leq  C(N,s_0,r,p) \left( \|\hat{u}\|_{L^2(\R^N_+)}+ \|g\|_{L^p(\R^N)}\right) \leq  C   \left( \|u\|_{L^2(\O)}+ \|g\|_{L^p(\R^N)}\right). 
\ee
We can now rewrite  \eqref{eq:almost-final-v} in the form,  
$$
\int_{\R^N\times\R^N}( \hat{u}(x)-\hat{u}(y))(\vp(x)-\vp(y))K^{\psi}_s(x,y)\, dxdy =\int_{\R^N} \hat{f} (x)\vp(x)\, dx,
$$
with $\hat{f}=1_{\R^N_+}\left( g- 1_{B_1} V\hat{u}\right)$. We note that by \eqref{eq:est-vL-infty},  \eqref{eq:-norm-g}  and \eqref{eq:norm-v},
$$
\|\hat{f}\|_{L^p (\R^N)}\leq C(N,s_0,r,p)  \left(  \|f\|_{L^p(\R^N)}+    \|u\|_{L^2(\O)}   \right) .
$$
The proof is thus complete in the Neumann case. Note that for the Dirichlet case, the only change in the above argument is to apply Corollary \ref{cor:C-gam-reg-Dir} in the place of Corollary \ref{eq:cor-reg} in order to get \eqref{eq:eta-L-infty}. 
\end{proof}
\subsection{Proof of the main results (completed)}
We are now in position to complete the proofs of the main results in the first section. For the following,  we assume for simplicity that $\|f\|_{L^p(\O)}+ \|u\|_{L^2(\O)} \leq 1$, $\|f\|_{L^p(\O)}+ \|v\|_{L^2(\O)} \leq 1$  and $\O$ is as in the beginning of Section \ref{s:poof-MR}. For $x\in \R^N$,   we let $\d(x)$ to  be $\textrm{dist}(x,\de\O)$ for $x$ in a tubular neighbourhood of $\de\O$ and   Lipschitz continuous and positive elsewhere.
\begin{proof}[Proof of Theorem \ref{th:amin1} (completed)]
The statement in the theorem follows from Lemma \ref{lem:flat-rpoblem} and Corollary \ref{eq:cor-reg}.
\end{proof}
\begin{proof}[Proof of Theorem \ref{th:amin2} (completed)]
We put $\g:=\min (2s-N/p,1+\b)$. Then in view of  Lemma \ref{lem:flat-rpoblem} and  Corollary  \ref{cor:near-grad-estim-Hold}, we can find two constants  $\ov r,  C>0$, depending only on $N,s,p,\O$ and $\b$ such that  for any $x\in B_{\ov r}\cap \O$, we have that $ |\n u (x)| + [\n u]_{C^{\g-1}(B_{\rho}(x))}\leq C$ and $\n u(q)\cdot \nu(q)=0$, with $\rho=\d(x)/4$ and $|x-q|=\d(x)$. Applying Lemma \ref{lem:Prop-of-RS}, after flattening the boundary, we get the stated results.
\end{proof}

 \begin{proof}[Proof of   Theorem \ref{th:amin1-Dir} (completed)]
The proof follows from Lemma \ref{lem:flat-rpoblem} and Corollary \ref{cor:C-gam-reg-Dir}.
 \end{proof}

 \begin{proof}[Proof of   Theorem \ref{th:amin3} (completed)]
In view of  Lemma \ref{lem:flat-rpoblem} and  Corollary \ref{cor:u-odr-ds}-$(i)$, we obtain  two  constants  $\ov r=\ov r(N,\O,\b)>0$ and $ C=C(N,s,p,\O,\b)>0$  such that  for any $x\in B_{\ov r}\cap \O$, we have that 
 $$
  | v (x)/\d^{2s-1}(x)| + [ v]_{C^{2s-1}(B_{\rho}(x))}+ [ v/\d^{2s-1}]_{C^{\min (1-N/p,\b)}(B_{\rho}(x))}\leq C
  $$
 with $\rho=\d(x)/4$. Applying Lemma \ref{lem:Prop-of-RS}, after flattening the boundary, we get $(i)$.
 Next, we let $ 2s-N/p >1$ and $\b>2s-1$. Then in this case we use  Corollary \ref{cor:u-odr-ds}-$(ii)$ to get
 $$
  | \d^{2s-2}(x) \n v(x)| +   [ \d^{2s-2}\n v]_{C^{2s-\frac{N}{p}-1} (B_{\rho}(x)) } \leq C
  $$
   and $\lim_{\d(x)\to 0}\left(\d(x)^{2-2s}\n v(x) \right)=(2s-1) [v/\d^{2s-1}](q) \nu(q)$, where $|x-q|=\d(x) $. We then apply Lemma \ref{lem:Prop-of-RS}, after flattening the boundary, to get $(ii)$.
 \end{proof}
 
\begin{remark}\label{rem:const-s}
We point out that provided    $2s_0-N/p>1$ and $\b<2s_0-1$,   the constant $C$ in  Theorem \ref{th:amin2} and Theorem \ref{th:amin3}-$(i)$ can be chosen to  remain bounded as $s\to 1$, while for Theorem \ref{th:amin3}-$(ii)$, we need to assume that $\b=1$.  Then the constant $C$ appearing in Theorem \ref{th:amin2} and Theorem \ref{th:amin3} remains bounded as $s\to 1$. This follows from Lemma \ref{lem:nonloc-to-loc} and the  blow up argument that is used in the proof of  Proposition \ref{prop:bound-Kato-abstract-HB1} and Proposition \ref{prop:bound-Kato-abstract-HB}. 
\end{remark} 
 \section{Appendix}\label{s:Append}
 This section is devoted to prove the convergence  of the nonlocal problem to the local problem as $s\to1$.
 \begin{lemma}\label{lem:nonloc-to-loc}
 \begin{enumerate}
 \item[(i)] Let $\O$ be an open subset of $\R^N$. Then for all $u,v\in C^1_c(\R^N)$, we have 
$$
 \lim_{s\to 1} c_{N,s}\int_{\O}\int_{\O}\frac{(u(x)-u(y))(v(x)-v(y))}{|x-y|^{N+2s}}\, dxdy=\g_N\int_{\O}\n u(x)\cdot \n v(x)\, dx .
$$

 \item[(ii)] Let $\psi_n\in C^{1}(\R^{N-1})$ be such that $[\n \psi_n]_{C^1(\R^{N-1})}\to 0 $ and $s_n\in (0,1)$ with $s_n\to 1$.  Let $w_n\in H^{s_n}_{loc}(\ov{\R^N_+})$ with  $w_n \to w\in L^2_{loc}(\ov{\R^N_+})$ and suppose that  there exists $C>0$ such that, for all $n\in\N$, 
$$
(1-s_n)[w_n]_{ H^{s_n}_{loc}(\R^N_+)}+ \int_{\R^N_+}\frac{|w_n(x)|}{1+|x|^{N+2s_n}}\, dx\leq C.
$$
Then, for all $\phi\in C^\infty_c(\R^N)$, we have 
$$
\lim_{n\to \infty}  \calD_{K^{\psi_n}_{s_n}}(w_n,\phi)=\frac{\g_N}{2}\int_{\R^N_+} \n w(x)\cdot \n \phi (x)  dx.
$$
 \end{enumerate}
Here,  $\g_N=\lim_{s\to 1} \frac{c_{N,s} |B_1|}{2(1-s)}$.
\end{lemma}
\begin{proof}
 Let $u,v\in C^\infty_c(\R^N)$ and $R>0$ be  such that   Supp$u$,  Supp$v$ $\subset B_{R}.$ We estimate
\begin{align}\label{eq:firts-e-s-to1}
&c_{N,s}\int_{\O}\int_{\O}\frac{(u(x)-u(y))(v(x)-v(y))}{|x-y|^{N+2s}}\, dxdy\nonumber\\
&=c_{N,s} \int_{\O\cap B_{2R}}\int_{\O}\frac{(u(x)-u(y))(v(x)-v(y))}{|x-y|^{N+2s}}\, dxdy \nonumber\\
&+   c_{N,s}\int_{\O\setminus B_{2R}}\int_{\O}\frac{(u(x)-u(y))(v(x)-v(y))}{|x-y|^{N+2s}}\, dxdy \nonumber\\
& =c_{N,s} \int_{\O\cap B_{2R}}\int_{\O}\frac{(u(x)-u(y))(v(x)-v(y))}{|x-y|^{N+2s}}\, dxdy \nonumber\\
&+ c_{N,s}\int_{\O\cap B_{R}}u(y)v(y) dy\int_{\O\setminus B_{2R}} \frac{1}{|x-y|^{N+2s}}\, dx \nonumber\\
& =c_{N,s} \int_{\O\cap B_{2R}}\int_{\O}\frac{(u(x)-u(y))(v(x)-v(y))}{|x-y|^{N+2s}}\, dxdy+O(c_{N,s}) [u]_{C^1(\R^N)}[v]_{C^1(\R^N)}.
\end{align}
By Taylor expansion,  for all $\e<0$, there exits $r_\e>0$ such that for $y\in B(x, r_\e)$
$$
u(x)-u(y)=\n u(x)\cdot (x-y)+ O(\e |x-y|)[u]_{C^1(\R^N)}.
$$
To alleviate the notations, we assume that $[u]_{C^1(\R^N)}[v]_{C^1(\R^N)}\leq 1$. We then have  
\begin{align*}
&c_{N,s} \int_{\O\cap B_{2R}}dx\int_{\O}\frac{(u(x)-u(y))(v(x)-v(y))}{|x-y|^{N+2s}}\, dy\\
&=c_{N,s} \int_{\O\cap B_{2R}}dx\int_{|x-y|\leq r_\e}\frac{(u(x)-u(y))(v(x)-v(y))}{|x-y|^{N+2s}}\, dy\\
&+c_{N,s} \int_{\O\cap B_{2R}}dx\int_{|x-y|\geq r_\e}\frac{(u(x)-u(y))(v(x)-v(y))}{|x-y|^{N+2s}}\, dy\\
&=c_{N,s} \int_{\O\cap B_{2R}}dx\int_{|x-y|\leq {r_\e}}\frac{\n u(x)\cdot (x-y) \n v(x)\cdot (x-y) }{|x-y|^{N+2s}}\, dy\\
&+c_{N,s} \int_{\O\cap B_{2R}}dx\int_{|x-y|\leq {r_\e}}\frac{ O(\e |x-y|^2) }{|x-y|^{N+2s}}\, dy +c_{N,s} \int_{\O\cap B_{2R}}dx\int_{ |x-y|\geq r_\e}\frac{O(1) }{|x-y|^{N+2s}}\, dy.
\end{align*}
We then obtain
\begin{align*}
&c_{N,s} \int_{\O\cap B_{2R}}\int_{\O}\frac{(u(x)-u(y))(v(x)-v(y))}{|x-y|^{N+2s}}\, dxdy\\
&=c_{N,s} \int_{\O}\int_{S^{N-1}} \n u(x)\cdot\th \n v(x)\cdot\th\, d\th  \int_0^{r_\e} r^{1-2s}dr +O(\e\frac{c_{N,s}}{1-s})  +c_{N,s} O( r_\e^{-2s}) \\
&=\frac{c_{N,s} |B_1| r_\e^{2-2s}}{2(1-s)} \int_{\O } \n u(x)\cdot \n v(x)  dx +\left( O(\e\frac{c_{N,s}}{1-s}) +c_{N,s} O( r_\e^{-2s} ) \right)  .
\end{align*}
Using this in \eqref{eq:firts-e-s-to1}, we conclude that 
\begin{align}\label{eq:Dirrr-final}
c_{N,s}\int_{\O}\int_{\O}&\frac{(u(x)-u(y))(v(x)-v(y))}{|x-y|^{N+2s}}\, dxdy=\frac{c_{N,s} |B_1| r_\e^{2-2s}}{2(1-s)} \int_{\O } \n u(x)\cdot \n v(x)  dx \nonumber\\
&+\left(O(\e\frac{c_{N,s}}{1-s})  +c_{N,s} O( r_\e^{-2s} \right) [u]_{C^1(\R^N)}[v]_{C^1(\R^N)}.
\end{align}
Since $c_{N,s}=O(1-s)$, then letting $s\to 1$ and then $\e\to 0$, we get $(i)$.\\
We now prove $(ii)$. Recall that  $\Psi_n(x',x_N)=(x',x_N+\psi_n(x))$ (see Section \ref{s:Reg-Neum}).
We define  $\O_n:=\Psi_n({\R^N_+})$.
Let $\phi,\vp \in  C^\infty_c(\R^N)$ and define $u_n:=\phi\circ\Psi_n$ and $v_n:=\vp\circ\Psi_n$. By a change of variable, we get
\begin{align*}
\calD_{K^{\psi_n}_{s_n}}(\phi,\vp)=\frac{c_{N,s_n}}{2} \int_{\O_n}\int_{\O_n}\frac{(u_n(x)-u_n(y))(v_n(x)-v_n(y))}{|x-y|^{N+2s_n}}\, dxdy
\end{align*}
By \eqref{eq:Dirrr-final}, we have 
\begin{align*} 
c_{N,s_n}\int_{\O_n}\int_{\O_n}&\frac{(u_n(x)-u_n(y))(v_n(x)-v_n(y))}{|x-y|^{N+2s_n}}\, dxdy=\frac{c_{N,s_n} |B_1| r_\e^{2-2s_n}}{2(1-s_n)} \int_{\O_n } \n u_n(x)\cdot \n v_n(x)  dx \nonumber\\
&+\left(O(\e\frac{c_{N,s_n}}{1-s_n})  +c_{N,s_n} O( r_\e^{-2s_n} \right) [u_n]_{C^1(\R^N)}[v_n]_{C^1(\R^N)}.
\end{align*}
Using now that $[\n \psi_n]_{C^1(\R^{N-1})}\leq \frac{1}{n}$, the fact that $c_{N,s_n}=O(1-s_n)$ and   the dominated convergence theorem, we deduce that 
 \begin{align*}
 2\lim_{n\to \infty}\calD_{K^{\psi_n}_{s_n}}(\phi,\vp)=\g_N\int_{\R^N_+} \n \phi(x)\cdot \n \vp (x)  dx+ O(\e)  [\phi]_{C^1(\R^N)}[\vp]_{C^1(\R^N)}.
 \end{align*}
 Now letting $\e\to 0$, we get, for all $\phi,\vp \in  C^\infty_c(\R^N)$ ,  
 \be \label{eq:ETSa}
 2 \lim_{n\to \infty}\calD_{K^{\psi_n}_{s_n}}(\phi,\vp)= \g_N\int_{\R^N_+} \n \phi(x)\cdot \n \vp (x)  dx.
 \ee
In addition,  for $\phi,\vp \in H^1(\R^N_+)$, we have that 
 $$
 |\calD_{K^{\psi_n}_{s_n}}(\phi,\vp)|\leq C(N)\|\phi\|_{H^1(\R^N_+)} \|\vp\|_{H^1(\R^N_+)}.
 $$
 Therefore  the symmetric bilinear form $\calD_{K^{\psi_n}_{s_n}}$ form-converges  in $H^1(\R^N_+)\times H^1(\R^N_+)$ to a symmetric bilinear form $\calD_{\infty}: H^1(\R^N_+)\times H^1(\R^N_+)\to \R$. Namely,  for all $\phi,\vp\in H^1(\R^N_+)$, 
$$
\lim_{n\to \infty} \calD_{K^{\psi_n}_{s_n}}(\phi,\vp)= \calD_{\infty}(\phi,\vp).
$$ 
By density and \eqref{eq:ETSa}, we have that for all   $\phi,\vp\in H^1(\R^N_+)$,
\be \label{eq:ETSb}
2 \lim_{n\to \infty} \calD_{K^{\psi_n}_{s_n}}(\phi,\vp)=2\calD_{\infty}(\phi,\vp)=\g_N\int_{\R^N_+} \n \phi(x)\cdot \n \vp (x)  dx.
\ee
Fix $M\geq 1$ and consider $\chi_M\in C^\infty_c(B_{2M})$ such that $\chi_M=1$ on $B_M$. Then for $\phi\in C^\infty_c(B_{M}/2)$, we get
\begin{align}\label{eq:ETS0}
& \calD_{K^{\psi_n}_{s_n}}(\chi_M w_n,\phi)= \calD_{K^{\psi_n}_{s_n}}(\chi_M w,\phi)+  \calD_{K^{\psi_n}_{s_n}}(\chi_M(w_n-w),\phi).
\end{align}
Letting $v_n=\chi_M(w_n-w)$, we have 
\begin{align*}
&2\left|   \calD_{K^{\psi_n}_{s_n}}( v_n,\phi) \right|= c_{N,s_n}\left| \int_{\R^N_+\times \R^N_+}\frac{(v_n(x)-v_n(y))(\phi(x)-\phi(y))}{ |x-y|^{N+2s_n}}\, dxdy \right| \nonumber\\
&+   \int_{\R^N\times \R^N}{|(v_n(x)-v_n(y))(\phi(x)-\phi(y))|}\left|K^{\psi_n}_{s_n}(x,y)-K_{s_n}^+(x,y) \right|\, dxdy.
\end{align*}
Using this, the fact that $w_n \to w $ in $H^{\s}_{loc}(\ov{\R^N_+})$ for every $\s\in (0,1)$ and \eqref{eq:remainder-Kernel-flat}, we obtain 
\be \label{eq:ETS2}
\left|   \calD_{K^{\psi_n}_{s_n}}( \chi_M(w_n-w),\phi) \right|=o(1) \qquad\textrm{ as $n\to \infty$}.
\ee
 Now, by a direct computation, for $\phi\in C^\infty_c(B_{M/2})$, we also have that 
\begin{align*}
\calD_{K^{\psi_n}_{s_n}}(\chi_M w_n,\phi)=\calD_{K^{\psi_n}_{s_n}}(w_n,\phi)+c_{N,s_n}\int_{\R^N}G_n(x)\phi(x)\, dx,
\end{align*}
where $G_n(x)=1_{B_{M/2}}\int_{\{|y|\geq M\}\cap \R^N_+} |x-y|^{-N-2s_n}w_n(y)\, dy.$ Since $c_{N,s_n}=O(1-s_n)\to 0$ as $n\to \infty$, by the above identity, \eqref{eq:ETS0},  \eqref{eq:ETS2} and \eqref{eq:ETSb}, we deduce that 
$$
\lim_{n\to \infty}  \calD_{K^{\psi_n}_{s_n}}(w_n,\phi)=\frac{\g_N}{2}\int_{\R^N_+} \n w(x)\cdot \n \phi (x)  dx.
$$
\end{proof}

\end{document}